\documentclass[a4paper]{amsart}

\usepackage[utf8]{inputenc}
\usepackage[T1]{fontenc}
\usepackage{lmodern, enumerate}
\usepackage{amssymb,amsxtra}
\usepackage[all]{xy}
\usepackage{nicefrac,mathtools}
\usepackage{microtype}
\usepackage{amscd}
\usepackage{mathbbol}
\usepackage{xcolor}
\usepackage{verbatim}
\usepackage[pdftitle={Amenability for partial actions and Fell bundles},
 pdfauthor={Fernando Abadie, Alcides Buss, Damian Ferraro},
 pdfsubject={Mathematics}]{hyperref}
\usepackage[lite]{amsrefs}

\newcommand*{\arxiv}[1]{\href{http://www.arxiv.org/abs/#1}{arXiv: #1}}

\numberwithin{equation}{section}
\theoremstyle{plain}
\newtheorem{theorem}[equation]{Theorem}
\newtheorem{lemma}[equation]{Lemma}
\newtheorem{proposition}[equation]{Proposition}
\newtheorem{corollary}[equation]{Corollary}
\theoremstyle{definition}
\newtheorem{definition}[equation]{Definition}

\theoremstyle{remark}
\newtheorem{remark}[equation]{Remark}
\newtheorem{example}[equation]{Example}

\DeclareMathOperator{\cspn}{\overline{span}}
\DeclareMathOperator{\spn}{span}

\DeclareMathOperator{\supp}{\mathrm{supp}}

\newcommand*{\dual}[1]{\widehat{#1}}
\newcommand*{\nb}{\nobreakdash}
\newcommand*{\Star}{\(^*\)\nobreakdash-}

\newcommand{\cA}{\mathcal{A}}
\newcommand{\A}{\mathcal{A}}
\newcommand{\cB}{\mathcal{B}}
\newcommand{\B}{\mathcal{B}}
\newcommand*{\M}{\mathcal{M}}

\newcommand*{\C}{\mathbb{C}}

\newcommand{\bk}{\mathbb{k}}
\newcommand{\bK}{\mathcal{K}}

\newcommand*{\Lb}{\mathcal{L}}%adjointable operators on a Hilbert module
\newcommand*{\cont}{C}% continuous functions
\newcommand*{\contz}{\cont_0}%continuous functions vanishing at infinity
\newcommand*{\contc}{\cont_c}%continuous functions with compact support
\newcommand*{\id}{\textup{id}}%identity
\newcommand*{\Ad}{\textup{Ad}}%conjugation by a unitary

\newcommand{\wot}{\operatorname{wot}}

\newcommand*{\defeq}{\mathrel{\vcentcolon=}}% used for definitions
\newcommand*{\congto}{\xrightarrow\sim}

\newcommand*{\ket}[1]{\lvert#1\rangle}
\newcommand*{\bra}[1]{\langle#1\rvert}

\newcommand*{\braket}[2]{\langle#1\!\mid\!#2\rangle}

\newcommand*{\sbe}{\subseteq} % inclusion
\newcommand*{\F}{\mathbb F}
\newcommand*{\cstar}{\texorpdfstring{$C^*$\nobreakdash-\hspace{0pt}}{$C^*$-}}

\newcommand*{\into}{\hookrightarrow}
\newcommand*{\onto}{\twoheadrightarrow}
\newcommand*{\red}{\mathrm{r}}

\newcommand*{\weaks}{{\operatorname{w}^*}}

\newcommand{\cspan}{\overline{\operatorname{span}}}

\newcommand*{\s}{s} % source map
\newcommand*{\rg}{r}% range map

\newcommand{\alg}{\operatorname{alg}}

\begin{document}

%\normalem

\title[Amenability for partial actions and Fell bundles]{Amenability and approximation properties\\for partial actions and Fell bundles}
\author{Fernando Abadie}
\email{fabadie@cmat.edu.uy}
\address{Centro de Matemática\\
Facultad de Ciencias\\
Universidad de la República\\
Iguá 4225, 11400, Montevideo\\ 
Uruguay.}

\author{Alcides Buss}
\email{alcides.buss@ufsc.br}
\address{Departamento de Matem\'atica\\
 Universidade Federal de Santa Catarina\\
 88.040-900 Florian\'opolis-SC\\
 Brazil}

\author{Dami\'{a}n Ferraro}
\email{dferraro@unorte.edu.uy}
\address{Departamento de Matemática \\
y Estadística del Litoral\\
Universidad de la República \\
Rivera 1350, 50000, Salto\\
Uruguay}

\date{\today}
\thanks{The second named author was supported by Print/Capes, Humboldt and CNPq.}

\begin{abstract}
Building on previous papers by Anantharaman-Delaroche (AD) we introduce and study the notion of AD-amenability for partial actions and Fell bundles over discrete groups. We prove that the cross-sectional $C^*$-algebra of a Fell bundle is nuclear if and only if the underlying unit fibre is nuclear and the Fell bundle is AD-amenable. If a partial action is globalisable, then it is AD-amenable if and only if its globalisation is AD-amenable. Moreover, we prove that AD-amenability is preserved by (weak) equivalence of Fell bundles and, using a very recent idea of Ozawa and Suzuki, we show that AD-amenabity is equivalent to an approximation property introduced by Exel.
\end{abstract}

\maketitle

\tableofcontents

\section{Introduction}

In her seminal paper \cite{Anantharaman-Delaroche:Systemes} Anantharaman-Delaroche introduced a notion of amenability (which we call AD-amenability) for actions of discrete groups on $C^*$-algebras. Her definition is based on previous papers \cites{Anantharaman-Delaroche:ActionI,Anantharaman-Delaroche:ActionII} where she studies amenability for group actions on W*-algebras (i.e. von Neumann algebras). More precisely, an action $\gamma$ of a discrete group $G$ on a W*-algebra $N$ is said to be amenable in the sense of Anantharaman-Delaroche (W*AD-amenable for short) if there exists a $G$-equivariant conditional expectation $P\colon \ell^\infty(G,N)\onto N$ with respect to the diagonal $G$-action $\tilde\gamma$ on $\ell^\infty(G,N)=\ell^\infty(G)\bar\otimes N$ where $G$ acts on $\ell^\infty(G)$ by (left) translations; the map $P$ should be interpreted as a $G$-equivariant mean for the action. An action $\alpha$ of $G$ on a $C^*$-algebra $A$ is then said to be AD-amenable if the induced action $\alpha''$ on the enveloping (bidual) W*-algebra $A''$ is W*AD-amenable.

One of the main results in \cite{Anantharaman-Delaroche:Systemes} (namely Theorem~3.3) shows that equivariant means $P\colon \ell^\infty(G,N)\to N$ can be approximated with respect to the pointwise weak* (i.e. ultraweak) topology by using certain nets of functions from $G$ to $Z(N)$.
One precise form of such approximation (that will be specially important to us) is given by a net of functions of finite support $\{a_i\colon G\to Z(N)\}_{i\in I}$ which is bounded when viewed as a net of the Hilbert $N$-module $\ell^2(G,N)$ and satisfies
\begin{equation}\label{eq:def-W*AD-amenable}
\braket{a_i}{\tilde\gamma_g(a_i)}_2=\sum_{h\in G}a_i(h)^*\gamma_g(a_i(g^{-1}h))\to 1
\end{equation}
with respect to the weak*-topology for all $g\in G$. This condition indeed characterises amenability and shows, among other things, that $\gamma$ is W*AD-amenable if and only if so is its restriction to the centre $Z(N)$. Moreover, W*AD-amenability behaves well with respect to injectivity of W*-algebras and nuclearity of $C^*$-algebras: if $N$ is injective then $\gamma$ is W*AD-amenable if and only if the W*-crossed product $N\bar\rtimes_\gamma G$ is injective. And similarly, if $A$ is a nuclear $C^*$-algebra, then the reduced $C^*$-crossed product $A\rtimes_{\red,\alpha}G$ is nuclear if and only if $\alpha$ is AD-amenable.

Notice that the AD-amenability of an action on a $C^*$-algebra $A$ requires (and is equivalent to) the existence of a net as above with values in $Z(A'')$. While this is a huge commutative algebra in general, finding explicitly such an approximate equivariant mean might be a very difficult task -- if not impossible. Hoping for more concrete realisations of such approximate means one might wonder whether it is not always possible to find a net with values in $Z(A)$ or at least in $ZM(A)$, the centre of the multiplier algebra. This is indeed possible for commutative $A$ (by \cite{Anantharaman-Delaroche:Systemes}*{Theorem~4.9}) and hence more generally for $A$ admitting (nondegenerate) $G$\nb-equivariant \Star{}homomorphism $\contz(X)\to ZM(A)$ for some amenable $G$-space $X$. Unfortunately this is not possible in general: striking recent results by Suzuki in \cite{MR3589332} show that every exact group admits an AD-amenable action on a unital simple nuclear $C^*$-algebra $A$ (and one can even choose such algebra for which the crossed product is in the same class). For such an $A$ we have $Z(A)=ZM(A)=\C\cdot 1$ so that the existence of an approximate mean as above with values in $Z(A)$ forces $G$ to be amenable. On the other hand, dropping the commutativity completely and  asking only for a net $\{a_i\}_{i\in I}\subseteq \ell^2(G,N)$ satisfying~\eqref{eq:def-W*AD-amenable} is also not a good idea because then one adds undesirable actions. For instance the adjoint action $\gamma=\Ad_\lambda$ of the left regular representation on $N=\Lb(\ell^2G)=\bK(\ell^2G)''$ has this weaker property because $\ell^\infty(G)\into \Lb(\ell^2G)$ equivariantly. But this action is AD-amenable only if $G$ is amenable.

Fortunately there is an alternative out of this: we only need to change~\eqref{eq:def-W*AD-amenable} slightly, requiring instead the existence of a bounded net of finitely supported functions $\{a_i\}_{i\in I}\subseteq \ell^2(G,N)$ satisfying
\begin{equation}\label{eq:def-W*AD-amenable1}
\braket{a_i}{b\tilde\gamma_g(a_i)}_2=\sum_{h\in G}a_i(h)^*b\gamma_g(a_i(g^{-1}h))\to b
\end{equation}
for the weak*-topology for all $g\in G$ and $b\in N$. It turns out that this is equivalent to the W*AD-amenability of $\gamma$ (and hence to the existence of a central net satisfying~\eqref{eq:def-W*AD-amenable}).
Moreover, we prove that if $A$ is a weak*-dense $G$-invariant $C^*$-subalgebra of $N$, then the above condition (hence the W*AD-amenability of $N$) is equivalent to the existence of a bounded net of finitely supported functions $\{a_i\}_{i\in I}\subseteq \ell^2(G,N)$ satisfying~\eqref{eq:def-W*AD-amenable1} for all $b\in A$, see Lemma~\ref{lem:characterization of AD amenability with dense subalgebra}.
In particular, an action $\alpha$ on a $C^*$-algebra $A$ is AD-amenable if and only if there exists a bounded net $\{a_i\}_{i\in I}\subseteq \ell^2(G,A)$ of functions with finite supports satisfying
\begin{equation}\label{eq:def-W*AD-amenable2}
\braket{a_i}{b\tilde\alpha_g(a_i)}_2=\sum_{h\in G}a_i(h)^*b\alpha_g(a_i(g^{-1}h))\to b
\end{equation}
with respect to the weak topology on $A$ for all $g\in G$ and $b\in A$. We call this the weak approximation property (WAP).  If we replace the weak by the norm topology, that is, if the net of \eqref{eq:def-W*AD-amenable2} converges in norm for all $g\in G$ and $b\in A$, then we arrive exactly at the approximation property (or just the AP for short) as defined by Exel in \cite{Exel:Amenability} for Fell bundles (over discrete groups); more precisely, this is the approximation property for the semidirect product bundle $\cB_\alpha=A\times G$ associated to $\alpha$. It would be desirable to get from AD-amenability the stronger AP instead of just the WAP because it is usually more concrete and easier to handle in practice.
Indeed, while this paper was under review the recent preprints \cite{buss2020amenability} by Buss, Echterhoff and Willett and \cite{ozawa2020characterizations} by Ozawa an Suzuki answered this in the positive (the former for discrete groups, the latter for general locally compact groups), that is, they show that an action on a $C^*$-algebra is AD-amenable if and only if it has the AP of Exel. The main step of the proof uses a convex approximation argument replacing the original net $\{a_i\}_{i\in I}$ by a new one where the norm convergence is achieved.

The main goal of this paper is to study amenability in the more general context of partial actions or even Fell bundles over (discrete) groups.
A Fell bundle over $G$ consists of a collection $\B=\{B_t\}_{t\in G}$ of
Banach spaces (called fibres) indexed over $G$ endowed with multiplications $B_s\times B_t\to
B_{st}$ and involutions $B_s\to B_{s^{-1}}$ with properties resembling
those of a $C^*$-algebra. In particular the fibre $B_e$ is a $C^*$-\nb-algebra,
where $e\in G$ denotes the unit of $G$. We refer to
\cite{Doran-Fell:Representations} for the basic theory of Fell
bundles, where they are called $C^*$-algebraic bundles. As already
indicated above, Fell bundles incorporate the theory of actions, even
that of (twisted) partial actions of groups.  We refer to
\cite{Exel:TwistedPartialActions} for the general construction of
(semidirect product) Fell bundles associated with (twisted) partial
actions of groups on $C^*$-algebras. To a Fell bundle $\B$ we can attach
two $C^*$-algebras, the so-called full and reduced cross-sectional
$C^*$-algebras of $\B$, denoted usually by $C^*(\B)$ and
$C^*_\red(\B)$. They play the role of the full $A\rtimes_\alpha G$ and reduced $A\rtimes_{\alpha,\red}G$ crossed products in the context of (partial) actions.

The approximation property (AP) of Exel makes sense for general Fell bundles; indeed, it was already defined in this context by Exel in \cite{Exel:Amenability} and the results proved there already indicate that it is a good amenability-type condition similar to the amenability of actions as defined by Anantharaman-Delaroche, that is, the AD-amenability. Since it is now known that these notions of amenability coincide for ordinary global actions, it is a natural question whether one can also extend AD-amenability to Fell bundles and prove a similar result. It is one of the main goals of this paper to address this question and introduce some amenability notions for Fell bundles, like AD-amenability and the WAP, proving that they end up all being equivalent. Besides the interesting fact that this gives different characterisations of amenability for Fell bundles, this also has other advantages and applications.

The first natural candidate for an amenability notion of a Fell bundle is the AP, and this is the spirit of our article. We are going to show that any other reasonable notion of amenability ends up being equivalent to the AP. 

There is a natural way to `convert' a Fell bundle $\cB=\{B_t\}_{t\in G}$ into an ordinary action while still keeping many of the original features of $\cB$; this applies in particular to amenability. More precisely, one can assign to $\cB$ the so-called $C^*$-algebra of kernels $\bk(\cB)$. This is defined in \cite{Abadie:Enveloping} and it carries a canonical action $\beta$ of $G$ that plays the role of the `Morita enveloping action' of $\B$. Indeed, it is proved in \cite{Abadie-Ferraro:Equivalence_of_Fell_Bundles} that $\cB$ is always `weakly equivalent' -- in a certain precise sense -- to the (semi-direct product) Fell bundle $\cB_\beta$. If $\cB=\cB_\alpha$ already comes from an action $\alpha$ of $G$, then $\alpha$ and $\beta$ are Morita equivalent as actions, hence amenability passes from $\alpha$ to $\beta$. More generally, we are going to show that all (equivalent) notions of amenability for Fell bundles are preserved by weak equivalence of Fell bundles. In particular a Fell bundle $\cB$ is AD-amenable if and only if its Morita enveloping $G$-action $\beta$ on $\bk(\cB)$ is AD-amenable, or equivalently, it has the AP. This has some interesting applications. For instance, taking advantage of this fact, one can answer a long-standing open question of Ara, Exel and Katsura (see \cite{Ara-Exel-Katsura:Dynamical_systems}*{Remark~6.5}) whether the nuclearity of $C^*_r(\B)$ for a Fell bundle $\B$ implies its approximation property. This is answered affirmatively in \cite{buss2020amenability}*{Corollary~4.23} using some of the main results of this paper.

Although one could define amenability of $\cB$ in terms of $\bk(\cB)$, we choose a more direct and intrinsic approach: we pass from $\cB$ to its bidual Fell bundle $\cB''$, which bring us back to the von Neumann algebra setting. Hence, in order to study amenability for Fell bundles, we take a detour via W*-Fell bundles (which we define and study here) mainly because approximation properties are easier to deal with in this context. We define AD-amenable Fell bundles and W*-Fell bundles using W*-Fell bundles and certain partial actions on $C^*$-algebras. Only after that we prove that amenability can be characterised via approximation properties.

We must say that in a first preprint version of the present article, published in the ArXiv, we were only able to show that the AD-amenability of the canonical action on $\bk(\cB)$ was equivalent to a weak version of the AP for $\cB$ (which we conveniently named the weak approximation property, WAP). Ozawa-Suzuki's characterisation of AD-amenability of $C^*$-dynamical systems in terms of the AP was the missing piece we needed to prove the WAP was equivalent to the AP. We thank the reviewer for pointing out this fact to us. We were able to adapt Ozawa-Suzuki's arguments to Fell bundles quite easily (Theorem~\ref{thm:the mega theorem}) and this has simplified many arguments of the original version of this article.

The structure of the paper is as follows. In Section~\ref{sec:partial-actions-vN} we introduce and study the notion of partial actions of groups on W*-algebras. It seems this has not been studied before, but it will be important to us here as it opens a canonical general link between the $C^*$- and W*-theory of partial actions. In particular we show that every partial action $\alpha$ on a $C^*$-algebra $A$ extends to a canonical enveloping W*-partial action $\alpha''$ on $A''$. One of the main results of this section states that every W*-partial action admits an enveloping global W*-action. This allows us to canonically extend Anantharaman-Delaroche's notion of amenability to partial actions on $C^*$- and W*-algebras; we do this in Section~\ref{sec:ADA-partial}. We give some basic examples and prove that amenability in this sense behaves well with respect to taking restrictions and enveloping actions. In Section~\ref{sec:Morita-ADA-partial} we prove that AD-amenability is preserved by equivalences of partial actions, both for $C^*$- and W*-algebras.

In Section~\ref{sec:ADA-Fell-bundles} we start to extend the theory of
AD-amenability to Fell bundles. We introduce the notion of W*-Fell
bundles and prove that every Fell bundle admits a canonical enveloping
W*-Fell bundle. We also introduce the W*-algebra of kernels of a
W*-Fell bundle in this section, which is a W*-counterpart of the
$C^*$-algebra of kernels of a Fell bundle introduced in
\cite{Abadie:Enveloping}. This W*-algebra always carries a global
W*-action and, by definition, the bundle is W*AD-amenable if this action is
W*AD-amenable (Definition~\ref{def:AD amenable}). We show that the canonical action on the $C^*$-algebra of kernels of a Fell bundle $\cB$ is AD-amenable if and only if $\cB''$ is W*AD-amenable.

In Section~\ref{sec:AP-ADA} we study Exel's approximation property and translate it into the context of W*-Fell bundles.
We prove that this new property, the W*AP, gives an alternative description of W*AD-amenability. Thus a Fell bundle $\cB$ is AD-amenable if and only if $\cB''$ has the W*AP. We will say $\cB$ has the \textit{weak AP} (WAP) if $\cB''$ has the W*AP and will justify this name by proving that the WAP is equivalent to the condition one obtains by taking Exel's original definition of the AP and replacing norm convergence by weak convergence (in the Banach space sense). Hence, AP implies WAP and a convex combination argument adapted from \cite{ozawa2020characterizations} will be used to prove the converse. Once this is done, AD-amenability and the AP become the same thing.

The seventh section is dedicated to study what kind of properties pass from the unit fibre $B_e$ of a Fell bundle $\cB$ to the cross-sectional $C^*$-algebra $C^*(\cB)$ provided $\cB$ has the AP. In \cite{ExelNg:ApproximationProperty} Exel and Ng associate a reduced cross-sectional $C^*$-algebra $C^*_{\red}(\cB)$ to every Fell bundle (over a locally compact group) and show $C^*(\cB)=C^*_{\red}(\cB)$ provided that $\cB$ has the AP. If $G$ is discrete and the unit fibre $B_e$ of $\cB$ is nuclear, then $C^*_{(\red)}(\cB)$ is nuclear if and only if $\cB$ has the AP \cites{Exel:Partial_actions,buss2020amenability}; where the parenthesis $(\red)$ means that the claim holds for full and reduced cross-sectional $C^*$-algebras. We show how to obtain these two results, among others, from Anantharaman-Delaroche's original results by using the canonical actions on the $C^*$-algebras of kernels.

Finally, in the last section, we study amenability for Fell bundles $\cB=\{B_t\}_{t\in G}$ whose unit fibre $B_e$ is Morita equivalent to a commutative $C^*$-algebra. The amenability of such bundles is equivalent to the amenability of a $C^*$-partial action of $G$ on $C_0(\widehat{B}_e)$ defined by $\cB,$ where $\widehat{B}_e$ denotes the spectrum of $B_e$ formed by equivalence classes of its irreducible representations.

We shall only consider discrete groups in this paper, although probably many of the things we do here also extend to locally compact groups. The approximation property of Fell  bundles has been extended to locally compact groups by Exel and Ng in \cite{ExelNg:ApproximationProperty}. The notions of equivalences of Fell bundles are also available for bundles over locally compact groups \cites{Abadie-Ferraro:Equivalence_of_Fell_Bundles,Abadie-Buss-Ferraro:Morita_Fell}. Anantharaman-Delaroche defines amenability for actions of discrete groups in \cite{Anantharaman-Delaroche:Systemes} although she actually considers locally compact groups when acting on von Neumann algebras \cites{Anantharaman-Delaroche:ActionI,Anantharaman-Delaroche:ActionII}.  AD-amenable actions of locally compact groups on $C^*$-algebras have received some attention recently and the equivalence between AD-amenability and Exel-Ng's approximation property was shown to hold for $C^*$-dynamical systems \cites{buss2020amenability,ozawa2020characterizations,mckee2020amenable}\footnote{Notice that the first version of the present article published in the ArXiv predates all those works, and some of them actually use ideas from this former version.}. We would not be surprised if most of the statements of the present article were shown to hold for Fell bundles over locally compact groups (see, for example, Proposition~\ref{prop:kB exact iff Be exact} and the proof of \cite{buss2020amenability}*{Theorem 3.11}).

\section{Partial actions on von Neumann algebras}\label{sec:partial-actions-vN}

A lot is already known about partial actions of groups on $C^*$-algebras but it seems that partial actions on W*-algebras (i.e. von Neumann algebras) have never been studied.
Probably the main reason is that every partial action in this setting is the restriction of a global action on a W*-algebra (see Proposition~\ref{pro:enveloping-vn-partial}).
However, starting with a partial action $\alpha$ on a $C^*$-algebra $A$, its bidual $A''$ von Neumann algebra carries a natural partial action $\alpha''$ that will serve as one of our main tools in this paper.
This is the reason why we develop the basic theory of partial actions on von Neumann algebras here.

\begin{definition}\label{def:set theoretic partial action}
A partial action of a group $G$ on a set $X$ is a family of functions $\sigma=\{\sigma_t\colon X_{t^{-1}}\to X_t\}_{t\in G}$ such that:
\begin{enumerate}[(i)]
\item For every $t\in G$, $X_t$ is a subset of $X$.
\item $X_e=X$ and $\alpha_e$ is the identity of $X$.
\item Given $s,t\in G$ and $x\in X_{t^{-1}}$ such that $\sigma_t(x)\in X_{s^{-1}}$,  it follows that $x\in X_{(st)^{-1}}$ and $\sigma_{st}(x)=\sigma_s(\sigma_t(x))$.
\end{enumerate}
\end{definition}

An action (or global action) is a partial action such that $X_t=X,$ $\forall\ t\in G$.

Given two partial actions, $\sigma$ and $\tau$ of $G$ on sets $X$ and $Y$ respectively, a morphism $f\colon \sigma\to \tau $ is a function $f\colon X\to Y$ such that $f(X_t)\sbe Y_t$ and $f(\sigma_t(x))=\tau_t(f(x))$ for all $t\in G$ and $x\in X_{t^{-1}}$. The composition of morphisms is just the composition of functions.

The restriction of $\sigma$ to a subset $Y\sbe X$ is $\sigma|_Y:=\{ \sigma|_{Y,t}\colon Y_{t^{-1}}\to Y_t  \}$, where $Y_t:=Y\cap \sigma_t(X_{t^{-1}}\cap Y)$ and $\sigma|_{Y,t}(y)=\sigma_t(y)$.
It follows that $\sigma|_Y$ is a partial action of $G$ on the set $Y$
and, if $Z\sbe Y$, then $\sigma|_Y|_Z=\sigma|_Z$. When
a partial action $\tau$ can be expressed as a restriction
$\tau=\sigma|_Y$ of a global action $\sigma$, we say that $\tau$ is
globalisable, and that $\sigma$ is a globalisation of $\tau$. 

In this article we work with discrete groups exclusively, so here
``group'' actually means ``discrete group''  (unless otherwise specified).

A $C^*$-partial action $\alpha=\{\alpha_t\colon A_{t^{-1}} \to A_t\}_{t\in G}$ of a group $G$ on a $C^*$-algebra $A$ is a partial action of $G$ on the set $A$ for which each $A_t$ is a closed two-sided ideal of $A$, and each
$\alpha_t$ is an isomorphism of $C^*$-algebras. If
$\beta$ is a global action of $G$ on a $C^*$-algebra $B$ and $A$ is a
closed two-sided ideal of $B$, then the restriction $\alpha:=\beta|_A$
is a $C^*$-partial action on $A$.

In case $B$ equals $I:=\cspn\{\beta_t(A)\colon t\in G\}$ we say $\beta$ ($B,$ respectively) is an enveloping action (enveloping algebra) for $\alpha.$
The equality $I=B$ is a minimality condition on the globalisation.
The restriction $\beta|_I$ is always an enveloping action for $\alpha,$ so the existence of globalisations is equivalent to the existence of enveloping actions.

Enveloping actions (and algebras) are unique if they exist; and up to Morita equivalence of partial actions they do always exist \cite{Abadie:Enveloping}.
A necessary and sufficient condition for the existence of enveloping actions is given in \cite{MR3795739}.
In our context we highlight the fact that if each $A_t$ is unital, then $\alpha$ has an enveloping action, this is always the case for the W*-partial actions we define below.

\begin{definition}
A \emph{W*-partial action} of a group $G$ on a W*-algebra $M$ is a set theoretic partial action of $G$ on $M$, $\gamma \defeq \{\gamma_t\colon M_{t^{-1}}\to M_t \}_{t\in G}$, where each $M_t$ is a W*-ideal of $M$ (possibly $\{0\}$) and each $\gamma_t$ is a W*-isomorphism.
A morphism of W*-partial actions is just a morphism of set theoretic partial actions which is also a morphism of W*-algebras (a $\weaks$-continuous morphism of *-algebras).
\end{definition}

Here we always view W*-algebras as Banach space duals, $M\cong (M_*)'$, endowed with the $\weaks$-topology.
A W*-ideal is then just a \Star{}ideal of $M$ that is closed for the $\weaks$-topology. A W*-isomorphism between two W*-algebras is a \Star{}isomorphism that is $\weaks$-continuous. Actually, every \Star{}isomorphism between W*-algebras is normal (preserves suprema of increasing bounded nets) and it is therefore automatically a W*-isomorphism \cite{MR2188261}*{Proposition~III.2.2.2}.

As in the case of actions on sets or on $C^*$-algebras,
we have suitable notions of restriction and globalisation of
partial actions in the category of $W^*$-algebras:  
   
\begin{example}[Restriction and globalisation]\label{exa:restriction} 
Given an ordinary (global) W*-action $\gamma$ of $G$ on a
W*-algebra $N$ and a W*-ideal $M\unlhd N$, the restriction
$\gamma|_M$ is a W*-partial action. More generally, the restriction of a W*-partial
action to a W*-ideal is again a W*-partial action.
When a given W*-partial action $\alpha$ on a W*-algebra $M$ can be written as
  $\alpha:=\gamma|_M$, where $\gamma$ is a global W*-action on a
  W*-algebra that contains $M$ as a W*-ideal, we say
  that $\alpha$ is globalisable, and that $\gamma$ is a globalisation
  of $\alpha$. 
\end{example}

\begin{example}[Bidual partial actions]
Given a $C^*$-partial action $\alpha=\{\alpha_t\colon A_{t^{-1}}\to A_t\}_{t\in G}$ of $G$ on a $C^*$-algebra $A$, the double dual (enveloping) W*-algebra $A''$ of $A$ carries a canonical W*-partial action $\alpha'':=\{\alpha_t''\colon A_{t^{-1}}''\to A_t''\}_{t\in G}$ which is the unique W*-partial action such that $\alpha''|_A=\alpha$.
Here we view the bidual algebra $A_t''$ as a W*-ideal of $A''$ and $\alpha_t''$ as the unique $\weaks$-continuous extension of $\alpha_t$.
\end{example}

One of our goals is to show that every W*-partial action is (isomorphic to) a restriction of a global W*-action as in Example~\ref{exa:restriction}.
One may think that this is trivial since every von Neumann algebra has a unit, so that all the ideals of a W*-partial action are unital.
However the following example shows that the $C^*$-enveloping action might be not a W*-algebra.

\begin{example}\label{exa:'trivial'-partial-action}
Consider the ``trivial'' partial action of $G$ on the W*-algebra $M\defeq \C$ in which all the ideals are zero except for $M_e\defeq M$.
This can also be viewed as the restriction of the global action of $G$ by (left) translations on the $C^*$-algebra $\contz(G)$ to the ideal $\C\cong \C\delta_e\sbe \contz(G)$. Moreover, since the linear orbit of this ideal is dense in the entire algebra $\contz(G)$, this action is (up to isomorphism) the enveloping action of the original partial action on $\C$.
But if $G$ is infinite, $\contz(G)$ is not a W*-algebra. On the other hand, we may also view $\C\cong \C\delta_e$ as a W*-ideal of the W*-algebra $\ell^\infty(G)$. And since the linear orbit of this ideal is $\weaks$-dense, this is a W*-enveloping action in the following sense.
\end{example}

\begin{definition}
A \emph{W*-enveloping action} of a W*-partial action $\gamma$ of a group $G$ on a W*-algebra $M$ is a W*-global action $\sigma$ of $G$ on a W*-algebra $N$ together with a W*-ideal $\tilde{M}$ of $N$ and an isomorphism of W*-partial actions $\iota\colon \gamma \to \sigma|_{\tilde{M}}$, such that the linear $\sigma$-orbit of $\tilde{M}$ is $\weaks$-dense in $N$.
We summarise this situation by saying that $(N,\sigma)$ is a W*-enveloping action of $(M,\gamma)$.
\end{definition}

\begin{proposition}\label{pro:enveloping-vn-partial}
Every W*-partial action $\gamma$ of $G$ on a W*-algebra $M$ has a W*-enveloping action that is unique up to isomorphism.
\end{proposition}
\begin{proof}
Define $\iota\colon M\to \ell^\infty(G,M)$ by $\iota(x)(t)\defeq \gamma_{t^{-1}}(1_t x)$, where $1_t$ denotes the unit of the W*-algebra $M_t$ ($1_t=0$ if $M_t=\{0\}$).
This unit is a central projection of $M$ because $M_t$ is a W*-ideal of $M$.
The map $\iota$ is an injective $\weaks$-continuous \Star{}homomorphism whose image consists of functions $f\in \ell^\infty(G,M)$ with $f(t)=\gamma_{t^{-1}}(f(e)1_t)$.
The image $\tilde M\defeq \iota(M)$ is a $\weaks$-closed W*-subalgebra of $\ell^\infty(G,M)$ which is therefore isomorphic to $M$ via $\iota$.
We now endow $\ell^\infty(G,M)$ with the $G$-action $\tau$ by left translations: $\tau_t(f)(s)\defeq f(t^{-1}s)$.
This is a W*-global action and $\iota\colon \gamma\to \tau$ is a morphism.
Let $N$ be the $\weaks$-closure of the linear $\tau$-orbit of $\tilde M$, that is, the $\weaks$-closure of $\spn\{\tau_t(f): f\in \tilde M,\ t\in G\}$.
Moreover, $N$ is $\tau$-invariant, so that $\tau$ restricts to a W*-global action $\sigma$ of $G$ on $N$ and this is the desired W*-enveloping action of $(M,\gamma),$ as we now show.

It follows from Definition \ref{def:set theoretic partial action} that $\gamma_t(1_{t^{-1}}1_s)=1_t1_{ts}$ for all $s,t\in G,$ because $1_{t^{-1}}1_s$ and $1_t1_{ts}$ are the units of $M_{t^{-1}}\cap M_s = M_{t^{-1}} M_s$ and  $M_t\cap M_{ts}=M_tM_{ts},$ respectively.
This implies $\iota\colon \gamma\to \sigma$ is a morphism because, for all $s,t\in G$ and $x\in M_t$:
\begin{equation*}
\tau_t(\iota(x))(s)=\gamma_{s^{-1}t}(x1_{t^{-1}s})=\gamma_{s^{-1}}(\gamma_t(x1_{t^{-1}}1_{t^{-1}s}))=\gamma_{s^{-1}}(\gamma_t(x)1_{s})=\iota(\gamma_t(x))(s).
\end{equation*}

In order to prove that $\tilde M$ is a W*-ideal of $N$ and that $\iota\colon \gamma\to \sigma|_{\tilde M}$ is an isomorphism it suffices to show that $\tau_t(\tilde M)\cap \tilde M = \iota(M_t),$ for all $t\in G.$
For all $x,y\in M:$
\begin{align*}
 [\tau_t(\iota(a))\iota(b)](s)
    & = \gamma_{s^{-1}t}(a1_{t^{-1}s})\gamma_{s^{-1}}(b1_s)
      = \gamma_{s^{-1}t}(a1_{t^{-1}s})1_{s^{-1}t}1_{s^{-1}}\gamma_{s^{-1}}(b1_s)\\
    & = \gamma_{s^{-1}t}(a1_{t^{-1}s})\gamma_{s^{-1}t}(1_{t^{-1}s}1_{t})\gamma_{s^{-1}}(b1_s) \\
    & = \gamma_{s^{-1}t}(a1_{t^{-1}s}1_{t^{-1}})\gamma_{s^{-1}}(b1_s)
       = \gamma_{s^{-1}}(\gamma_t(a1_{t^{-1}})b1_{s^{-1}})\\
    & = \iota(\gamma_t(a1_{t^{-1}})b)(s).
\end{align*}
We then conclude that $\tau_t(\tilde M)\cap \tilde M = \tau_t(\tilde M)\tilde M\subseteq \iota(M_t)\subseteq \tau_t(\tilde M)\cap \tilde M,$ where the last inclusion follows from the fact that $\iota\colon \gamma\to \sigma$ is a morphism.
At this point we know $(N,\sigma)$ is a W*-enveloping action for $(M,\gamma).$

For uniqueness, assume that $M$ is a W*-ideal of a W*-algebra $\tilde N$ carrying a W*-global action $\tilde\sigma$ whose restriction to $M$ is $\gamma$ and such that the linear $\tilde\sigma$-orbit of $M$ is $\weaks$-dense in $\tilde N$ (for simplicity, we omit the inclusion map $M\into\tilde N$ here, that is, we already assume $M\sbe \tilde N$).
Then we extend $\iota$ to $\tilde\iota\colon \tilde N\to \ell^\infty(G,M)$ by $\tilde\iota(x)(t)\defeq \tilde\sigma_{t^{-1}}(x)1_e$.
First we show that $\tilde \iota$ is in fact an extension of $\iota$.
For $x\in M$:
\begin{equation*}\tilde\iota(x)(t) = \tilde\sigma_{t^{-1}}(x)1_e=\tilde\sigma_{t^{-1}}(x)1_{t^{-1}}=\tilde\sigma_{t^{-1}}(x1_t)=\sigma_{t^{-1}}(x1_t) = \iota(x)(t)\end{equation*}
because $\tilde\sigma_{t^{-1}}(x)1_e\in \tilde\sigma_{t^{-1}}(M)M=M1_t$.
A similar computation shows that $\tilde\iota$ is equivariant.
Observe that $\tilde\iota$ is injective (hence isometric) because, if $\tilde\iota(x)=0$, then $x\tilde\sigma_t(1_e)=0$ for all $t\in G$.
This is equivalent to $xy=0$ for all $y$ in the linear $\tilde\sigma$-orbit of $M$, which is $\weaks$-dense in $\tilde N$ by assumption.
Since $\tilde \iota$ is an isometry and it is $\weaks$-continuous in $\{x\in \tilde N\colon \|x\|\leq 1\}$, its range is $\weaks$-closed (hence a W*-subalgebra) and it is a W*-isomorphism over its image.
Finally, $\tilde\iota(\tilde N)$ is the $\weaks$-closure of the linear span of $\tilde\iota(M)=\iota(M)$.
Thus $\tilde\iota(\tilde N)=N$ and $\tilde\sigma$ is isomorphic, as a W*-partial action, to $\sigma$.
\end{proof}

\begin{proposition}\label{pro:center-enveloping}
Let $\gamma=\{\gamma_t\colon M_{t^{-1}}\to M_t\}_{t\in G}$ be a W*-partial action of $G$ on a W*-algebra $M$ and let $(N,\sigma)$ be its enveloping W*-action. Then the restriction $\gamma|_{Z(M)}=\{\gamma_t\colon Z(M_{t^{-1}})\to Z(M_t)\}$ of $\gamma$ to $Z(M)$ is a W*-partial action whose enveloping W*-action is the restriction of $\sigma$ to $Z(N)$.
\end{proposition}
\begin{proof}
First notice that $Z(M_t)$ is indeed a W*-ideal of $Z(M)$. In fact, if $M_t=1_tM$, where $1_t\in Z(M)$ is the central projection of $M$ representing the unit of $M_t$, then $Z(M_t)=1_tZ(M)$. It is then clear that the restriction of $\gamma$ to $Z(M)$ defines a W*-partial action. For the same reason, viewing $M$ as a W*-ideal of $N$, $M$ is then the ideal generated by the central projection $p=1_e$, and then $Z(M)=pZ(N)$ is the W*-ideal of $Z(N)$ generated by the same projection.
The restriction $\sigma|_{Z(N)}$ is clearly $\sigma|_{Z(N)}|_{Z(M)}= \sigma|_{Z(M)} = \sigma|_M|_{Z(M)} =\gamma|_{Z(M)}$.
To see that $\sigma|_{Z(N)}$ is the enveloping action of $\gamma|_{Z(M)}$ it remains to show that the linear $\sigma$-orbit of $Z(M)$ is $\weaks$-dense in $Z(N)$. For each finite subset $F\sbe G$, we define $M_F:=\sum_{t\in F} \sigma_t(M)$. This is a W*-ideal of $N$ (being a finite sum of such) and the union of all these ideals is the linear $\sigma$-orbit of $M$, so it is $\weaks$-dense in $N$ since $(N,\sigma)$ is the enveloping action of $(M,\gamma)$. On the other hand the linear $\sigma$-orbit of $Z(M)$ is the $\weaks$-closure $P$ of the ideal $\cup_{F}Z(M)_F$, where $Z(M)_F=\sum_{t\in F}\sigma_t(Z(M))$. Notice $P$ is a W*-ideal of $Z(N)$. To see that $P=Z(N)$ it is enough to show that the unit of $N$ is contained in $P$. For this notice that the unit $1_F$ of $M_F$ is a (finite) linear combination of $1_t$, and this is also the unit of $Z(M)_F$. The net $(1_F)_F$ is increasing and bounded and its $\weaks$-limit is the unit of $N$ because $\cup_F M_F$ is $\weaks$-dense in $N$. However this limit is also the unit of $P$.
\end{proof}

\begin{remark}\label{rem:abelian-injective}
If $(N,\sigma)$ is a W*-enveloping action of $(M,\gamma)$, then $N$ is abelian if and only if $M$ is abelian.
Indeed, clearly $M$ is abelian if $N$ is. For the converse observe that, by the proof of the Proposition~\ref{pro:enveloping-vn-partial}, $N$ is isomorphic to a W*-subalgebra of the abelian W*-algebra $\ell^\infty(G,M)$.

Another property that is preserved by taking enveloping actions is
injectivity, in the sense that if $N$ is the W*-enveloping action
of $M$, then $M$ is injective if and only if $N$ is injective. Indeed,
since injectivity passes to ideals, the reverse direction is
clear. For the converse one uses that injectivity passes to (finite)
sums and directed unions of ideals and the description of $N$ as the $\weaks$-closure of $\cup_F M_F$ as in the proof of the last proposition above.
\end{remark}

\section{Amenability of partial actions}\label{sec:ADA-partial}

First let us recall the notion of amenability for (global) actions of  groups on $C^*$-algebras and W*-algebras introduced by Anantharaman-Delaroche, see \cites{Anantharaman-Delaroche:Systemes, Anantharaman-Delaroche:ActionI,Anantharaman-Delaroche:ActionII}.

\begin{definition}
We say that a (global) action of a group $G$ on a W*-algebra $M$ is \emph{Anantharaman-Delaroche amenable} (\emph{W*AD-amenable for short}) if there exists a linear positive contractive and $G$-equivariant map $P\colon\ell^\infty(G,M)\onto M$ whose composition with the canonical embedding (by constant functions) $M\into \ell^\infty(G,M)$ is the identity map $M\to M$. Here $\ell^\infty(G,M)$ is endowed with the diagonal $G$-action: $\tilde{\gamma}_t(f)(r)=\gamma_t(f(t^{-1}r))$, where $\gamma$ denotes the $G$-action on $M$.

An action $\alpha$ of $G$ on a $C^*$-algebra $A$ is \emph{AD-amenable} 
if the corresponding double dual W*-action $\alpha''$ on $A''$ is W*AD-amenable.
\end{definition}

\begin{example}\label{exa:traslation action}
The translation $G$-action on itself, viewed as a $G$-action on the $C^*$-\nb-algebra $\contz(G),$ is always AD-amenable (this action is even proper). By definition, this means that the translation action on $\contz(G)''\cong\ell^\infty(G)$ is W*AD-amenable as a W*-action. However the translation action on $\ell^\infty(G)$ is AD-amenable as a $C^*$-action if and only if $G$ is exact, see \cite{Brown-Ozawa:Approximations}*{Theorem~5.1.7}.

More generally, if $M$ is a $G$-W*-algebra, the $G$-W*-algebra $\ell^\infty(G,M)$ endowed with the diagonal $G$-action is always W*AD-amenable because we have a canonical $G$-equivariant unital embedding $\ell^\infty(G)\into Z\ell^\infty(G,M)=\ell^\infty(G,Z(M))$ (see \cite{Anantharaman-Delaroche:ActionII}*{Corollary~3.8}). As before, here $\ell^\infty(G)$ carries the translation $G$-action.
\end{example}

Before we proceed, let us highlight some of the most important characterisations of AD-amenability obtained by Anantharaman-Delaroche in her papers \cites{Anantharaman-Delaroche:Systemes, Anantharaman-Delaroche:ActionI,Anantharaman-Delaroche:ActionII}.

\begin{theorem}[Anantharaman-Delaroche]\label{the:ADA-characterisations}
The following are equivalent for a global action $\gamma$ of a group $G$ on a W*-algebra $M$:
\begin{enumerate}[(i)]
\item $\gamma$ is W*AD-amenable;
\item the restriction of $\gamma$ to the center $Z(M)$ is W*AD-amenable, that is, there is a $G$-equivariant norm-one projection $\ell^\infty(G,Z(M))\onto Z(M)$;
\item there is a net $\{a_i\colon G\to Z(M)\}_{i\in I}$ of finitely supported functions with $$\braket{a_i}{a_i}_2\defeq \sum_{g\in G}a_i(g)^*a_i(g)\leq 1$$ for all $i$ and $\braket{a_i}{\tilde\gamma_g(a_i)}_2\to 1$ ultraweakly for all $g\in G$.
\end{enumerate}
Moreover, if $M$ is injective as a W*-algebra, then the above are also equivalent to
\begin{enumerate}
\item[(iv)] the W*-crossed product $M\bar\rtimes G$ is an injective W*-algebra.
\end{enumerate}
If $\alpha$ is an AD-amenable action of $G$ on a $C^*$-algebra $A$, then the full and reduced $C^*$-crossed products coincide, that is, $A\rtimes_{\alpha}G=A\rtimes_{\red,\alpha}G$. And if $A$ is nuclear, then $\alpha$ is AD-amenable if and only if $A\rtimes_{\red,\alpha} G$ is a nuclear $C^*$-algebra.
\end{theorem}
Let us also remark that for an action on a commutative $C^*$-algebra $A=\contz(X)$, its AD-amenability is equivalent to amenability of the associated transformation groupoid $X\rtimes G$ in the sense of Anantharaman-Delaroche and Renault, see \cite{Renault_AnantharamanDelaroche:Amenable_groupoids}. This is usually called \emph{topological amenability} of the $G$-action on $X$.
Moreover, the AD-amenability in this case is equivalent to the existence of a net $\{a_i\colon G\to Z(A)=A\}_{i\in I}$ with the same properties as in (iii) above, except that the ultraweak convergence in (iii) can be strengthened to the convergence with respect to the strict topology on $A\sbe M(A)$ (the multiplier algebra), see \cite{Anantharaman-Delaroche:Systemes}*{Th\'eor\'eme~4.9}. This cannot be expected -- and indeed it is not true -- for noncommutative algebras because simple unital $C^*$-algebras can carry AD-amenable actions of non-amenable groups, see Remark~\ref{rem:Suzuki}.

We are now ready to introduce the notion of amenability for partial actions on $C^*$-algebras and W*-algebras:

\begin{definition}
A partial action of a group $G$ on a W*-algebra $M$ is \emph{W*AD-amenable} if its W*-enveloping action
is W*AD-amenable.

We say that a partial action of $G$ on a $C^*$-algebra $A$ is \emph{AD-amenable}
if the induced W*-partial action on $A''$ is W*AD-amenable.
\end{definition}

Of course, a global action is AD-amenable if and only if it is AD-amenable as a partial action. Before we give some proper examples of amenable partial actions, we observe the following general fact:

\begin{proposition}\label{prop:AD-amenable-center}
A W*-partial action $(M,\gamma)$ is W*AD-amenable if and only if its restriction to the center $Z(M)$ is W*AD-amenable.
\end{proposition}
\begin{proof}
This follows directly from the definition, Proposition~\ref{pro:center-enveloping} and Theorem~\ref{the:ADA-characterisations}.
\end{proof}

\begin{remark}\label{rem:Suzuki}
The above result does not hold for partial actions on $C^*$-algebras, not even for global actions. Indeed, the results of Suzuki in \cite{MR3589332} show that every exact group admits an AD-amenable action on a unital simple (and nuclear) $C^*$-algebra. Such an algebra has trivial center and a trivial global action can only be AD-amenable if the group is amenable.
\end{remark}

\begin{example}\label{exa:trivial-partial}
The ``trivial'' partial action of $G$ on $A=\C$ appearing in Example~\ref{exa:'trivial'-partial-action} is AD-amenable, both in $C^*$- and W*-sense. This is because $A=A''$ has as its enveloping W*-action the global translation $G$-action on $\ell^\infty(G)$ as explained in Example~\ref{exa:'trivial'-partial-action}, and this W*-action is W*AD-amenable.
In the same way, we can consider the partial action on a W*-algebra $M$ as in Example~\ref{exa:'trivial'-partial-action} with all domain ideals $M_g=0$ except for $M_e=M$. This partial action is always W*AD-amenable because its enveloping W*-action is the translation action on $\ell^\infty(G,M)$, which is W*-amenable by Example~\ref{exa:traslation action}.
For the same reason, any $C^*$-algebra $A$ endowed with the ``trivial'' partial action (in which all the domain ideals are $A_g=0$ except for $A_e=A$) is always AD-amenable because then $A''$ carries the ``trivial'' partial $G$-action which is W*AD-amenable.
\end{example}

The example above can be generalised as follows.

\begin{example}
Take an amenable subgroup $H\sbe G$ acting (globally) on a $C^*$-algebra $A$ (or on a W*-algebra $M$) and ``extend'' this to a partial $G$-action on $A$ ``by zero'' in the sense that $A_h=A$ for $h\in H$, $A_g=0$ for $g\in G\backslash H$ and $\alpha_g\colon A_{g^{-1}}\to A_g$ acts via the original $H$-action for $g\in H$ and by zero otherwise.  This partial action (which is global only if $H=G$) is always AD-amenable.
Indeed, the canonical action of $G$ on $\ell^\infty(G/H)$, $\nu_t(f)(sH)=f(t^{-1}sH)$, plays an important role here.
This W*-action is W*AD-amenable if (and only if) $H$ is amenable. Indeed, $\ell^\infty(G/H)=\contz(G/H)''$ and the crossed product $\contz(G/H)\rtimes_{\red} G$ is Morita equivalent to $C^*_\red(H)$ by Green's imprimitivity theorem.

To see that the partial action of $G$ defined above is AD-amenable, it is enough to consider the von Neumann algebraic situation. For this, let us write $\gamma$ for the action of $H$ on a W*-algebra $M$ and name $\bar\gamma$ its extension to $G$.

To prove amenability of this partial $G$-action, we give an explicit description of its W*-enveloping action. Consider the W*-subalgebra of $\ell^\infty(G,M)$
\begin{equation*}N:=\{f\in  \ell^\infty(G,M)\colon  f(s)=\gamma_{s^{-1}t}(f(t))\mbox{ if }sH=tH \}.\end{equation*}
This subalgebra is invariant under the action $\tau$ of $G$ on $\ell^\infty(G,M)$ given by $\tau_t(f)(s)=f(t^{-1}s)$.
We name $\delta$ the restriction of $\tau$ to $N$.
In order to view $M$ as a W*-ideal of $N$ in such a way that $\delta$ is the W*-globalisation of $\bar\gamma$, we consider the map $\iota\colon M\to N$ given by
\begin{equation*} \iota(a)(s)=\begin{cases}
                \gamma_{s^{-1}}(a) \mbox{ if }s\in H\\
                0 \mbox{ if }s\notin H.
               \end{cases}
\end{equation*}
Note that in case $M=\C$, we have $N=\ell^\infty(G/H)$ and $\tau=\nu$.
In any case, we may view $\ell^\infty(G/H)$ as a unital $\delta$-invariant subalgebra of $Z(N)$ by considering the inclusion $\kappa\colon \ell^\infty(G/H)\to N$, $\kappa(f)(t)=f(tH)$.
Moreover, the restriction of $\delta$ to $\ell^\infty(G/H)$ is $\nu$, from which it follows that $\delta$ is W*AD-amenable, hence so is $\bar\gamma$.
\end{example}

\begin{example}\label{exa:restriction to subgroup}
If $(M,\gamma)$ is an AD-amenable partial W*-action of $G$ and $H\sbe G$ is any subgroup, then the restriction of the partial $G$-action on $M$ to $H$, namely, $\gamma|_H=\{\gamma_h\colon M_{h^{-1}}\to M_h\}_{h\in H}$ is also AD-amenable.
A similar assertion holds for $C^*$-partial actions.
Indeed, it is clearly enough to check this for W*-partial
actions. This is known to hold for global W*-actions and a simple proof is as follows. Let $C$ be the set of cosets of the form $Ht$ ($t\in G$) and choose a function $C\to G,\ \alpha\mapsto t_\alpha$,  such that $t_\alpha\in \alpha$ for all $\alpha\in C.$
Define the function $\iota\colon \ell^\infty(H,M)\to \ell^\infty(G,M)$ as the unique such that $\iota(f)(st_\alpha)=f(s)$ for all $f\in \ell^\infty(H,M),$ $s\in H$ and $\alpha \in C.$ If $P\colon \ell^\infty(G,M)\to M$ is a positive, contractive and equivariant linear projection, then $Q:=P\circ \iota\colon \ell^\infty(H,M)\to M$ is also.
Thus $\gamma|_H$ is W*AD-amenable if $\gamma $ is.

Now for a general W*-partial action $(M,\gamma)$ of $G$, take its
globalisation W*-action $(N,\sigma)$. For simplicity we identify
$M$ as a W*-ideal of $N$. Let $H\cdot M=\sum_{t\in H}\sigma_t(M)$
be the $H$-linear orbit of $M$ in $N$. Notice that this is an ideal of
$N$ and its $\weaks$-closure $N_H:=\overline{H\cdot M}^{\weaks}$ is an
$H$-invariant W*-ideal of $N$ which can be viewed as an $H$-globalisation of $\gamma|_H$. If $(M,\gamma)$ is W*AD-amenable, then by definition $(N,\sigma)$ is W*AD-amenable, and then so is $(N,\sigma|_H)$ and hence also every $H$-invariant W*-ideal, like $(N_H,\sigma|_H)$. Therefore $(M,\gamma|_H)$ is W*AD-amenable.
\end{example}

Next we look at restrictions of partial actions to ideals and prove that amenability behaves nicely also in this direction.

\begin{proposition}\label{prop:restriction of AD amenable are amenable}
The restriction of a W*AD-amenable W*-partial action of a group to a W*-ideal is again W*AD-amenable. Similarly, AD-amenability is preserved by restrictions to $C^*$-ideals.
\end{proposition}
\begin{proof}
 First we deal with W*-partial actions.
 Let $\gamma$ be a W*AD-amenable W*-partial action of the group $G$ on $M$ and let $J$ be a W*-ideal of $M$.
 We know that $M$ can be viewed as a W*-ideal of a W*-algebra $N$ carrying a W*AD-amenable W*-global action $\sigma$ of $G$ with $\sigma|_M=\gamma$.
 Then $J$ is a W*-ideal of $N$ and the $\weaks$-closure of $\sum_{t\in G}\sigma_t(J)$, denoted by $[J]$, is a $\sigma$-invariant W*-ideal of $N$. Moreover, $\sigma|_{[J]}$ is the W*-enveloping action of $\gamma|_J$ because $\sigma|_{[J]}|_J=\sigma|_J=\sigma|_M|_J=\gamma|_J$ and it is also W*AD-amenable because it is a restriction of a global W*AD-amenable W*-action to a $G$-invariant W*-ideal.

 Now let $\beta$ be an AD-amenable $C^*$-partial action of $G$ on $B$ and let $A$ be a $C^*$-ideal of $B$.
Then we may view $A''$ as the $\weaks$-closure of $A$ in $B''$.
 Notice $(\beta|_A)''$ is the unique W*-partial action of $G$ on $A''$ extending $\beta|_A$.
 But $\beta''|_{A''}$ is a W*-partial action  such that $\beta''|_{A''}|_A = \beta''|_A = \beta''|_B|_A=\beta|_A$.
 Thus $\beta''|_{A''}=(\beta|_A)''$.
 By the previous paragraph, $\beta|_A$ is AD-amenable if $\beta$ is.
\end{proof}

\begin{proposition}\label{prop:enveloping action of AD amenable is AD amenable}
 Let $\beta$ be a $C^*$-global action of a group $G$ on $B$ and let $A$ be a $C^*$-ideal of $B$ such that the norm closure of $\sum_{t\in G}\beta_t(A)$ is $B$.
 In other words, $\beta$ is the $C^*$-enveloping action of $\alpha:=\beta|_A$.
 Then $\beta''$ is the W*-enveloping action of $\alpha''$ and $\alpha$ is AD-amenable if and only if $\beta$ is AD-amenable.
\end{proposition}
\begin{proof}
We view $A''$ as the $\weaks$-closure of $A$ in $B''$, thus $A''$ is a W*-ideal of $B''$.
In the proof of Proposition~\ref{prop:restriction of AD amenable are amenable} we showed that $\beta''|_{A''}=\alpha''$.
Thus, to show that $\beta''$ is a W*-enveloping action of $\alpha''$, it suffices to prove that $B''$ is the $\weaks$-closure of $J_0:=\sum_{t\in G}\beta_t''(A'');$ let us write $J$ for this closure.
 Notice $\sum_{t\in G}\beta_t(A)\sbe J_0\sbe J$ and, taking norm closure, this implies $B\sbe J$.
 Now taking $\weaks$-closure we get $J=B''$.
 The rest of the proof follows directly from the definition of AD-amenability for partial actions.
\end{proof}

\section{Morita equivalence of partial actions}\label{sec:Morita-ADA-partial}

Many $C^*$-partial actions do not admit a $C^*$-enveloping action,
but every $C^*$-partial action has a Morita enveloping action, as defined in \cite{Abadie:Enveloping}, which is unique up to Morita equivalence of actions. It is therefore important to see how amenability behaves in terms of Morita equivalences.

We shortly recall  some facts on (Morita) equivalence of $C^*$- and W*-algebras (and partial actions). The machinery described here is not new, see for example \cites{RieffelMoritaEquivalenceCandW,RieffelInducedRep,raeburn1998morita}. Equivalences of $C^*$-partial actions are defined in \cite{Abadie:Enveloping} and here we adapt this concept to the W*-case.

 Given $C^*$-algebras $A$ and $B,$ a Hilbert $A-B-$bimodule is a (not necessarily full) Hilbert left $A-$ right $B-$module $X\equiv {}_AX_B$ with inner products ${}_A\braket{\ }{\ }$ and $\braket{\ }{\ }_B$ such that ${}_A\braket{x}{y}z=x\braket{y}{z}_B$ for all $x,y,z\in X.$
We say $X$ is an equivalence bimodule (imprimitivity bimodule  \cite{raeburn1998morita}*{Definition 3.1}) if the spaces spanned by inner products are norm-dense in $A$ and $B.$
When convenient we think of Hilbert modules as Zettl's ternary $C^*$-rings ($C^*$-trings) specially if we need to represent them as TROs, see \cite{Zl83} for more details.

 Morita equivalence of $C^*$-partial actions is defined in \cite{Abadie:Enveloping} by thinking in terms of $C^*$-trings, thus it is natural to define Morita equivalence of W*-partial actions using Zettl's W*-trings.
In doing so one must require the ideals of \cite{Abadie:Enveloping}*{Definition 4.2} to be $\weaks-$closed and the *-homomorphisms between W*-trings to be $\weaks-$continuous.
Then one can adapt all the constructions from Definition 4.3 to Proposition 4.5 of \cite{Abadie:Enveloping} to the W*-context.
For example, the tensor products of Lemma 4.2 should be replaced with the W*-tensor product and Proposition 4.4 implies that Morita equivalence of W*-partial actions is an equivalence relation.

 We use W*-trings in our arguments mostly because one can treat them as if they were W*-algebras, but we recognise most readers do not prefer this way of thinking.
So it is convenient to state all our definitions and results in terms of Hilbert W*-modules.
Essentially, a W*-equivalence module is a W*-tring for which the map $T$ of \cite{Zl83}*{Theorem 4.1} is the identity map.
Thus we get the following.

\begin{definition}\label{def:Wstar equivalence}
   The W*-partial actions of the group $G$ on the W*-algebras $M$ and $N,$ $\gamma=\{\gamma_t\colon M_{t^{-t}}\to M_t\}_{t\in G}$ and $\delta=\{\delta_t\colon N_{t^{-t}}\to N_t\}_{t\in G}$ respectively, are equivalent if there exists a Hilbert $M-N-$bimodule $X\equiv {}_MX_N$ with inner products ${}_M\braket{\ }{\ }$ and $\braket{\ }{\ }_N$ and a set theoretic partial action $\sigma=\{\sigma_t\colon X_{t^{-1}}\to X_t\}_{t\in G}$ such that:
 \begin{enumerate}
  \item $X$ is isometrically isomorphic to the dual space of a Banach space (the predual $X_*$ is unique by \cite{Zl83}).
  \item For all $t\in G,$ $X_t$ is a $\weaks-$closed linear subspace of $X$ and $NX_t\subseteq X_t\supseteq X_tM.$
  \item For all $t\in G,$ the sets ${}_M\braket{X_t}{X_t}$ and $\braket{X_t}{X_t}_N$ span $\weaks-$dense subspaces of $M_t$ and $N_t,$ respectively.
  \item For all $t\in G$ and $x,y,z\in X_{t^{-1}},$ $\sigma_t\colon X_{t^{-1}}\to X_t$ is linear and 
  \begin{align*}
   {}_M\braket{x}{y}z & = x\braket{y}{z}_N\\
   \sigma_t({}_M\braket{x}{y}z) = \gamma_t({}_M\braket{x}{y})\sigma_t(z) & = {}_M\braket{\sigma_t(x)}{\sigma_t(y)}\sigma_t(z) =\sigma_t(x)\delta_t(\braket{y}{z}_N)
  \end{align*}
 \end{enumerate}
\end{definition}

\begin{remark}
  The definition of equivalence for $C^*$-partial actions is obtained from the one above by forgetting the $\weaks$-topology and replacing ``$\weaks-$closed'' and ``$\weaks-$dense'' with ``norm-closed'' and ``norm-dense'', respectively. 
\end{remark}

\begin{remark}\label{rem:bidual of pa on module}
 Assume $\delta$ is a partial action on an equivalence module ${}_AX_B$ establishing a Morita equivalence between $C^*$-partial actions $\alpha=\{\alpha_t\colon A_{t^{-1}}\to A_t\}_{t\in G}$ and $\beta=\{\beta_t\colon B_{t^{-1}}\to B_t\}_{t\in G}.$
 The linking partial action $\delta=\{\delta_t\colon L_{t^{-1}}\to L_t\}_{t\in G}$ of $\gamma$ (defined in \cite{Abadie:Enveloping}) is a $C^*$-partial action of $G$ on the linking algebra $L$ of $X.$
 By construction, $\delta|_A=\alpha,$ $\delta|_X=\gamma$ and $\delta|_B=\beta.$
 Since the bidual $X''$ may be regarded as the $\weaks-$closure of $X$ in the enveloping W*-algebra $L'',$ we may define the bidual partial action $\gamma''$ as the (set theoretic) restriction of $\delta''$ to $X''.$
 Then $\gamma''$ establishes an equivalence between $\alpha''$ and $\beta''.$
\end{remark}

 The reader may feel we have forgotten some conditions in Definition~\ref{def:Wstar equivalence}, but this is not the case as the following result shows.

\begin{proposition}\label{prop:fact on W-Hilbert modules}
  In the conditions of the last definition above, the following claims hold:
 \begin{enumerate}
  \item The module operations and the inner products are continuous when one of the variables is fixed (i.e. they are ``separate $\weaks-$continuous'').
  \item $X$ is a self-dual Hilbert module (see \cite{Zl83}*{Section 2} or \cite{Blecher-Merdy:Operator}*{Section 8.5}).
  \item For all $t\in G,$ $x,y\in X_{t^{-1}},$ $a\in M_{t^{-1}}$ and $b\in N_{t^{-1}},$
  \begin{align*}
   \gamma_t({}_M\braket{x}{y}) &  ={}_M\braket{\sigma_t(x)}{\sigma_t(y)} & \delta_t(\braket{x}{y}_N) & =\braket{\sigma_t(x)}{\sigma_t(y)}_N\\
   \gamma_t(a)\sigma_t(x) & = \sigma_t(ax) & \sigma_t(x)\delta_t(b) & = \sigma_t(xb).
  \end{align*}
  \item For all $t\in G,$ $\gamma_t\colon X_{t^{-1}}\to X_t$ is a $\weaks-$homeomorphism.
  \item $M$ is the multiplier algebra of the $C^*$-algebra of generalised compact operators $\bK(X)$ or, in other words, $M$ is the $C^*$-algebra $\Lb(X)$ of adjointable operators of $X$ (considered as a right Hilbert $N-$module).
 \end{enumerate}
 \begin{proof}
 We deal with the right $N-$module structure, the claims for the Hilbert $M-$module structure follow by symmetry.

 By \cite{Zl83}*{Proposition 3.2}, we view the $C^*$-algebra $\mathfrak{A}$ of \cite{Zl83}*{Section 4} as the norm closed linear subspace of $N$ generated by the inner products.
Then \cite{Blecher-Merdy:Operator}*{Proposition 8.5.3} says that $N$ is $C^*$-isomorphic to the multiplier algebra $M(\mathfrak{A})$ of \cite{Zl83}*{Section 4}, which is a W*-algebra.
Uniqueness of preduals implies that we may identify $N=M(\mathfrak{A})$ (as W*-algebras).
Hence claims (1), (2) and (5) follow from the results and constructions of \cite{Zl83}*{Section 4}.

 The identity $\delta_t(\braket{x}{y}_N)=\braket{\sigma_t(x)}{\sigma_t(y)}_N$ holds for all $x,y\in X_{t^{-1}}$ and $t\in G$ because $(\delta_t(\braket{x}{y}_N)-\braket{\sigma_t(x)}{\sigma_t(y)}_N)b=0$ for all $b\in N_t.$
Using (4) in Defintion~\ref{def:Wstar equivalence}, notice that $\sigma_t(x)\delta_t(b)=\sigma_t(xb)$ for all $x\in X_{t^{-1}}$ and $b\in N_{t^{-1}}$ of the form $b=\braket{y}{z}_N$, and $t\in G$.
Then the $\weaks$-density of $\spn \braket{X_{t^{-1}}}{X_{t^{-1}}}$ in $N_{t^{-1}}$ and the separate $\weaks-$continuity imply that the same holds for $b\in N_{t^{-1}}$.

 Each $\gamma_t$ is an isometry because for all $x\in X_{t^{-1}}$ we have
\begin{equation*}
 \|\sigma_t(x)\|=\|\braket{\sigma_t(x)}{\sigma_t(x)}_N\|^{1/2}=\|\delta_t(\braket{x}{x}_N)\|^{1/2}=\|\braket{x}{x}_N\|^{1/2}=\|x\|.
\end{equation*}
 Hence, the (unique) predual of $X_t$ is also the (unique) predual of $X_{t^{-1}}$ and $\gamma_t$ must be a $\weaks-$homeomorphism.
 \end{proof}
\end{proposition}

Hilbert W*-modules are defined in \cite{Blecher-Merdy:Operator} as self-dual Hilbert modules over W*\nb-al\-ge\-bras, this approach is equivalent to ours (via W*-trings) by \cite{Blecher-Merdy:Operator}*{Lemma 8.5.4}.
The advantage of W*-modules is that the $\weaks$-topology of the module is (implicitly) determined by that of the algebra.
More precisely, if $X$ is a (right) Hilbert W*-module over $M,$ then on bounded sets of $X$ the $\weaks$-topology agrees with the locally convex topology determined by the functionals of the form $X\to \C,\ x\mapsto \varphi(\braket{y}{x}_M),$ with $y\in X$ and $\varphi\in M_*.$ 
This property is used by Zettl to represent W*-trings as we explain below.

 Suppose $X$ is a (right) Hilbert W*-module over $N,$ with inner products spanning a $\weaks-$dense subset of $N.$
Take a unital faithful W*-representation $N\subseteq \Lb(H)$ and set $K:=X\otimes_N H.$
Then $\pi\colon X\to \Lb(H,K),\ \pi(x)h=x\otimes_N h ,$ is an isometric linear map. Moreover, when restricted to the closed unit ball of $X,$ $\pi$ is $\weaks-\wot$ continuous ($\wot$ being the weak operator topology).
This construction guarantees the existence of a unique normal faithful representation $\pi^r\colon N\to \Lb(H)$ such that $\pi^r(\braket{x}{y}_N)=\pi(x)^*\pi(y).$
In other words, $\pi^r$ is just the inclusion $N\subseteq \Lb(H)$ and $\braket{x}{y}_N=\pi(x)^*\pi(y).$
In case $X$ is a W*-equivalence bimodule between $M$ and $N,$ there exists a faithful unital and normal representation $\pi^l\colon M\to \Lb(K)$ such that $\pi^l({}_M\braket{x}{y})=\pi(x)\pi(y)^*.$
Notice that this last claim is a consequence of Proposition~\ref{prop:fact on W-Hilbert modules}.

\begin{remark}\label{rem:isomorphism between the centres}
 Every equivalence bimodule ${}_AX_B$ induces a homeomorphism between the primitive ideals spaces of $A$ and $B$ \cite{raeburn1998morita}*{Corollary 3.33} (see also \cite{raeburn1998morita}*{Proposition 5.7}).
 The W*-counterpart of this is \cite{RieffelMoritaEquivalenceCandW}*{Proposition 8.1}, which we state as follows: given a W*-equivalence bimodule ${}_MX_N$ there exists a unique W*-isomorphism $\pi\colon Z(N)\to Z(M)$ such that $xa=\pi(a)x$ for all $a\in Z(N)$ and $x\in X.$
\end{remark}

\begin{proposition}\label{pro:Morita=same-center}
 Let $\mu$ and $\nu$ be W*-Morita equivalent W*-partial actions of a group $G$ on the algebras $M$ and $N$, respectively.
 Then the restrictions $\sigma:=\mu|_{Z(M)}$ and $\tau:=\nu|_{Z(N)}$ are isomorphic (as W*-partial actions).
\end{proposition}
\begin{proof}
 Let ${}_MX_N$ be a W*-Morita equivalence bimodule equipped with a W*-partial action $\gamma$ of $G$ inducing an equivalence between W*-partial actions $\mu$ and $\nu$ (on $M$ and $N,$ respectively).
It suffices to check the W*-isomorphism $\pi\colon Z(N)\to Z(M)$ of Remark~\ref{rem:isomorphism between the centres} intertwines the partial actions $\sigma$ and $\tau$.
 First of all, if $p_t$ and $q_t$ are the units of $M_t$ and $N_t$ (respectively), then for all $x\in X$: $p_tx = (p_tx)q_t=p_t(xq_t)=xq_t$.
 Hence $\pi(q_t)=p_t$ and $\pi(Z(N)_t)=Z(M)_t$.

 Now fix $t\in G$ and $a\in Z(N)_{t^{-1}}$.
 For every $x\in X$ we have
 \begin{align*}
  \pi(\nu_t(a))x  & = x \nu_t(a) = xq_t\nu_t(a) = \gamma_t(\gamma_{t^{-1}}(xq_t)a )\\
   &=\gamma_t( \pi(a) \gamma_{t^{-1}}(p_tx)) = \mu_t(\pi(a))p_tx = \mu_t(\pi(a))x.
 \end{align*}
This implies $\pi(\mu_t(a))=\nu_t(\pi(a))$ and the proof is complete.
\end{proof}

As a consequence we derive the following important result:

\begin{proposition}\label{prop:cornerstone}
 AD-amenability is preserved by Morita equivalence of partial actions, both in $C^*$- and W*-contexts.
\end{proposition}
\begin{proof}
 By Remark~\ref{rem:bidual of pa on module}, it suffices to consider the W*-case; which follows directly as a combination of Propositions~\ref{pro:Morita=same-center} and~\ref{prop:AD-amenable-center}.
\end{proof}

The above result applies in particular to global actions and shows that AD\nb-amena\-bi\-li\-ty is preserved by Morita equivalence of group actions. We believe that this is known for specialists but we could not find a reference.

\begin{corollary}\label{cor:ADamenability and enveloping}
 A $C^*$-partial action $\alpha$ is AD-amenable if and only if one (hence all) of its Morita enveloping actions is AD-amenable.
\end{corollary}
\begin{proof}
 Let $\alpha$ be a $C^*$-partial action of a group and let $\beta$ be one of its Morita enveloping actions.
 This means that $\alpha$ is Morita equivalent to a restriction $\gamma$ of $\beta$ and $\beta$ is the $C^*$-enveloping action of $\gamma$.
 By Propositions~\ref{prop:cornerstone} and~\ref{prop:enveloping action of AD amenable is AD amenable}, $\alpha$ is AD-amenable if and only if $\gamma$ is AD-amenable if and only if $\beta$ is AD-amenable.
\end{proof}

\section{AD-amenability of Fell bundles}\label{sec:ADA-Fell-bundles}
One of our goals in this paper is to extend Anantharaman-Delaroche's notion of amenability to Fell bundles over discrete groups.
For this we need some preparation because, as the case of partial actions already indicates, the definition of AD-amenability requires going to the W*-setting.

 When we say that $\cB=\{B_t\}_{t\in G}$ is a Fell bundle we mean 
 $\cB$ is a $C^*$-algebraic bundle in the sense of \cite{Doran-Fell:Representations}; meaning that $G$ is the base group and the fibre over $t\in G$ is the Banach space $B_t.$
The multiplication and the involution of $\cB,$ $(a,b)\mapsto ab$ and $a\mapsto a^*,$ are compatible with the multiplication and the involution of $G$ (respectively) in the sense that $B_sB_t\subseteq B_{st}$ and $B_t^*=B_{t^{-1}}.$
The identities $\|ab\|\leq \|a\|\|b\|$ and $\|a^*a\|=\|a\|^2$ hold for all $a,b\in \cB$ and this makes the fibre $B_e$ (over the unit $e\in G$) a $C^*$-algebra.
We denote $B_sB_t$ the closed linear span of $\{ab\colon a\in B_s,b\in B_t\}$ and similarly for $B_sB_t^*.$

The $L^1-$cross-sectional algebra $\ell^1(\cB)$ is formed by all the cross-sections $f\colon G\to \cB$ such that $\sum_{t\in G}\|f(t)\|<\infty;$ the product being given by $f*g(t)=\sum_{s\in G}f(s)g(s^{-1}t)$ and the involution by $f^*(t)=f(t^{-1})^*.$
$L^1-$cross-sectional algebras of Fell bundles have injective *-representations as bounded operators in Hilbert spaces \cite{Doran-Fell:Representations}*{VIII 16.4}.
Consequently, we may view $\ell^1(\cB)$ as a dense *-subalgebra of its enveloping algebra $C^*(\cB),$ which is called the cross-sectional $C^*$-algebra of $\cB.$
We view each $b\in B_t$ as a section taking the value $b$ at $t$ and vanishing elsewhere.
Thus we get inclusions $\cB\subseteq \ell^1(\cB)\subseteq C^*(\cB)$ that determine the *-algebraic operations of $\cB$ and $\ell^1(\cB).$

 From now on, by a representation of a Banach *-algebra or a Fell bundle we mean a *-representation in the sense of \cite{Doran-Fell:Representations}. Given a representation $\pi\colon C^*(\cB)\to \Lb(X),$ the restriction $\pi|_{\cB}$ is a repre\-sen\-tation of $\cB$ and every representation of $\cB$ arises in this way because we are working with discrete groups and, in this case, the ``recovery'' process of \cite{Doran-Fell:Representations}*{VIII 13} is given by restriction $\pi\mapsto \pi|_\cB.$
Fell's integration process (which in the discrete case only involves sums) is the inverse of restriction.
Notice that the inclusion $\cB\subseteq C^*(\cB)$ determines the norm of $\cB$ because there exist representations of $\cB$ which are isometric on each fibre \cite{Doran-Fell:Representations}*{VIII 16.10}.

We recall from the introduction that semidirect product bundles of (twisted) $C^*$-partial actions are Fell bundles \cite{Exel:TwistedPartialActions}. The semidirect product bundle of a $C^*$- or W*-partial action $\alpha=\{\alpha_t\colon A_{t^{-1}}\to A_t\}_{t\in G}$ will be denoted $\cB_\alpha=\{A_t\delta_t\}_{t\in G},$ with product and involution given by $(a\delta_s) (b\delta_t)=\alpha_s(\alpha_{s^{-1}}(a)b)\delta_{st}$ and $(a\delta_s)^*=\alpha_{s^{-1}}(a)^*\delta_{s^{-1}}.$ The norm is $\|a\delta_s\|=\|a\|$ and the fibre $A_s\delta_s$ is isometrically isomorphic to $A_s$ (as a Banach space, and as a $C^*$-algebra for $s=e$).

The $C^*$-algebra of kernels of a Fell bundle and the canonical action on it (both defined and studied in  \cite{Abadie:Enveloping}) will play a central r\^{o}le in the rest of the article; so it is convenient to recall their definition with some detail. A kernel of the Fell bundle $\cB=\{B_t\}_{t\in G}$ is a function $k\colon G\times G\to \cB$ such that $k(s,t)\in B_{st^{-1}}$ for all $s,t\in G.$ The set $\bk_c(\cB)$ of kernels of finite support is a normed *-algebra when equipped with the product $h*k(r,s)=\sum_{t\in G} k(r,t)h(t,s); $ the involution $k^*(r,s)=k(s,r)^*$ and the norm $\|k\|_2:=(\sum_{s,t\in G}\|k(s,t)\|^2)^{1/2}.$ The $C^*$-algebra of kernels $\bk(\cB)$ is, by definition, the enveloping $C^*$-algebra of the Banach *-algebra obtained by the completion of $(\bk_c(\cB),\|\ \|_2).$ The canonical action $\beta$ of $G$ on $\bk(\cB)$ is the unique $C^*$-action of $G$ on $\bk(\cB)$ such that $\beta_t(k)(r,s)=k(rt,st)$ for all $r,s,t\in G$ and $k\in \bk_c(\cB).$

The right $B_e-$Hilbert module $\ell^2(\cB)$ of \cite{Exel:Amenability} (or, more generally, of \cite{ExelNg:ApproximationProperty}) is the completion of the finitely supported sections of $\cB,$ $C_c(\cB),$ with the inner product $\braket{f}{g}:=\sum_{t\in G}f(t)^*g(t)$ and the action $(fa)(t)=f(t)a.$ Alternatively, $\ell^2(\cB)$ is the direct sum of the fibres $\{B_t\}_{t\in G}$ considered as right $B_e-$Hilbert modules. The (left) regular representation $\Lambda\colon \cB\to \Lb(\ell^2(\cB))$ is determined by the condition $(\Lambda_b g)(s)= bg(t^{-1}s)$ for all $s,t\in G,$ $b\in B_t$ and $g\in C_c(\cB).$ The unique *-homomorphism $\widetilde{\Lambda}\colon C^*(\cB)\to \Lb(\ell^2(\cB))$ such that $\widetilde{\Lambda}|_\cB=\Lambda$ is determined by the condition $\widetilde{\Lambda}_fg=f*g$ (holding for all $f,g\in C_c(\cB)$). The reduced $C^*$-algebra of $\cB$ is $C^*_\red(\cB):=\widetilde{\Lambda}(C^*(\cB))\subseteq \Lb(\ell^2(\cB))$ and $\widetilde{\Lambda}$ is the regular representation of $C^*(\cB).$

In \cite{Abadie-Buss-Ferraro:Morita_Fell} the authors developed the notions of weak and strong Morita equivalence for Fell bundles and characterised them using the canonical actions on the algebras of kernels. We will use this second form of the equivalence, which is as follows. Two Fell bundles, $\cA$ and $\cB,$ over the same group are weakly equivalent if and only if the respective canonical actions ($\alpha$ and $\beta$) on the $C^*$-algebras of kernels ($\bk(\cA)$ and $\bk(\cB)$) are Morita equivalent $C^*$-actions. Recall from \cite{Abadie:Enveloping} that the $C^*$-algebra of generalised compact operators $\bK(\cA):=\bK(\ell^2(\cA))$ is a $C^*$-ideal of $\bk(\cA)$ and that $\alpha$ is (canonically) the enveloping action of $\alpha|_{\bK(\cA)}.$
Then $\cA$ and $\cB$ are strongly equivalent if and only if $\alpha|_{\bK(\cA)}$ is Morita equivalent to $\beta|_{\bK(\cB)}.$

\subsection{W*-enveloping Fell bundles}
 We want to construct a W*-enveloping bundle of a Fell bundle pretty much in the same way one constructs enveloping W*-algebras (biduals) of $C^*$-algebras. Biduals of Fell bundles over inverse semigroups are already described in \cite{Buss-Exel-Meyer:Reduced}*{Section~3}, although there only saturated Fell bundles are considered.
For the convenience of the reader and to make this article as self-contained as possible, we provide the complete construction here.

 Let $\cB=\{B_t\}_{t\in G}$ be a Fell bundle. The $\weaks-$completion of each fibre $B_t$ in the W*-enveloping algebra $C^*(\cB)''$ is isometricaly isomorphic to the bidual $B_t''.$
Thus we may use the W*-structure of $C^*(\cB)''$ to make the bundle of biduals $\cB'':=\{B_t''\}_{t\in G}$ into a Fell bundle in such a way that $\cB$ is a Fell subbundle of $\cB''.$ This trick works because one may view $B_t$ as a $B_tB_t^*-B_t^*B_t-$equivalence module and consider its linking algebra as a $C^*$-subalgebra of the $2\times 2$ matrices with entries in $C^*(\cB).$ We call $\cB''=\{B_t''\}_{t\in G}$ the bidual of $\cB$ and it is a W*-Fell bundle in the following sense.

\begin{definition}
 A \emph{W*-Fell bundle} (or \emph{W*-algebraic bundle}) over the group $G$ is a Fell bundle $\M=\{M_t\}_{t\in G}$ such that each $M_t$ is isometrically isomorphic to the dual of a Banach space and, for every $s,t\in G$ and $a\in M_s$, the functions $M_t\to M_{t^{-1}}, \ b\mapsto b^*$, and $M_t\to M_{st},\ b\mapsto ab$, are $\weaks$-continuous.
\end{definition}

 Two remarks are in order. Firstly, the preduals of the fibres $M_t$ are unique because each $M_t$ is a W*-tring with the operation $(x,y,z):=xy^*z,$ see \cite{Zl83}. Secondly, the $\weaks$-topology of each fibre $M_t$ is determined by that of $M_e$ because $M_t$ is a self-dual $M_e-$Hilbert module.

\begin{remark}\label{remark: bidual partial action and bidual bundle}
 Let $\alpha$ be a partial action of $G$ on a $C^*$-algebra $A$.
 Then $(\cB_\alpha)''$ is canonically W*-isomorphic to $\cB_{\alpha''}$.
\end{remark}

\subsection{Central partial actions of W*-Fell bundles}\label{sec:central-partial}
Take a W*-Fell bundle $\M$ over a group $G$.
For every $t\in G$ we define $I_t$ as the W*-algebra generated by $M_tM_t^*$ in $M_e$.
Notice $I_t$ is in fact a W*-ideal of $M_e$.
The fibre $M_t$ has a natural W*-equivalence $I_t$-$I_{t^{-1}}$-bimodule structure with the multiplication of $\M$ defining the left and right actions and the inner products ${}_{I_t}\langle x,y\rangle:=xy^*$ and $\langle x,y\rangle_{I_{t^{-1}}}:=x^*y$.
The equivalence bimodule $M_t$ then induces an isomorphism $\sigma_t\colon Z(I_{t^{-1}})\to Z(I_t)$.

We claim that $\sigma:=\{\sigma_t\colon Z(I_{t^{-1}})\to Z(I_t)\}_{t\in G}$ is a W*-partial action of $G$ on $Z(M_e)$.
To prove this it suffices to show that $\sigma$ is a set theoretic partial action and when doing this we write $Z_t$ instead of $Z(I_{t})$.

It is clear that $I_e=M_e$. Moreover, $\sigma_e$ is the isomorphism corresponding to the W*-algebra $M_e$ viewed as the identity W*-equivalence $M_e$-$M_e$-bimodule, hence $\sigma_e$  is the identity of $Z_e$.

Let us show that $\sigma_t( Z_{t^{-1}}\cap Z_s )\sbe Z_t\cap Z_{ts}$.
Writing $p_t$ for the unit of $I_t$, it suffices to prove that $\sigma_t(p_{t^{-1}}p_s)=p_tp_{ts}$.
For every $x\in M_t$ we have $ \sigma_t(p_{t^{-1}}p_s)x =x p_{t^{-1}}p_s$, hence $\sigma_t(p_{t^{-1}}p_s)$ is the unit of
\begin{multline*}
 \cspan^\weaks M_t p_{t^{-1}}p_s(M_t p_{t^{-1}}p_s)^*
     =\cspan^\weaks M_t  p_{t^{-1}}p_sM_{t}^*\\
     =\cspan^\weaks M_t M_t^* M_t M_sM_s^* M_t^*
     \sbe \cspan^\weaks  M_t M_t^* M_{ts} M_{ts}^*
     \sbe I_t\cap I_{ts}\\
     \sbe \cspan^\weaks  M_t M_t^* M_{ts} M_{ts}^*  p_t
       = \cspan^\weaks  M_t M_t^* M_{ts} M_{ts}^*  M_tM_{t}^*\\
     \sbe \cspan^\weaks  M_tp_{t^{-1}} M_{s} M_{s}^* M_{t}^*
    \sbe  \cspan^\weaks M_t  p_{t^{-1}}p_sM_{t}^*.
\end{multline*}
Thus  $\sigma_t(p_{t^{-1}}p_s)$ is the unit of $I_t\cap I_{ts}$ and we have $\sigma_t(p_{t^{-1}}p_s)=p_tp_{ts}$.

Now take $x\in Z_{t^{-1}}\cap Z_{t^{-1}s^{-1}}$.
We already know that $\sigma_t(x)\in Z_t\cap Z_{s^{-1}}$ and $\sigma_s(\sigma_t(x))\in Z_{st}\cap Z_s$.
Also $\sigma_{st}(x)\in Z_{st}\cap Z_s$.
We can write $p_s$ as a $\weaks$-limit of the form $p_s =\lim_i \sum_{j=1}^{n_i}  u_{i,j}v_{i,j}^*$ with $u_{i,j},v_{i,j}\in M_s$.
Then, for all $z\in M_{st}$:
\begin{align*}
 \sigma_s(\sigma_t(x))z
  & = \sigma_s(\sigma_t(x))p_sz
    = \sigma_s(\sigma_t(x)) \lim_i\sum_{j=1}^{n_i} u_{i,j}v_{i,j}^*z
    = \lim_i \sum_{j=1}^{n_i}\sigma_s(\sigma_t(x)) u_{i,j}v_{i,j}^*z\\
  & = \lim_i \sum_{j=1}^{n_i} u_{i,j}\sigma_t(x) v_{i,j}^*z
    = \lim_i \sum_{j=1}^{n_i} u_{i,j} v_{i,j}^*z x
    = p_s z x
    =\sigma_{st}(x)p_s z\\
  & =\sigma_{st}(x) z.
\end{align*}
This implies $\sigma_{st}(x)=\sigma_s(\sigma_t(x))$.

\begin{definition}
 Let $\M$ be a W*-Fell bundle over a group $G$.
 The \emph{central partial action} of $\M$ is the W*-partial action $\sigma$ of $G$ on $Z(M_e)$ constructed above.
\end{definition}

\begin{example}\label{exa:central partial action of semidirect product bundle}
 Let $\gamma=(\{M_t\}_{t\in G},\{\gamma_t\}_{t\in G})$ be a W*-partial action of a group $G$ on a W*-algebra $M$.
 If $\M$ is the semidirect product bundle of $\gamma$, which is a W*-Fell bundle, then the central partial action of $\M$ is the restriction of $\gamma$ to $Z(M)$.

 To prove the claim above note that the W*-ideals $I_t$ of $M=M\delta_e$ generated by $(M_t\delta_t)(M_t\delta_t)^*=\gamma_t(\gamma_{t^{-1}}(M_tM_t^*))\delta_e=M_t\delta_e$ are just $M_t$ seen as a subalgebra of $M=M\delta_e$.
 Then the domains of $\sigma$ and $\gamma|_{Z(M)}$ agree.
 If $x\in Z(M_{t^{-1}})$ and $y\in M_t$, then
 \begin{equation*}
  \gamma_t(x)\delta_e y\delta_t
     = \gamma_t(x)y\delta_t
     = y\gamma_t(x)\delta_t
     = \gamma_t(\gamma_{t^{-1}}(y)  x)\delta_t
     = y\delta_t x\delta_e
     = \sigma_t(x)\delta_ey\delta_t
 \end{equation*}
and this implies $\gamma_t(x)=\sigma_t(x)$ (because we identify $x\in M_e$ with $x\delta_e\in M_e\delta_e$).
\end{example}

\subsection{Cross-sectional W*-algebras of W*-Fell bundles}\label{ssec:reduced von Neuman algebra}

To a W*-Fell bundle $\M$ one can naturally assign a cross-sectional W*-algebra $W^*_\red(\M)$ as follows: the usual Hilbert $M_e$-module $\ell^2(\M)$ is not suitable here because it might not be a W*-module, that is, it is possibly not self-dual. We look at its self-dual completion, which can be concretely described as follows. Let $\ell^2_{\weaks}(\M)$ be the space of sections $\xi\colon G\to \M$ for which the net of finite sums $\sum_{t\in F}\xi(t)^*\xi(t)$ with $F\sbe G$ finite is bounded; since this net is increasing and consists of positive elements, it $\weaks$-converges to some element $\braket{\xi}{\xi}_{M_e}:=\sum_{t\in G}\xi(t)^*\xi(t)\in M_e$. The space $\ell^2_\weaks(\M)$ is then a right W*-Hilbert $M_e$-module when endowed with right $M_e$-action $(\xi b)(t):=\xi(t) b$ and inner product $\braket{\xi}{\eta}_{M_e}:=\sum_{t\in G}\xi(t)^*\eta(t)$, the limit of this sum being with respect to the $\weaks$-topology, for all $\xi,\eta\in \ell^2_{\weaks}(\M).$

Next we define the left regular representation of $\M$. This is done as in the $C^*$-case, except that we now act on $\ell^2_{\weaks}(\M)$.
More precisely, for each $t\in G$ we define the map $\Lambda_t\colon
M_t\to \Lb(\ell^2_{\weaks}(\M))$ by
$(\Lambda_t(a)\xi)(s):=a\xi(t^{-1}s)$ (the multiplication
performed in $\M$) for all $s\in G$, $a\in M_t$ and $\xi\in
\ell^2_{\weaks}(\M)$. As in the $C^*$-setting \cite{Exel:Partial_dynamical}*{Section 17}, a routine argument
shows that $\Lambda_t(a)$ is a well-defined adjointable operator with
$\Lambda_t(a)^*=\Lambda_{t^{-1}}(a^*)$ and that
$\Lambda=\{\Lambda_t\}_{t\in G}$ is a representation of $\M;$  which can be extended to $\ell^1(\M)$ and from there to $C^*(\M)$ as an integrated form $\tilde{\Lambda}.$
Recall that $\Lb(\ell^2_{\weaks}(\M))$ is a $W^*$-algebra (see for example \cite{MR0355613}*{Proposition~3.10}).

\begin{definition}
The cross-sectional W*-algebra of $\M$ is the W*-subalgebra
$W^*_\red(\M)$ of $\Lb(\ell^2_{\weaks}(\M))$ generated by the image of
its regular representation $\Lambda$.
\end{definition}

The (reduced) cross-sectional W*-algebra of a W*-Fell bundle $\M$ is exactly the W*-counterpart of the reduced cross-sectional $C^*$-algebra of a Fell bundle.

The linear span of the image of $\Lambda$ is already a \Star{}subalgebra, so that $W^*_\red(\M)$ is just the $\weaks$-closure of that subalgebra. We also observe that the cross-sectional $C^*$-algebra $C^*_\red(\M)$ embeds as a $\weaks$-dense $C^*$-subalgebra of $W^*_\red(\M)$. Moreover, since $\ell^2_{\weaks}(\M)$ is the self-dual completion of $\ell^2(\M)$, every adjointable operator on $\ell^2(\M)$ extends to an adjointable operator on $\ell^2_{\weaks}(\M)$ and this gives a $C^*$-embedding $\Lb(\ell^2(\M))\into \Lb(\ell^2_{\weaks}(\M))$ that restricts to the embedding $C^*_\red(\M)\into W^*_\red(\M)$.
 This inclusion can be also constructed adapting the proof of the proposition below and, additionally, justifies the abuse of notation we committed when using the symbol $\Lambda$ to denote the representations of $\M$ in $\Lb(\ell^2(\M))$ and in $\Lb(\ell^2_\weaks(\M))$.

\begin{proposition}\label{prop:faithful representation of W(B)}
Let $\M$ be a W*-Fell bundle over a group $G$ and $\pi\colon M_e\to \Lb(H)$ be a $\weaks-$continuous representation.
 Then the map
 \begin{equation*}
 \Lambda_\pi \colon \M\to \Lb( \ell^2_\weaks(\M)\otimes_\pi H),\ b\mapsto  \Lambda(b) \otimes \id
 \end{equation*}
 is a representation which is $\weaks$-continuous on each fibre.
 The integrated form $\tilde{\Lambda}_\pi$ factors through a representation of $C^*_{\red}(\M)$ that can be extended to a weak* continuous representation $\tilde{\Lambda}_\pi^{\weaks}$ of $W^*_{\red}(\M)$ in a unique way.
 Moreover, $\tilde{\Lambda}_\pi^{\weaks}$ is unital and $\tilde{\Lambda}_\pi^{\weaks}(W^*_{\red}(\M))= \tilde{\Lambda}_\pi^{\weaks}(\M)''$ (the bicommutant).
 If $\pi$ is injective, then so is $\tilde{\Lambda}_\pi^{\weaks}$  and $\Lambda_\pi$ is isometric on each fibre.
\end{proposition}
\begin{proof}
 Consider the map $\rho\colon \Lb(\ell^2_\weaks(\M))\to \Lb(\ell^2_\weaks(\M)\otimes_\pi H),$ $\rho(R)=R\otimes \id$.
 Then $\Lambda_\pi:= \rho\circ \Lambda \colon \M\to \Lb(\ell^2_\weaks(\M)\otimes_\pi H)$ is clearly a representation that, when restricted to the closed unit ball of a fibre, is $\weaks$-continuous.
 Hence $\Lambda_\pi$ is a representation which is $\weaks$-continuous on each fibre.
 If $\pi$ is faithful, then so is $\rho$ and $\Lambda_\pi$ is isometric on each fibre because $\Lambda$ is.

 Notice that $\ell^2_\weaks(\M)\otimes_\pi H=\ell_2(\M)\otimes_\pi H.$
 Thus we may think of $\Lambda_\pi\colon \M\to \Lb(\ell_2(\M)\otimes_\pi H)$ as the composition of the $C^*$-regular representation $\M\to \Lb(\ell_2(\M))$ with $C^*_{\red}(\M)\subseteq \Lb(\ell_2(\M))\to \Lb(\ell_2(\M)\otimes_\pi H), $ $T\mapsto T\otimes \id.$
 This clearly implies that $\tilde{\Lambda}_\pi$ factors through $C^*_{\red}(\M).$

 The restriction $\rho|_{W^*_{\red}(\M)}\colon W^*_{\red}(\M)\to \Lb(\ell^2_\weaks(\M)\otimes_\pi H)$ is a $\weaks$-continuous representation that clearly extends the integrated form of $\Lambda_\pi.$
 Hence this integrated form can be extended to $W^*_{\red}(\M)$ (as $\rho|_{W^*_{\red}(\M)}$) and the extension is faithful if $\pi$ is.
 The rest of the proof follows by the Bicommutant Theorem.
\end{proof}

\begin{theorem}[c.f. {\cite{ExelNg:ApproximationProperty}*{Corollary 2.15}}]\label{theo:Exel-Ng representations}
Assume that $\M=\{M_t\}_{t\in G}$ is a W*-Fell bundle over $G$ and write $\lambda$ for the left regular representation of $G$ by unitary operators on $\ell_2(G).$
 Let $T \colon \M\to\Lb(H)$ be a nondegenerate representation which is weak* continuous on each fibre and let $\mu_{\lambda,T}\colon \M\to \Lb(\ell_2(G,H))$ be the representation such that $\mu_{\lambda,T}(m)=\lambda_t\otimes T_m,$ for every $m\in M_t$ and $t\in G.$
 Then the integrated form of $\mu_{\lambda,T},$ denoted $\tilde{\mu}_{\lambda,T},$ factors through a representation of $C^*_{\red}(\M)$ that has a unique $\weaks$-continuous extension to a representation $\tilde{\mu}^\weaks_{\lambda,T}$ of $W^*_{\red}(\M).$
 Moreover, if $T|_{M_e}$ is faithful then so is $\tilde{\mu}^\weaks_{\lambda,T}.$
\end{theorem}
\begin{proof}
 It was shown in \cite{ExelNg:ApproximationProperty} that $\tilde{\mu}_{\lambda,T}$ factors through a representation of $C^*_{\red}(\M).$
 To extend $\tilde{\mu}_{\lambda,T}$ to $W^*_{\red}(\M)$ take $R\in W^*_\red(\M).$
 Then, by \cite{MR641217}, there exists a bounded net $\{f_j\}_{j\in J}\subseteq C^*(\M)$ such that $R = \weaks\lim_j \Lambda(f_j).$
 Notice that $\{\Lambda(f_j)\}_{j\in J}\subseteq C^*_\red(\cB).$

 Every closed ball of $\Lb(\ell_2(G,H))$ is compact in the weak* topology, thus there exists $S\in \Lb(\ell_2(G,H))$ and a subnet $\{f_{j_l}\}_{l\in L}$ such that $S=\weaks \lim_l \tilde{\mu}_{\lambda,T}(f_{j_l}).$
 In order to prove that $S=\weaks \lim_j \tilde{\mu}_{\lambda,T}(f_j)$ it suffices to show the existence of a set $X\subseteq \ell_2(G,H)$ spanning a dense subset of $\ell_2(G,H)$ and such that, for every $x,y\in X,$ $\{\langle x,\tilde{\mu}_{\lambda,T}(f_j)y\rangle\}_{j\in J}$ is a convergent net.

 Using the notation of \cite{ExelNg:ApproximationProperty}*{Propositon 2.13} we define
 \begin{equation*}
   X:=\bigcup_{r\in G} \{ (\rho_r\otimes 1)\circ V u \colon u\in \ell_2(\M)\otimes_{T} H, h\in H\},
 \end{equation*}
 recalling that $\rho\colon G\to \ell_2(G)$ is the right regular representation and that $\rho_r\otimes 1$ lies in the commutant of $\tilde{\mu}_{\lambda,T}(C^*_{\red}(\M)).$
 Recall also that $V\colon \ell_2(\M)\otimes_{T} H\to \ell_2(G,H)$ is an isometry such that $V(z\otimes h)(t)=T_{z(t)}h.$
 Take $x=\rho_r\otimes 1\circ V u\in X$ and $y=\rho_r\otimes 1\circ Vv\in X.$
 Then, by \cite{ExelNg:ApproximationProperty}*{Proposition 2.6},
 \begin{align*}
  \lim_j \langle x,\tilde{\mu}_{\lambda,T}(f_j)y\rangle
     &= \lim_j\langle V u,  (\rho_{r^{-1}}\otimes 1)\tilde{\mu}_{\lambda,T}(f_j) (\rho_s\otimes 1)Vv\rangle\\
     &= \lim_j\langle V u, (\rho_{r^{-1}s}\otimes 1)V(\Lambda(f_j)\otimes 1)v\rangle\\
     & =\langle Vu, (\rho_{r^{-1}s}\otimes 1)V (R\otimes 1)v\rangle.
 \end{align*}

 This not only shows that $\{\tilde{\mu}_{\lambda,T}(f_{j})\}_{j\in J}$ converges in the weak (and weak*) topology, but also that its limit is completely determined by $R=\weaks\lim_j \Lambda(f_j).$
 Of course, we define $\tilde{\mu}^\weaks_{\lambda,T}(R):=S.$

 Define $V^r:=(\rho_r\otimes 1)\circ V.$
 Then $\tilde{\mu}^\weaks_{\lambda,T}\colon W^*_{\red}(\M)\to \Lb(\ell_2(G,H))$ is uniquely determined by the condition
 \begin{equation}\label{equ:condition defining the extension}
  \langle V^r u,\tilde{\mu}_{\lambda,T}(R) V^s v\rangle = \langle V^{s^{-1}r}u, V (R\otimes 1)v\rangle,\ \forall \ u,v\in \ell_2(\M), \ r,s\in G.
 \end{equation}
 This condition immediately implies that $\tilde{\mu}^\weaks_{\lambda,T}$ is linear and $\weaks$-continuous in any closed ball.
 Moreover, it is also straightforward to prove that $\tilde{\mu}^\weaks_{\lambda,T}$ preserves the involution.
 To show that $\tilde{\mu}^\weaks_{\lambda,T}$ is multiplicative take $R,S\in W^*_{\red}(\M)$ and bounded nets $\{f_j\}_{j\in J},\{g_l\}_{l\in L}\subseteq C^*_{\red}(\M)$ weak* converging to $R$ and $S,$ respectively.
 Then, using that multiplication is separately weakly continuous, we deduce
 \begin{align*}
  \tilde{\mu}^\weaks_{\lambda,T}(RS)
    & =\lim_j \tilde{\mu}^\weaks_{\lambda,T}(f_jS)
      =\lim_j \lim_l \tilde{\mu}^\weaks_{\lambda,T}(f_jg_l)
      =\lim_j \lim_l \tilde{\mu}^\weaks_{\lambda,T}(f_j)\tilde{\mu}^\weaks_{\lambda,T}(g_l)\\
    & = \lim_j \tilde{\mu}^\weaks_{\lambda,T}(f_j)\tilde{\mu}^\weaks_{\lambda,T}(S)
      = \tilde{\mu}^\weaks_{\lambda,T}(R)\tilde{\mu}^\weaks_{\lambda,T}(S).
 \end{align*}

 Assume $T|_{M_e}$ is faithful and $\tilde{\mu}^\weaks_{\lambda,T}(R)=0.$
 Then \eqref{equ:condition defining the extension} implies (with $r=s=e$) that $\langle u,(R\otimes 1)v\rangle=0$ for all $u,v\in \ell^2_\weaks(\M)\otimes_{T|_{M_e}} H.$
 Since  $T|_{M_e}$ is faithful, we have $R=0.$
\end{proof}

\begin{definition}
 Let $\M$ be a W*-Fell bundle over the discrete group $G.$
 We say that the subset $\mathcal{N}\subseteq \M$ is a W*-Fell subbundle if it is a Fell subbundle and the $\weaks$-topology of each fibre $N_t$ is the restriction of the $\weaks$-topology of $M_t.$
\end{definition}

 The definition above implies that $\mathcal{N}$ is a W*-Fell bundle on its own right. Reciprocally, if a Fell subbundle $\cB$ of $\M$ is a W*-Fell bundle (with some appropriate $\weaks$-topology) then $\cB$ is a W*-Fell subbundle of $\M$ if and only if $B_e$ is $\weaks-$closed in $M_e.$ This is so because the $\weaks$-topology of $B_e$ determines the $\weaks$-topology of all the fibres $B_t.$

\begin{proposition}
Let $\mathcal{N}\subseteq \M$ be a W*-Fell subbundle.
 If we view $C^*_{\red}(\mathcal{N})$ as a $C^*$-subalgebra of $C^*_{\red}(\M)$ (as in \cite{Abadie:Enveloping}*{Proposition 3.2}), and if $C^*_{\red}(\M)$ is viewed as a $C^*$-subalgebra of $W^*_{\red}(\M)$, then $W^*_{\red}(\mathcal{N})$ is isomorphic to the $\weaks$-closure of $C^*_{\red}(\mathcal{N})$ in $W^*_{\red}(\M).$
\end{proposition}
\begin{proof}
  Our proof is a slight modification of that of \cite{Abadie:Enveloping}*{Proposition 3.2}.
  By Proposition \ref{prop:faithful representation of W(B)} there
  exists a representation $T\colon \M\to \Lb(H)$ with $T|_{M_e}$
  faithful and $\weaks$-continuous on each fibre. 
  Define $H_0:=T_{1_N}H,$ where $1_N$ is the unit of $N_e,$ and the restriction map $R\colon \mathcal{N}\to \Lb(H_0)$ by $R_a:=T_a|_{H_0}.$
  Then $R$ is a representation $\weaks$-continuous on each fibre and with $R|_{N_e}$ faithful.

  In terms of the decomposition $\ell_2(G,H)=\ell_2(G,H_0)\oplus \ell_2(G,H_0)^\perp,$ we have
  \begin{equation}\label{equ: mu lambda T and mu lambda R}
   \mu_{\lambda,T}(a)=\left(\begin{array}{cc}
                             \mu_{\lambda,R}(a) & 0\\
                             0  & 0
                            \end{array}\right),\ \forall a\in \mathcal{N}.
  \end{equation}

  We get the desired result by considering the integrated forms of $\mu_{\lambda,T}$ and $\mu_{\lambda,R}$ and the respective $\weaks-$continuous extensions to $W^*_{\red}(\M)$ and $W^*_{\red}(\mathcal{N}).$
\end{proof}

\begin{remark}
 If we add to the hypotheses of the last theorem the condition that $\mathcal{N}$ is hereditary in $\M$ (that is $\mathcal{N}\M\mathcal{N}\subseteq \mathcal{N}$) then $W^*_{\red}(\mathcal{N})$ is hereditary in $W^*_{\red}(\M).$
 Indeed, this follows from separate $\weaks$-continuity of the product and the fact that $C^*_{\red}(\mathcal{N})$ is hereditary in $C^*_{\red}(\M).$
\end{remark}

\subsection{The W*-algebra of kernels}\label{ssec:wstar algebra of kernels}
Here we define a W*-version of the $C^*$-algebra of kernels $\bk(\M)$ and also of the canonical action $\beta$ on $\bk(\M).$

Consider the canonical representation $\pi\colon \bk(\M)\to \Lb(\ell^2(\M))$ given by $\pi(k)f(s)=\sum_{t\in G}k(s,t)f(t)$ for every $k\in \bk_c(\M)$ and $f\in C_c(\cB)$ with finite support. This representation has been already considered in \cite{Abadie:Enveloping}.
Using the canonical embedding $\Lb(\ell^2(\M))\into \Lb(\ell^2_\weaks(\M))$, 
we view $\pi$ as a representation $\pi\colon \bk(\M)\to \Lb(\ell^2_\weaks(\M))$.

With the canonical action of $G$ on $\bk(\M)$ we construct the *-homomorphism $\pi^\beta\colon \bk(\M)\to \ell^\infty(G,\Lb(\ell^2_\weaks(\M)))$ defined by $\pi^\beta(f)|_t=\pi(\beta_t^{-1}(f))$.
Notice that $\pi^\beta$ is equivariant with respect to the translation W*-action $\gamma$ on $\ell^\infty(G,\Lb(\ell^2_\weaks(\M)))$.

Recall that we may view the algebra of (generalised) compact operators $\bK(\M):=\bK(\ell^2(\M))$ as an ideal of $\bk(\M)$ and $\beta$ is the enveloping action of $\beta|_{\bK(\M)}$.
In the $C^*$-case we know that $\pi\colon \bk(\M)\to \Lb(\ell^2(\M))$ is the identity when restricted to $\bK(\M)$. In particular $\pi\colon \bk(\M)\to \Lb(\ell^2_\weaks(\M))$ is injective on $\bK(\M)$. 

Now we can prove that $\pi^\beta$ is injective.
Indeed,  $\pi^\beta(f)=0$ implies that $\pi(\beta_t(f)x)=0$ for every $t\in G$ and $x\in \bK(\M)$ and since $\pi$ is faithful on $\bK(\M)$, this implies $f\beta_t(x)=0$ for every $t\in G$ and $x\in\bK(\M)$, and this is equivalent to $f=0$ because the linear $G$-orbit of $\bK(\M)$ is dense in $\bk(\M)$.

Let $\bk_\weaks(\M)$ and $\bK_\weaks(\M)$ be the
$\weaks$-closures of $\pi^\beta(\bk(\M))$ and $\pi^\beta(\bK(\M))$, respectively.
Then clearly $\bK_\weaks(\M)$ is a W*-ideal of $\bk_\weaks(\M)$ and $\beta^\weaks:=\gamma|_{\bk_\weaks(\M)}$ is the W*-enveloping action of $\beta^\weaks|_{\bK_\weaks(\M)}=\gamma|_{\bK_\weaks(\M)}$.

Our construction implies that $\beta^\weaks$ is a quotient of $\beta''$.
This quotient is such that we can faithfully view $\beta$ as a restriction of $\beta^\weaks$.
Notice that $\bK(\ell^2_\weaks(\M))$ is $\weaks$-dense in $N:=\Lb(\ell^2_\weaks(\M))$ (this follows, for instance, from \cite{Blecher-Merdy:Operator}*{Lemma~8.5.23}).
We claim that $\bK_\weaks(\M)$ is canonically isomorphic to $N$. Indeed, the evaluation at $e\in G$, $ev_e\colon \ell^\infty(G,N)\to N,$ is a surjective $\weaks$-continuous *-homomorphism. Moreover, $ev_e$ is injective when restricted to $\bK_\weaks(\M)$ because $ev_e\circ\pi^\beta|_{\bK_\weaks(\M)}$ is just $\pi|_{\bK_\weaks(\M)}$.
Thus $ev_e|_{\bK_\weaks(\M)}$ is an isomorphism between $\bK_\weaks(\M)$ and $N=\Lb(\ell^2_\weaks(\M))$.

\begin{definition}
 The W*-algebra $\bk_\weaks(\M)$ constructed above will be called the \emph{W*-algebra of kernels of
 $\M$}. It will be always endowed with the canonical W*-action $\beta^\weaks$ of $G$ defined above.
\end{definition}

\begin{definition}\label{def:equivalence of wstar Fell bundles}
 We say that two W*-Fell bundles are \emph{weakly W*-equivalent} if the canonical actions on their W*-algebras of kernels are W*-Morita equivalent.
\end{definition}

\begin{remark}
 W*-equivalence of W*-Fell bundles is an equivalence relation because, as in the $C^*$-case, we have inner tensor products of W*-equivalence bimodules.
\end{remark}

\begin{theorem}\label{thm:enveloping action of the central partial action}
  Let $\M$ be a W*-Fell bundle over a group $G$. Then the W*\nb-envelo\-ping action of the central partial action $\sigma$ of $\M$ is the restriction of $\beta^\weaks$ to the centre of $\bk_\weaks(\M)$.
\end{theorem}
\begin{proof}
By Proposition~\ref{pro:center-enveloping}, $\beta^\weaks|_{Z(\bk_\weaks(\M))}$ is the W*-enveloping action of $\tau:=\beta^\weaks|_{Z(\bK_\weaks(\M))}$.
  Hence all we need is to show that $\tau$ is isomorphic to $\sigma$.

  The module $\ell^2_\weaks(\M)$ is a W*-equivalence bimodule between $\bK_\weaks(\M)$ and $M_e$, hence it induces a W*-isomorphism $\mu\colon Z(M_e)\to Z(\bK_\weaks(\M))$ which we claim is an isomorphism of W*-partial actions between $\sigma$ and $\tau$.

  To simplify our notation we write $ZM$ and $Z\M$ instead of $Z(M_e)$ and $Z(\bK_\weaks(\M))$, respectively.
  Consequently, the domains of $\sigma$ and $\tau$ will be denoted $ZM_t$ and $Z\M_t$ for $t\in G$.

  We must show that $\mu(ZM_t)=Z\M_t$ or, equivalently, that $\ell^2_\weaks(\M)$ induces the ideal $I_t=\cspn^\weaks M_tM_t^*$ to $J_t:=\bK_\weaks(\M)\cap \beta^\weaks_t(\bK_\weaks(\M))$.
  From the proof of \cite{Abadie-Buss-Ferraro:Morita_Fell}*{Theorem~3.5} we know that $\ell^2(\M)\sbe \ell^2_\weaks(\M)$ induces $\cspn^{\|\|}M_tM_t^*$ to $\bK(\M)\cap \beta_t(\bK(\M))$.
  By taking $\weaks$-closures in $M_e$ and $\ell^\infty(G,\Lb(\ell^2_\weaks(\M)))$, respectively, we get the desired induction.

  The composition $\mu\circ \sigma_t$ equals the composition $\mu|_t\circ \sigma_t$, where $\mu|_t$ represents the restriction and co-restriction of $\mu$ to $ZM_t$ (in the domain) and $Z\M_t$ (in the co-domain).
  But $\mu|_t$ is the isomorphism corresponding to the bimodule
  \begin{equation*}
  X_t:=J_t\ell^2_\weaks(\M)I_t=\ell^2_\weaks(\M)I_t=J_t\ell^2_\weaks(\M).
  \end{equation*}

  Hence we may view $\mu\circ \sigma_t$ as the isomorphism corresponding to the bimodule
  $X_t\otimes^\weaks_{I_t} M_t $.
  In the same way we may view $\tau_t\circ \mu$ as the isomorphism corresponding to the bimodule $J_t\delta_t\otimes^\weaks_{J_{t^{-1}}\delta_e} X_{t^{-1}}$, where $J_t\delta_t$ is the fibre over $t$ of the semidirect product bundle of $\beta^\weaks|_{\bK_\weaks(\M)}$, $\cB^\weaks$, and $J_{t^{-1}}\delta_e$ is the ideal $J_{t^{-1}}$ seen as an ideal of the unit fibre of that bundle.
  Once again we will make use of the $C^*$-version of all these constructions.

  The semidirect product bundle of $\beta|_{\bK(\ell^2(\M))}$ will be denoted $\cB$, and we will think of it as a Fell subbundle of $\cB''$.
  The fibre over $t$ of $\cB$ is $\bK(\M)_t\delta_t$ and $\bK(\M)_t\delta_t(\bK(\M)_t\delta_t)^*=\bK(\M)_t\delta_e\sbe \bK(\M)\delta_e$

  Define $I_t^{\|\|}$ and $J_t^{\|\|}$ as the $C^*$-algebras generated by $M_tM_t^*$ and
  \begin{equation*}
    (\bK(\M)_t\delta_t)(\bK(\M)_t\delta_t)^*
  \end{equation*}
  in $M_e$ and $\bK(\M)\delta_e$, respectively.
  If we set
  \begin{equation*}
   X_t^{\|\|} :=J^{\|\|}_t\ell^2(\M)I^{\|\|}_t = \ell^2(\M)I^{\|\|}_t = J_t^{\|\|}\ell^2(\M),
  \end{equation*}
  then $X_t^{\|\|}\otimes_{I_t^{\|\|}} M_t$ and $\bK(\M)_t\delta_t\otimes_{J^{\|\|}_{t^{-1}}}X_{t^{-1}}$ are isomorphic as $C^*$-trings.
  To prove this claim consider the canonical $L^2$-bundle of $\M$, $\Lb\M=\{ L_t \}_{t\in G}$, which establishes a strong equivalence between $\cB$ and $\M$ \cite{Abadie-Buss-Ferraro:Morita_Fell}.
  Then $X_t^{\|\|}$ is exactly $J_t^{\|\|}L_e=L_eI_t^{\|\|}=J_t^{\|\|}L_eI_t^{\|\|}$, and we have canonical injective maps
  \begin{align*}
   \nu_1 &\colon X_t^{\|\|}\otimes_{I_t^{\|\|}} M_t\to L_t,\ x\otimes y\mapsto xy,\\
   \nu_2&\colon \bK(\M)_t\delta_t\otimes_{J^{\|\|}_{t^{-1}}}X_{t^{-1}} \to L_t, \ T\otimes x\mapsto Tx,
  \end{align*}
  where the actions used are the actions of $\cB$ and $\M$ on $\Lb\M$.
   These maps are homomorphisms of $C^*$-trings in the sense of \cite{Abadie:Enveloping}*{Proposition 4.1}.
  Hence, they define left and right maps $\nu^r_j$ and $\nu^l_j$ for $j=1,2$.
    
  The images of $\nu_1$ and $\nu_2$ are $L_eM_t\sbe L_t$ and $\bK(\M)_t\delta_t L_e\sbe L_t$, respectively, because $M_t=I_t^{\|\|}M_t$ and $\bK(\M)_t=\bK(\M)_tJ^{\|\|}_{t^{-1}}$ (due to Cohen-Hewitt Theorem we do not need closed linear spans here).
  Recalling the definition of strong equivalence and understanding the products below as norm closed linear spans of products, we obtain:
  \begin{align*}
   L_e M_t
     & = L_e M_tM_t^*M_t
       \sbe  L_e \langle L_{t^{-1}},L_{t^{-1}}\rangle_{\M} M_t
       \sbe   {}_{\Lb\cB}\langle L_e ,L_{t^{-1}}\rangle L_{t^{-1}} M_t \\
     & \sbe   {}_{\Lb\cB}\langle L_e ,L_{t^{-1}}\rangle L_e
       \sbe    \bK(\M)_t\delta_t L_e\sbe \ldots \sbe L_e M_t.
  \end{align*}
  It can be directly shown that $\nu^r_j$ and $\nu^l_j$ ($j=1,2$) are the natural inclusions of $I^{\|\|}_t$ and $J^{\|\|}_t$ on $M_e$ and $\bK(\M)\delta_e$.
  This is due to the fact that we are allowed to use the inner products of $\Lb\M$ in the computations of the tensor product.
  Hence
  \begin{equation}\label{equ:key isomorphism, previous}
 \nu_2^{-1}\circ \nu_1\colon   X_t^{\|\|}\otimes_{I_t^{\|\|}} M_t \to \bK(\M)_t\delta_t\otimes_{J^{\|\|}_{t^{-1}}}X_{t^{-1}}
  \end{equation}
  is an isomorphism of $C^*$-trings with $(\nu_2^{-1}\circ \nu_1)^r$ and $(\nu_2^{-1}\circ \nu_1)^l$ being the identities on $I^{\|\|}_{t^{-1}}$ and $J^{\|\|}_{t}$, respectively.

  The question is now if we can extend $\nu_2^{-1}\circ \nu_1$ to an isomorphism
  \begin{equation}\label{equ:key isomorphism}
   \overline{\nu_2^{-1}\circ \nu_1}\colon X_t\otimes^\weaks_{I_t} M_t\to J_t\delta_t\otimes^\weaks_{J_{t^{-1}}\delta_e} X_{t^{-1}}
  \end{equation}
 We give an indication of how to do this and leave the details to de reader. Choose a concrete (faithful, normal) W*-representation  $I_{t^{-1}}\subseteq \Lb(H)$ and represent the $C^*$-\nb-equivalence modules of \eqref{equ:key isomorphism, previous} and the W*-equivalence modules of \eqref{equ:key isomorphism} using the concrete description of $I_t$ as we explained right before Remark~\ref{rem:isomorphism between the centres}.
By thinking in terms of concrete (i.e. ``represented'') modules, the W*-modules are the $\weaks-$closures of the $C^*$-modules and the isomorphism of \eqref{equ:key isomorphism} is the unique $\weaks-$extension of \eqref{equ:key isomorphism, previous}.

  After constructing the isomorphism $\overline{\nu_2^{-1}\circ \nu_1}$ as a $\weaks$-extension of $\nu_2^{-1}\circ \nu_1$ it follows directly that $\overline{\nu_2^{-1}\circ \nu_1}^r$ and $\overline{\nu_2^{-1}\circ \nu_1}^l$ are the identities (on $J_t$ and $I_t$, respectively).
  Now the isomorphism in \eqref{equ:key isomorphism} implies $\mu\circ \sigma_t=\tau_t\circ\mu$.
\end{proof}

\begin{corollary}\label{cor:equivalence of AD amenability}
 For a W*-Fell bundle $\M$ the following are equivalent:
 \begin{enumerate}[(i)]
  \item The canonical action $\beta^\weaks$ on $\bk_\weaks(\M)$  is W*AD-amenable.
  \item The restriction of $\beta^\weaks$ to $Z(\bK_\weaks(\M))$ is W*AD-amenable.
  \item The central partial action of $\M$ is W*AD-amenable.
 \end{enumerate}
\end{corollary}
\begin{proof}
 Follows at once from the definition of W*AD-amenability of partial actions, Proposition~\ref{prop:AD-amenable-center} and Theorems~\ref{thm:enveloping action of the central partial action} and~\ref{the:ADA-characterisations}. 
\end{proof}

\begin{definition}\label{def:AD amenable}
A W*-Fell bundle is said to be \emph{W*AD-amenable} if the equivalent conditions of the corollary above are satisfied.
 A Fell bundle $\cB$ is \emph{AD-amenable} if the enveloping W*-Fell bundle $\cB''$ is W*AD-amenable.
\end{definition}

By \cite{Anantharaman-Delaroche:ActionI}*{Proposition 3.6}, every W*-action of an amenable group is W*AD-amenable.
Hence every (W*-)Fell bundle over an amenable group is (W*)AD-amenable.

Given a W*-Fell bundle $\M=\{M_t\}_{t\in G}$ and a subgroup $H\subset G,$ the restriction $\M_H=\{B_t\}_{t\in H}$ is a W*-Fell bundle. Moreover, the central partial action of $\M_H$ is the restriction to $H$ of the central partial action of $\M$. Hence, Example~\ref{exa:restriction to subgroup} implies $\M_H$ is W*AD-amenable if $\M$ is. Moreover, for every Fell bundle $\cB$ over $G$ we have $(\cB_H)''=(\cB'')_H$ so that $\cB_H$ is AD-amenable if $\cB$ is. Therefore, restriction of AD-amenable Fell bundles to a subgroup remain AD-amenable.

\begin{remark}
W*AD-amenability is preserved by weak equivalence of W*-Fell bundles.
\end{remark}

\begin{remark}\label{rem:AD amenability of partial action and semidirect product bundle}
 Proposition~\ref{prop:AD-amenable-center} and Example~\ref{exa:central partial action of semidirect product bundle} imply that a W*-partial action is W*AD-amenable if and only if its semidirect product bundle (which is a W*-Fell bundle) is W*AD-amenable.
 Hence the same conclusion holds for $C^*$-partial actions and AD-amenability.
\end{remark}

\begin{theorem}\label{the:bidual and wstar algebra of kernels}
 Let $\cB$ be a Fell bundle over a group and let $\cB''$ be the enveloping W*-Fell bundle of $\cB$.
 Then the canonical action $\beta^\weaks$ on $\bk_\weaks(\cB'')$ and the bidual $\beta''$ of the canonical action $\beta$ on $\bk(\cB)$ are isomorphic as W*-actions.
 In particular, $\bk_\weaks(\cB'') $ is isomorphic to $\bk(\cB)''$.
\end{theorem}
\begin{proof}
We first show that $\ell^2_\weaks(\cB'')\cong \ell^2(\cB)''$ as W*-Hilbert $B_e''$-modules. This is the crucial point
if we follow the original construction of $\bk_\weaks(\cB'')$ at the beginning of Section~\ref{ssec:wstar algebra of kernels}.

 Let $\rho\colon B_e\to \Lb(H)$ be the universal representation; we extend it to the bidual and view it as a faithful W*-representation $\rho''\colon B_e''\to \Lb(H)$.
We may then view $\ell^2_\weaks(\cB'')$ as a $\wot$-closed subspace of $\Lb(H,\ell^2_\weaks(\cB'')\otimes_{\rho''}H)$.
 But $\ell^2_\weaks(\cB'')\otimes_{\rho''}H=\ell^2(\cB)\otimes_{\rho}H=:K$ and we have a faithful representation $U\colon \ell^2_\weaks(\cB'')\to \Lb(H,K)$ such that $U(x)h=x\otimes h$. Moreover, $\ell^2_\weaks(\cB'')$ is the $\wot$-closure of $U(\ell^2(\cB))$, i.e. $\ell^2(\cB)''$.

Looking at the linking algebra $L$ of $\ell^2(B_e)$, we may view $\bK_\weaks(G,\cB'')$ as the W*-completion of $\bK(\cB)$ in $L''$.
 But this closure is also equal to $\bK(\cB)''$.
 Now we have
  \begin{align*}
 (\beta''|_{\bK(\cB)''})|_{\bK(\cB)}
     & =\beta''|_{\bk(\cB)}|_{\bK(\cB)}
       =\beta|_{\bK(\cB)}
     = \beta^\weaks|_{\bk(\cB'')}|_{\bK(\cB)}
      = (\beta^\weaks|_{\bK(\cB)''})|_{\bK(\cB)}.
 \end{align*}
 Hence we have two W*-partial actions on $\bK(\cB)''$, namely $\beta''|_{\bK(\cB)''}$ and $\beta^\weaks|_{\bK_\weaks(\cB'')}$, which are the unique W*-actions extending the $C^*$-partial action $\beta|_{\bK(\cB)}$. Therefore $\beta''|_{\bK(\cB)''}=\beta^\weaks|_{\bK_\weaks(\cB'')}$.
 But $\beta''$ and $\beta^\weaks$ are both W*-enveloping actions of $\beta''|_{\bK(\cB)''}$, then uniqueness of W*-enveloping actions implies that $\beta''$ is isomorphic to~$\beta^\weaks$.
\end{proof}

\begin{corollary}\label{cor:invariance of AD amenability under weak equivalence}
 If two Fell bundles $\cA$ and $\cB$ over the same group are weakly equivalent, then their enveloping W*-Fell bundles $\cA''$ and $\cB''$ are weakly W*-equivalent.
 In particular, AD-amenability of Fell bundles is preserved by weak equivalence of Fell bundles.
\end{corollary}
\begin{proof}
 The canonical actions on $\bk(\cA)$ and $\bk(\cB)$, $\alpha$ and $\beta$ respectively, are Morita equivalent through a partial action $\gamma$ on a $\bk(\cA)$-$\bk(\cB)$-equivalence bimodule $X$ (as we recalled at the beginning of Section~\ref{sec:ADA-Fell-bundles}).
  Then everything follows from Remark~\ref{rem:bidual of pa on module}, Theorem~\ref{the:bidual and wstar algebra of kernels} and Definitions~\ref{def:equivalence of wstar Fell bundles} and~\ref{def:AD amenable}.
\end{proof}

\begin{corollary}\label{cor:Fell bundle AD amenable and action on the kernels}
 A Fell bundle $\cB$ is AD-amenable if and only if the canonical action on $\bk(\cB),$ $\beta,$ is AD-amenable.
\end{corollary}
\begin{proof}
 Just recall that $\cB$ is weakly equivalent to $\cB_\beta$ \cites{Abadie-Ferraro:Equivalence_of_Fell_Bundles,Abadie-Buss-Ferraro:Morita_Fell} and use the Corollary above.
\end{proof}

\subsection{The dual coaction: another picture for the W*-algebra of kernels}\label{sec:dual-coaction}

In this section we want to show that the sectional W*-algebra $W^*_\red(\M)$ of a W*-Fell bundle $\M$ over $G$ carries a canonical $G$-coaction and identify the crossed product by this coaction with the W*-algebra of kernels $\bk_\weaks(\M)$.

Recall that a coaction of $G$ on a W*-algebra $N$ is a faithful unital W*-homo\-morphism $\delta\colon N\to N\bar\otimes W^*_\red(G)$ satisfying $(\delta\otimes\id)\delta=(\id\otimes\delta_G)\delta$, where $\bar\otimes$ denotes the (spatial) tensor product of W*-algebras.
Given such a coaction, the W*-crossed product  is defined as the W*-subalgebra of $N\bar\otimes\Lb(\ell^2(G))$ generated by $\delta(N)$ and $1\otimes\ell^\infty(G)$ where, as usual, $\ell^\infty(G)$ is represented as a W*-subalgebra of $\Lb(\ell^2G)$ via multiplication operators. We omit this representation here for simplicity, that is, we view $\ell^\infty(G)$ as a subalgebra of $\Lb(\ell^2G)$. It turns out that
\begin{equation*}
N\bar\rtimes_\delta  G=\cspn^{\weaks}\{\delta(n)(1\otimes f): n\in N, f\in \ell^\infty(G)\}.
\end{equation*}
Representing $N$ on a Hilbert space or, more generally, on a self-dual Hilbert module $H$, the W*-crossed product $N\bar\rtimes_\delta  G$ gets represented as a W*-subalgebra  of $\Lb(H)\bar\otimes
\Lb(\ell^2G)=\Lb(H\otimes \ell^2G)$. This crossed product carries a canonical $G$-action, the so called dual action $\dual\delta$. It is given on a generator $\delta(n)(1\otimes f)$ by $\dual\delta_t(\delta(n)(1\otimes f))=\delta(n)(1\otimes\tau_t(f))$, where $\tau_t(f)(s):=f(st)$ denotes the right translation $G$-action on $\ell^\infty(G)$. This can also be described as $\dual\delta_t(x)=(1\otimes \rho_t) x(1\otimes \rho_t^{-1})$, where $\rho\colon G\to \Lb(\ell^2G)$ denotes the right regular representation of $G$.

Now, returning to the case of a W*-Fell bundle $\M$, we want to define a coaction $\delta_\M\colon W^*_\red(\M)\to W^*_\red(\M)\bar\otimes W^*_\red(G)$ that acts on generators $\Lambda_t(a)\in M_t$ with $a\in M_t$ by the formula
\begin{equation}\label{eq:dual-coaction-relation}
\delta_\M(\Lambda_t(a))=\Lambda_t(a)\otimes\lambda_t.
\end{equation}
This is therefore an extension of the usual dual coaction on $C^*_\red(\M)\sbe W^*_\red(\M)$. Here  $W^*_\red(G)$ denotes the group W*-algebra of $G$, that is, the W*-subalgebra of $\Lb(\ell^2G)$ generated by the left regular representation $\lambda\colon G\to \Lb(\ell^2G).$

To prove that $\delta_\M$ exists, we proceed as in the $C^*$-algebra situation (see \cite{Abadie:Enveloping}*{Section~8} or \cite{ExelNg:ApproximationProperty}): Let $\M\times G$ be the pullback of $\M$ along the first coordinate projection $G\times G\to G$. This is a W*-Fell bundle over $G\times G$ whose W*-algebra is canonically isomorphic to $W^*_\red(\M\times G)=W^*_\red(\M)\bar\otimes W^*_\red(G)$, in particular we have a canonical W*-embedding
\begin{equation*}W^*_\red(\M)\bar\otimes W^*_\red(G)\sbe \Lb(\ell^2_{\weaks}(\M\times G)).\end{equation*}
Now we define a unitary operator $V$ on the Hilbert W*-module $\ell^2_{\weaks}(\M\times G)$ by the formula
\begin{equation*}
V\zeta(s,t):=\zeta(s,s^{-1}t)\quad\mbox{for all }\zeta\in \ell^2_{\weaks}(\M\times G),\, s,t\in G.
\end{equation*}
Straightforward computations show that this is indeed a unitary operator with adjoint $V^*\zeta(s,t)=\zeta(s,st)$.
Now we define a $\weaks$-continuous injective unital homomorphism $\delta_\M\colon \Lb(\ell^2_{\weaks}(\M))\to \Lb(\ell^2_{\weaks}(\M\times G))$ by
\begin{equation*}\delta_\M(a):=V(a\otimes 1)V^*,\quad a\in W^*_\red(\M).\end{equation*}
It is easy to see that~\eqref{eq:dual-coaction-relation} is satisfied. Moreover, since $\{\Lambda_t(a)\colon a\in M_t,\ t\in G\}$ generates $W^*_\red(\M)$ as a W*-algebra, the above formula restricts to an injective $\weaks$-continuous unital homomorphism
\begin{equation*}\delta_\M\colon W^*_\red(\M)\to W^*_\red(\M)\bar\otimes W^*_\red(G).\end{equation*}
This is indeed a coaction, that is, the coassociativity identity $(\delta_\M\otimes\id)\circ\delta_\M =(\id\otimes\delta_G)\circ \delta_M$ holds, where $\delta_G\colon W^*_\red(G)\to W^*_\red(G)\bar\otimes W^*_\red(G)$ denotes the comultiplication of $W^*_\red(G)$ (which, incidentally,  is the coaction $\delta_\M$ for the trivial one-dimensional Fell bundle $\M=\C\times G$).

\begin{remark}\label{rem:cond-exp}
There is a canonical normal conditional expectation $E\colon W^*_\red(\M)\onto M_e$ given on generators by $E(\Lambda(a))=\delta_{t,e}(a)$ for all $a\in M_t$. This can be proved as in the $C^*$-case, or it can be deduced from the existence of the dual coaction $\delta_\M$ above as follows: Consider the canonical tracial state $\tau\colon W^*_\red(G)\to \C$ given by $\tau(x)=\braket{\delta_e}{x\delta_e}$. Then $E=(\id\otimes\tau)\circ\delta_\M$ is the desired conditional expectation.
\end{remark}

\begin{proposition}\label{pro:kernel-CP}
For a W*-Fell bundle $\M$, we have a canonical isomorphism
\begin{equation*}W^*_\red(\M)\bar\rtimes_{\delta_\M} G\cong \bk_\weaks(\M),\end{equation*}
that identifies a generator $\delta_\M(a)(1\otimes f)\in W^*_\red(\M)\bar\rtimes_{\delta_\M} G$ with the kernel $k_{a,f}(s,t):=a(st^{-1})f(t)$ for $a\in \contc(\M)$ and $f\in \ell^\infty(G)$.
This isomorphism is $G$-equivariant with respect to the dual $G$-action on $W^*_\red(\M)\bar\rtimes_{\delta_\M} G$ and the canonical $G$-action on $\bk_\weaks(\M)$.
\end{proposition}
\begin{proof}
Let $N:=W^*_\red(\M)$ and $\delta:=\delta_\M$.
We show how to turn $\ell^2_\weaks(\M)$ into a W*-Hilbert $N\bar\rtimes_\delta G$-$M_e$-bimodule.

Consider the map $\iota\colon \contc(\M)\to N\otimes_{\alg}\contc(G)\sbe N\bar\otimes\ell^2(G)$ defined by $\iota(\xi)=\delta(\xi)(1\otimes\delta_e)=\sum_{s\in G}\Lambda(\xi(s))\otimes\delta_s$. Here and throughout this proof $\{\delta_s\}_{s\in G}$ will also denote the standard ortonormal basis of $\ell^2(G)$ -- apologies for the overuse of the symbol $\delta$ here! Let $X$ be the $\weaks$-closure of the image of $\iota$ in $N\bar\otimes\ell^2(G)$.
Notice that with respect to the $N$-inner product on $N\bar\otimes\ell^2(G)$ we have
\begin{equation*}\braket{\iota(\xi)}{\iota(\eta)}_N=\sum_{s,t\in G}\braket{\Lambda(\xi(s)\otimes\delta_s)}{\Lambda(\eta(t))\otimes\delta_t}=\Lambda(\braket{\xi}{\eta}_{M_e}),\end{equation*}
for all $\xi,\eta\in \contc(\M)$, where $\braket{\xi}{\eta}_{M_e}$ denotes the $M_e$-valued inner product on $\contc(\M)\sbe \ell^2_\weaks(\M)$. Since $\Lambda$ is a W*-embedding $M_e\into N$, it follows that the image of the $N$-valued inner product on $X$ takes values in $\Lambda(M_e)\cong M_e$ so that $X$ can be viewed as a right W*-Hilbert $M_e$-module and $\iota$ extends to an isomorphism $\ell^2_\weaks(\M)\cong X$ of W*-Hilbert $M_e$-modules.
The advantage of this picture is that $X$ is also canonically a left W*-Hilbert $N\bar\rtimes_\delta G$-module, where the left inner product is defined by $_I\braket{\xi}{\eta}:=\delta(\xi)(1\otimes \chi_e)\delta(\eta^*)\in I$ for $\xi,\eta\in \contc(\M)$. The image of this inner product generates a W*-ideal $I$ of $N\bar\rtimes_\delta G$, namely the W*-ideal  generated by the projection $p:=\chi_e$. It follows that $I\cong \Lb(\ell^2_\weaks(\M))$; this isomorphism identifies $\delta(\xi)(1\otimes p)\delta(\eta^*)$ with $\theta_{\xi,\eta}=\ket{\xi}\bra{\eta}\in \bK(\ell^2_\weaks(\M))\sbe \Lb(\ell^2_\weaks(\M))$, and it is determined by this formula and the fact that it is $\weaks$-continuous.

Next, considering the dual $G$-action $\dual\delta$ on $Q:=N\bar\rtimes_\delta G$, we notice that the linear $G$-orbit of $I$ is $\weaks$-dense. This is because $\dual\delta_{t^{-1}}(\chi_e)=\chi_t,$ so that $\dual\delta_{t^{-1}}(I)$ is the W*-ideal of $N\bar\rtimes_\delta G$ generated by the projection $p_t=\chi_t$, and these projections generate $\ell^\infty(G)$ as a W*-algebra. Therefore $\dual\delta$ can be viewed as the W*-enveloping action of its restriction $\dual\delta|_I$.
On the other hand, the $G$-action on the W*-algebra of kernels $\bk_\weaks(\M)$ is also enveloping for a partial action on $\Lb(\ell^2_\weaks(\M))$. By uniqueness of enveloping W*-actions (Proposition~\ref{pro:enveloping-vn-partial}), to see that $\bk_\weaks(\M)\cong Q$, it is enough to see that the restriction of $\dual\delta$ to $I$ coincides with the partial action on $\Lb(\ell^2_\weaks(\M))$ obtained as restriction of the $G$-action $\beta^\weaks$ on $\bk_\weaks(\M)$. But by definition, $\beta^\weaks$ is the unique $\weaks$-continuous extension of the $G$-action $\beta$ on the $C^*$-algebra of kernels $\bk(\M)$ given by $\beta_r(k)(s,t)=k(sr,tr)$ for a kernel $k\in \bk_c(\M)$. An elementary compact operator $\theta_{\xi,\eta}\in \bK(\ell^2(\M))\sbe \bK(\ell^2_\weaks(\M))$ identifies with the kernel function $k_{\xi,\eta}(s,t):=\xi(s)\eta(t)^*$.
And by \cite{Abadie:Enveloping}*{Proposition~8.1} we have a $C^*$-isomorphism $\bk(\M)\cong B:=C^*_\red(\M)\rtimes_{\delta} G$ that is $G$-equivariant for the dual $G$-action $\dual\delta$ on $B$ and $\beta$ on $\bk(\M)$. Here $\delta$ also denotes the dual coaction of $G$ on $C^*_\red(\M)$; this is a restriction of the dual coaction on $N=W^*_\red(\M)$, denoted by the same symbol $\delta$. The isomorphism $\bk(\M)\cong B$ is given as in the statement (see the proof of Proposition~8.1 in \cite{Abadie:Enveloping}). The $C^*$-algebra of compact operators $\bK(\ell^2(\M))$ identifies, as above, with the $C^*$-ideal $J$ of $B$ generated by $p=\chi_e$. This is $\weaks$-dense in $I$. Since the partial $G$-action on $J$ we get from viewing it as an ideal of $\bk(\M)$ coincides with the partial action coming from the dual action on $B$, the same has to be true for the $\weaks$-closures, that is, via the isomorphism $\Lb(\ell^2_\weaks(\M))\cong I$ the partial action on $I$ we get by restriction of $\dual\delta$ is the partial action we get from $\bk_\weaks(\M)$ by restricting it to the W*-ideal $\Lb(\ell^2_\weaks(\M))$.
\end{proof}

\begin{corollary}\label{cor:duality}
For every W*-Fell bundle $\M$ we have a canonical isomorphism
\begin{equation*}\bk_\weaks(\M)\rtimes_{\beta^\weaks}G\cong W^*_\red(\M)\bar\otimes\Lb(\ell^2G).\end{equation*}
\end{corollary}
\begin{proof}
This follows from Proposition~\ref{pro:kernel-CP} and general duality theory for crossed products by W*-coactions, see \cite{Nakagami-Takesaki:Duality}.
\end{proof}

We recall from \cite{Abadie-Buss-Ferraro:Morita_Fell} that given a Fell subbundle $\cA$ of  $\cB$ we can identify $\bk(\cA)$ with the norm closure of $\bk_c(\cA)$ in $\bk(\cB).$
This inclusion has a W*-counterpart.

\begin{corollary}\label{cor:inclusion of wstar algebra of kernels}
If $\mathcal{N}$ is a W*-Fell subbundle of $\M$ and we view $\bk(\mathcal{N})$ as a $C^*$-subalgebra of $\bk(\M)\subseteq \bk_\weaks(\M),$ then $\bk_\weaks(\mathcal{N})$ is isomorphic to the $\weaks-$closure of $\bk(\mathcal{N})$ in $\bk_\weaks(\M).$
\end{corollary}
\begin{proof}
The inclusion $\bk_\weaks(\mathcal{N})\subseteq \bk_\weaks(\M)$ is just the inclusion $W^*_\red(\mathcal{N})\bar\rtimes_{\delta_\mathcal{N}} G\subseteq W^*_\red(\M)\bar\rtimes_{\delta_\M} G$ provided by Proposition \ref{pro:kernel-CP}.
\end{proof}

\section{Exel's approximation property and AD-amenability}\label{sec:AP-ADA}

The main goal of this section is to compare the notion of amenability in the sense of Anantharaman-Delaroche with the approximation property introduced by Exel in \cite{Exel:Amenability}. We start by recalling Exel's approximation property:

\begin{definition}\label{def:AP}
 A Fell bundle $\cB=\{B_t\}_{t\in G}$ has the approximation property (AP) if there exists a net $\{a_i\}_{i\in I}$ of functions $a_i\colon G\to B_e$ with finite support such that
 \begin{enumerate}[(i)]
  \item $\sup_{i\in I} \|\sum_{r\in G} a_i(r)^*a_i(r)\|<\infty$.
  \item For every $t\in G$ and $b\in B_t$, $\lim_i \| b- \sum_{r\in G} a_i(tr)^*ba_i(r)\|=0$.
 \end{enumerate}
 A partial action $\alpha$ on a $C^*$-algebra has the AP if the semidirect product bundle $\cB_\alpha$ has the AP.
\end{definition}

\begin{remark}\label{rem:equivalent to AP}
Notice that (i) above means that $\{a_i\}_{i\in I}$ is a bounded net when viewed as a net in the Hilbert $B_e$-module $\ell^2(G,B_e)$. Indeed, the original definition of the AP in \cite{Exel:Amenability} uses such nets and Proposition~4.5 in \cite{Exel:Amenability} says that both definitions are equivalent (the difference being whether the supports of the functions are required to be finite or not).

Condition (ii) can also be weakened: it is enough to check the norm convergence in (ii) for $b$ in total subsets of $B_t$, that is, for $b$ in a subset $B_t^0$ spanning a norm-dense subset of  $B_t$ for each $t\in G$.
\end{remark}

As a way of combining Exel's approximation property and amenability in the sense of Anantharaman-Delaroche \cites{Anantharaman-Delaroche:ActionI,Anantharaman-Delaroche:ActionII,Anantharaman-Delaroche:Systemes}, we introduce the following:

\begin{definition}\label{def:WAP for W-Fell bundles}
 A W*-Fell bundle $\M=\{M_t\}_{t\in G}$ has the W*-approximation property (W*AP) if there exists a net of functions $\{a_i\colon G\to M_e\}_{i\in I}$ with finite support such that
 \begin{enumerate}[(i)]
  \item $\sup_{i\in I} \|\sum_{r\in G} a_i(r)^*a_i(r)\|<\infty$, and
  \item for every $t\in G$ and $b\in M_t$, 
  $$\lim_i \sum_{r\in G} a_i(tr)^*ba_i(r)=b$$ in the $\weaks$-topology of $M_t$.
 \end{enumerate}
 We say that a W*-partial action $\gamma$ has the W*AP if the associated W*-Fell bundle $\cB_\gamma$ has the W*AP.

 A ($C^*$-)Fell bundle $\cB$ has the WAP if its W*-enveloping Fell bundle $\cB''$ has the W*AP and a $C^*$-partial action
 $\alpha$ has the WAP if $\alpha''$ has the W*AP.
\end{definition}

 The acronym WAP should be read ``weak approximation property''; the
 reason for this is that we will be able to translate the WAP into a condition which is apparently weaker than the AP (compare the definition of the AP with claim (\ref{item:netbe}) in Theorem \ref{thm:the mega theorem}). For the time being we must treat AD-amenability, the WAP and AP  as possibly non equivalent conditions (and similarly in the W*-case). After Theorem~\ref{thm:the mega theorem}, there will be no point in making this distinction (the W*-version of this being Theorem \ref{thm: WAP iff AD amenable for wstar Fell bundles}).
We recommend the reader to consult the statements of Theorems \ref{thm:the mega theorem} and \ref{thm: WAP iff AD amenable for wstar Fell bundles} at this point to get a feeling of what we want to do next.

Remark~\ref{remark: bidual partial action and bidual bundle} implies that a $C^*$-partial action $\alpha$ has the WAP if and only if $\cB_\alpha$ has the WAP. We shall prove in what follows that AD-amenability and the WAP are equivalent notions, first for global actions and later also for general Fell bundles. This is not trivial, even for global actions, because the AD-amenability of a global action requires the existence of a certain net that takes central values (see Theorem~\ref{the:ADA-characterisations}) while for the WAP this is not explicitly necessary (Definition~\ref{def:WAP for W-Fell bundles}). 

Let $\gamma$ be a (global) action of $G$ on the W*-algebra $N$. As usual, we write $\tilde{\gamma}$ for the action of $G$ on $\ell^\infty(G,N)$ given by $\tilde{\gamma}_t(f)(r)=\gamma_t(f(t^{-1}r))$ and view $N$ as the subalgebra of constant functions in $\ell^\infty(G,N)$. Abusing the notation we also use the same notation for the $G$-action on functions $f\in \ell^2(G,N)$.
The following result gives an explicit characterisation of the W*AP for global actions.

\begin{proposition}
 Let $\gamma$ be a global action of $G$ on a W*-algebra $N$. Then $\gamma$ has the W*AP if and only if there exists a net $\{a_i\}_{i\in I}$ of finitely supported functions $a_i\colon G\to N$ such that $\{a_i\}_{i\in I}$ is bounded in $\ell^2(G,N)$ and $\{\langle a_i,b\tilde{\gamma}_t(a_i)\rangle_2\}_{i\in I}$ $\weaks$\nb-converges to $b$ for all $b\in N$ and $t\in G$.
\end{proposition}
\begin{proof}
We view the fibre of $\cB_\gamma$ at $t$ as $N\delta_t$ and denote its elements by $x\delta_t$.
With this notation $x\delta_ty\delta_s = x\gamma_t(y)\delta_{ts}$ and $(x\delta_t)^* = \gamma_{t^{-1}}(x^*)\delta_{t^{-1}}$.
Viewing $x\in N$ as $x\delta_e$, we can then think of a function $a\colon G\to N$ as a function from $G$ to the unit fibre $N\delta_e\equiv N$.
If $a\colon G\to N$ has finite support, then for every $t\in G$ and $b\in N$ we have: $\sum_{r\in G} (a(r)\delta_e)^*(a(r)\delta_e)=\sum_{r\in G}a(r)^*a(r)$ and
\begin{align*}
 \sum_{r\in G} (a(tr)\delta_e)^*b\delta_t a(r)\delta_e
  & = \sum_{r\in G} a(tr)^*b\gamma_t(a(r)))\delta_t\\
  & = \sum_{r\in G} a(r)^*b\gamma_t(a(t^{-1}r)))\delta_t
  = \langle a,b\tilde{\gamma}_t(a)\rangle\delta_t.
\end{align*}
The proof follows directly from the computations above and from the fact that under the identification $N\to N\delta_t,\ x\mapsto x\delta_t$, the $\weaks$-topology of $N\delta_t$ is just the $\weaks$-topology of $N$.
\end{proof}

In order to show that the  W*AD-amenability is equivalent to W*AP for W*-Fell bundles we shall need the following result:

\begin{lemma}\label{lem:characterization of AD amenability with dense subalgebra}
 Let $\gamma$ be a W*-global action of $G$ on $N.$
 Then $\gamma$ is AD-amenable if and only if there exists a $\gamma$-invariant $\weaks$-dense *-subalgebra $A\subseteq N$ and a bounded net $\{a_i\}_{i\in I}\subseteq \ell^2(G,N)$ of functions with finite support such that
for all $b\in A$ and $t\in G,$ $\{\langle a_i,b\tilde{\gamma}_t(a_i)\rangle_2\}_{i\in I}$ $\weaks$-converges to $b$.
\end{lemma}
\begin{proof}
 The direct implication follows from Theorem~\ref{the:ADA-characterisations} (with $A=N$).
 For the converse we view $N$ as a concrete (unital) von Neumann algebra of operators on some Hilbert space $H$, $N\subseteq \Lb(H),$ and take a *-subalgebra $A\sbe N$ and a net $\{a_i\}_{i\in I}$ as in the statement. For $t\in G$ and $i\in I$ we define $\varphi_i^t\colon N\to N$ by $\varphi^t_i(b):=\braket{a_i}{b\tilde\gamma_t(a_i)}$. Then $\{\varphi^t_i\}_{i\in I}$ is a net of uniformly bounded linear maps (with uniform bound $c:=\sup_i \|a_i\|_2^2<\infty$). By assumption $\varphi_i^t(b)\to b$ in the weak*-topology for every $b\in A$ and $t\in G$. A standard argument shows that the same happens for all $b$ in the norm closure of $A$, which is then a (w*-dense) $C^*$-subalgebra of $N$. Hence we may assume, without loss of generality, that $A$ is already a $C^*$-algebra. In particular we may assume that $A$ is closed by continuous functional calculus and that $\Lambda:=\{x\in A\colon 0\leq x\leq 1\}$ is an approximate unit for $A$ and thus $\weaks$-converges to $1_N$ (this follows from the assumption that $A$ is w*-dense in $N$).

 For each $(i,\lambda)\in I\times \Lambda$ we define
 \begin{equation*}
  P_{i,\lambda}\colon \ell^\infty(G,N)\to N, \ P_{i,\lambda}(f) = \langle \lambda^{1/2} a_i,f \lambda^{1/2}a_i\rangle,
 \end{equation*}
  where the product $f \lambda^{1/2}a_i$ represents the diagonal
  action of $f \lambda^{1/2}\in \ell^\infty(G,N)$ on
  $\ell^2(G,N)$. Each $P_{i,\lambda}$ is a completely positive linear map with norm $\|P_{i,\lambda}\|=\|\lambda^{1/2}a_i\|^2_2\leq c:=\sup_{i\in I}\|a_i\|_2^2<\infty$.

 Let $K$ be the Hilbert space $\ell^2(\Lambda,H)$ and define, for each $i\in I,$ the function
 \begin{equation*}
  P_i\colon \ell^\infty(G,N) \to \Lb(K),\ P_i(f)g|_\lambda=P_{i,\lambda}(f)(g|_\lambda).
 \end{equation*}
 If we view $K=\ell^2(\Lambda,H)$ as the direct sum of $\Lambda$-copies of $H$, then $P_i(f)$ is the ``diagonal'' operator formed by the family $\{P_{i,\lambda}\}_{\lambda}.$
 Thus $P_i$ is completely positive and $\|P_i\|\leq c$ for all $i\in I$.

The set of completely positive maps $Q\colon \ell^\infty(G,N)\to
\Lb(K)$ with $\|Q\|\leq c$ is compact with respect to the topology of
pointwise $\weaks$-convergence, thus there exists a completely
positive map $P\colon \ell^\infty(G,N)\to \Lb(K)$ and a subnet $\{P_{i_j}\}_{j\in J}$ such that $P(f)=\lim_j P_{i_j}(f)$ in the $\weaks$-topology for every $f\in \ell^\infty(G,N).$
By passing to a subnet we may therefore assume that $\{P_i\}_{i\in I}$ converges to $P$ for the pointwise $\weaks$-topology.

 As a consequence of the last paragraph we get that, for each $f\in \ell^\infty(G,N)$ and $\lambda\in \Lambda,$ $\{P_{i,\lambda}(f)\}_{i\in I}$ $\weaks$-converges to some $P_\lambda(f).$
 In fact, the map $P_\lambda\colon \ell^\infty(G,N)\to N,$ $f\mapsto P_\lambda(f),$ is completely positive and $\|P_\lambda\|\leq c$ for all $\lambda\in \Lambda.$

 For each $\lambda\in \Lambda$ we define $Q_\lambda\colon
 \ell^\infty(G,Z(N))\to N\sbe \Lb(H)$ as the restriction of $P_\lambda.$
 We claim that $\{Q_\lambda\}_{\lambda\in \Lambda}$ converges $\weaks$-pointwise.
 Indeed, it suffices to prove that for each positive $f\in \ell^\infty(G,Z(N))$ the net $\{Q_\lambda(f)\}_{\lambda\in \Lambda}$ is increasing.
 Take $\lambda,\mu\in \Lambda$ with $\lambda\leq \mu.$
 Then, for every $h\in H$ and $i\in I$:
 \begin{align*}
  \langle h,P_{i,\lambda}(f)h\rangle
    & = \sum_{r\in G} \langle a_i(r)h,\lambda^{1/2}f(r)\lambda^{1/2} a_i(r)h\rangle \\
    &  = \sum_{r\in G} \langle a_i(r)h,f^{1/2}(r)\lambda f^{1/2}(r)a_i(r)h\rangle \\
    & \leq  \sum_{r\in G} \langle a_i(r)h,f^{1/2}(r)\mu f^{1/2}(r)a_i(r)h\rangle
      = \langle h,P_{i,\mu}(f)h\rangle.
 \end{align*}
 Taking limit in $i$ we get $\langle h,Q_{\lambda}(f)h\rangle \leq \langle h,Q_{\mu}(f)h\rangle$ and it follows $Q_\lambda(f)\leq Q_\mu(f).$
 Let $Q\colon \ell^\infty(G,Z(N))\to N$ be the pointwise $\weaks$-limit of $\{Q_\lambda\}_{\lambda\in \Lambda}.$

 Let us prove that the image of $Q$ is contained in $Z(N).$
 It suffices to show that for $f\in \ell^\infty(G,Z(N))^+$ and a self-adjoint $b\in A,$ $Q(f)b$ is self-adjoint.
 Let $\{P_{\lambda_j}\}_{j\in J}$ be a pointwise $\weaks$-convergent subnet of $\{P_\lambda\}_{\lambda\in \Lambda}.$
 Clearly, both $fb$ and $P(fb)$ are self-adjoint.
 We claim that $Q(f)b=P(fb).$
 Fix $h,k\in H.$
 Using the inner product of $\ell^2_{\weaks}(G,N)\otimes_{N}H $ in the following computations, we deduce
 \begin{align*}
  |\langle h, (Q(f)b-P(fb))k\rangle |
    & = \lim_j \lim_i |\langle \lambda_j^{1/2} a_i\otimes h,f\left( \lambda_j^{1/2}a_i\otimes bk - b\lambda_j^{1/2} a_i\otimes k\right) \rangle  |\\
    & \leq \lim_j \lim_i \sqrt{c}\|h\|\|f\|\| \lambda_j^{1/2}a_i\otimes bk - b\lambda_j^{1/2} a_i\otimes k\|.
 \end{align*}
 The double limit above is zero because $\lim_j \lim_i \| \lambda_j^{1/2}a_i\otimes bk - b\lambda_j^{1/2} a_i\otimes k\|^2$ is
 \begin{align*}
  \lim_j \lim_i \langle bk,\langle a_i,\lambda_j a_i\rangle bk \rangle & + \langle k,\langle a_i,\lambda_j^{1/2}b^2\lambda_j^{1/2} a_i\rangle k \rangle-\\
  & -  \lim_j\lim_i \langle bk,\langle a_i,\lambda_j^{1/2}b\lambda_j^{1/2} a_i\rangle k \rangle
      - \langle k,\langle a_i,\lambda_j^{1/2}b\lambda_j^{1/2} a_i\rangle bk \rangle\\
  = \lim_j \langle bk,\lambda_j bk \rangle  +& \langle k,\lambda_j^{1/2}b^2\lambda_j^{1/2} k \rangle- \langle bk,\lambda_j^{1/2}b\lambda_j^{1/2}  k \rangle
      - \langle k,\lambda_j^{1/2}b\lambda_j^{1/2}  bk \rangle\\
 = \langle bk,bk \rangle +\langle k,&b^2  k \rangle  - \langle bk,b  k \rangle
     - \langle k,bbk \rangle=0.
 \end{align*}
This shows that $Q(f)b=P(fb)$.
From now on we think of $Q$ as a completely positive map from $\ell^\infty(G,Z(N))$ to $Z(N).$

We claim that $Q$ is a projection.
Take $a\in Z(N).$
Then, in the $\wot$ topology:
\begin{equation*}
 Q(a) = \lim_\lambda \lim_i P_{i,\lambda}(a)=\lim_\lambda \lim_i \langle a_i,\lambda a a_i\rangle =\lim_\lambda \lim_i a \langle a_i,\lambda a_i\rangle=\lim_\lambda a\lambda = a.
\end{equation*}
In particular $Q(1)=1,$ so $Q$ is a norm one projection.
The proof will be completed once we show $Q$ is equivariant.

Suppose we can prove, for all $f\in \ell^\infty(G,Z(N)),$ that
\begin{equation}\label{equ:almost equivariance with lambda}
 Q_{\lambda}(\tilde{\gamma}_t(f))=\gamma_t(Q_{\gamma_{t^{-1}}(\lambda)}(f)).
\end{equation}
Since $\{\gamma_t(\lambda)\}_{\lambda\in \Lambda}$ is a subnet of $\{\lambda\}_{\lambda\in \Lambda},$ if we take the $\weaks$-limit in \eqref{equ:almost equivariance with lambda} we obtain
\begin{equation*}
Q(\tilde{\gamma}_t(f))=\lim_\lambda Q_{\lambda}(\tilde{\gamma}_t(f)) =\gamma_t( \lim_{\lambda} Q_{\gamma_{t^{-1}}(\lambda)}(f)) = \gamma_t( Q(f)).
\end{equation*}
Hence the proof will be complete after we show~\eqref{equ:almost equivariance with lambda}.

Fix $f\in \ell^\infty(G,Z(N))^+,$ $\lambda\in \Lambda$ and $t\in G.$
In the $\weaks$-topology:
\begin{align*}
 Q_\lambda(\tilde{\gamma}_t(f))
  & = \lim_i \langle \lambda^{1/2}a_i,\tilde{\gamma}_t(f)\lambda^{1/2} a_i\rangle
    = \lim_i \gamma_t(\langle \tilde{\gamma}_{t^{-1}}(a_i),\gamma_{t^{-1}}(\lambda)f\tilde{\gamma}_{t^{-1}}(a_i)\rangle)
\end{align*}
We know that the net $\{\langle \tilde{\gamma}_{t^{-1}}(a_i),\gamma_{t^{-1}}(\lambda)f\tilde{\gamma}_{t^{-1}}(a_i)\rangle\}_{i\in I}$ $\weaks$-converges.
We only need to prove it $\wot$-converges to $Q_{\gamma_{t^{-1}}(\lambda)}(f).$
To avoid the annoying inverse $t^{-1},$ we change $t$ by $t^{-1}.$

Notice $\langle \tilde{\gamma}_{t}(a_i),\gamma_{t}(\lambda)f\tilde{\gamma}_{t}(a_i)\rangle= \langle  f^{1/2}\tilde{\gamma}_{t}( \lambda^{1/2} a_i),f^{1/2}\tilde{\gamma}_{t}( \lambda^{1/2} a_i)\rangle$ is self-adjoint, and so it is $Q_{\gamma_{t}(\lambda)}(f).$
Thus, it suffices to show that, for all $h\in H,$
\begin{equation*}
 \lim_i\langle h, \langle \tilde{\gamma}_{t}(a_i),\gamma_{t}(\lambda)f\tilde{\gamma}_{t}(a_i)\rangle h\rangle = \langle h,Q_{\gamma_t(\lambda)}(f)h\rangle
\end{equation*}

In any Hilbert space we have $|\|x\|^2-\|y\|^2|\leq (\|x\|+\|y\|)\|x-y\|.$
In particular we use this inequality in $\ell^2_{\weaks}(G,N)\otimes_{N}H $:
\begin{align*}
 \lim_i|\langle h, \langle \tilde{\gamma}_{t}(a_i),&\gamma_{t}(\lambda)f\tilde{\gamma}_{t}(a_i)\rangle h\rangle - \langle h,Q_{\gamma_t(\lambda)}(f)h\rangle | = \\
 & = \lim_i | \|f^{1/2}\tilde{\gamma}_t(\lambda^{1/2}a_i)\otimes h\|^2- \| f^{1/2}\gamma_t(\lambda^{1/2})a_i\otimes h \|^2 |\\
 & \leq \lim_i 2\|f\|^{1/2}\|a_i\|_2\|h\|\|f^{1/2}(\tilde{\gamma}_t(\lambda^{1/2}a_i)-\gamma_t(\lambda^{1/2})a_i)\otimes h \| \\
 & \leq \lim_i 2\|f\|\|a_i\|_2\|h\|\|\gamma_t(\lambda^{1/2})(\tilde{\gamma}_t(a_i)-a_i)\otimes h \| \\
\end{align*}
Moreover, $\{\|a_i\|_2\}_{i\in I}$ is bounded and the limit of $\{\|\gamma_t(\lambda^{1/2})(\tilde{\gamma}_t(a_i)-a_i)\otimes h \|^2\}_{i\in I}$ is
\begin{multline*}
\lim_i \langle h,  \gamma_t(\langle a_i,\lambda a_i\rangle) h\rangle   +
     \langle h, \langle a_i,\gamma_t(\lambda)a_i\rangle h\rangle-\\
-\lim_i \langle h, \langle \tilde{\gamma}_t(a_i),\gamma_t(\lambda)a_i\rangle h\rangle+\langle h, \langle a_i,\gamma_t(\lambda)\tilde{\gamma}_t(a_i)\rangle h\rangle\\
=\langle h,  \gamma_t(\lambda) h\rangle  +
     \langle h, \gamma_t(\lambda) h\rangle- \langle h, \gamma_t(\lambda) h\rangle-\langle h, \gamma_t(\lambda) h\rangle = 0.
\end{multline*}
This implies \eqref{equ:almost equivariance with lambda} and the proof is complete.
\end{proof}

The next remark will be extremely useful to show that W*AD-amenability and the W*AP are equivalent.

\begin{remark}\label{rem:matrix algebras of Fell bundles}
Given a Fell bundle $\cB=\{B_t\}_{t\in G}$ and a finite subset
 $F=\{t_1,\ldots,t_n\}$ of
 $G$, let $\mathbb{M}_F(\cB)$ be the subset of $\bk(\cB)$ formed by
 the kernels supported in $F\times F$. It is readily checked that
 $\mathbb{M}_F(\cB)$ is a *-subalgebra of $\bk(\cB)$. Note that one
 can think of the elements of $\mathbb{M}_F(\cB)$ as the matrices
 $M=\{M_{i,j}\}_{i,j=1}^n$ with entries in $\cB$ such that $M_{i,j}\in
 B_{t_i t_j^{-1}},$ so in this way $\mathbb{M}_F(\cB)$ can be
 identified with a $C^*$-subalgebra of $\mathbb{M}_{|F|}(C^*(\cB))$
 (see \cite{Abadie-Ferraro:Equivalence_of_Fell_Bundles}*{Lemma 2.8}
 for details). Observe that in case $\{e,t^{-1}\}\subset F,$
 $\mathbb{M}_F(\cB)$ contains a copy of the linking algebra of $B_t.$

\par If $\M=\{M_t\}_{t\in G}$ is a W*-Fell bundle, then $\mathbb{M}_F(\M)$ is a W*-algebra. Indeed, we may take a representation $T\colon \M\to \Lb(H)$ which is isometric and a $\weaks-$homeomorphism when restricted to a fibre, which exists by Proposition~\ref{prop:faithful representation of W(B)}. 
 Then we have a representation $T^F\colon \mathbb{M}_F(\M)\to \Lb(H^n)\equiv\mathbb{M}_n(\Lb(H))$ such that $T^F(a_{i,j})_{i,j=1}^n =(T_{a_{i,j}})_{i,j=1}^n.$
 Since $\weaks-$convergence in $\mathbb{M}_n(\Lb(H))$ is equivalent to $\weaks-$convergence in the entries and the subspaces $T(M_t)$ are $\weaks-$closed, the range $T^F(\mathbb{M}_F(\M))$ is a W*-algebra.
 Moreover, $\weaks$ convergence in $\mathbb{M}_F(\M)$ is equivalent to $\weaks-$convergence in the entries.
 Consequently, if $\cB=\{B_t\}_{t\in G}$ is a Fell bundle then $\mathbb{M}_F (\cB)''=\mathbb{M}_F(\cB'').$
\end{remark}

\begin{theorem}\label{thm: WAP iff AD amenable for wstar Fell bundles}
 A W*-Fell bundle is W*AD-amenable if and only if it has the W*AP.
\end{theorem}
\begin{proof}
 Assume that the W*-Fell bundle $\M$ over the group $G$ is W*AD-amenable.
 By Corollary~\ref{cor:equivalence of AD amenability} the central partial action $\gamma$ on $Z:=Z(M_e)$ is W*AD-amenable.
 Let $\delta$ be the W*-enveloping action of $\gamma,$ acting on the commutative W*-algebra $Y.$
 We know that $Z$ is a W*-ideal of $Y$ and that $\delta$ is W*AD-amenable.

 Let $\{\xi_i\}_{i\in I}\subseteq \ell^2(G,Y)$ be a net for $\gamma$ as in Theorem~\ref{the:ADA-characterisations} and let $p\in Y$ be the unit of $Z.$
 We define $a_i:=\xi_i p$ and claim that $\{a_i\}_{i\in I}\subseteq \ell^2(G,Z)$ is a net as in the definition of the W*AP.

 First of all note that $\{a_i\}_{i\in I}$ is bounded because $\langle a_i,a_i\rangle = p \langle \xi_i,\xi_i\rangle,$ for all $i\in I.$
 If $p_t$ is the unit of $Z_t=Z\cap\delta_t(Z),$ then for every $t\in G$ and $x\in M_t$ we have (by the definition of the central partial action):
 \begin{align*}
  \lim_i \sum_{r\in G} a_i(tr)^* xa_i(r)
      & = \lim_i \sum_{r\in G} a_i(tr)^* xp_{t^{-1}}a_i(r)
        = \lim_i \sum_{r\in G} a_i(tr)^*\gamma_t(p_{t^{-1}}a_i(r)) x \\
      & = \lim_i \sum_{r\in G} p\xi_i(tr)^*p_t\delta_t(p) \delta_t(\xi_i(r)) x
        = \lim_i p_t\langle \xi_i,\tilde{\delta}_t(\xi_i)\rangle x\\
      & = p_t x = x,
 \end{align*}
where the limits are taken in the $\weaks$-topology.
This shows that $\M$ has the W*AP.

Now assume that $\M$ has the W*AP.
We will show that the canonical W*-action $\beta^\weaks$ on $\bk_\weaks(\M)$ is W*AD-amenable using Lemma~\ref{lem:characterization of AD amenability with dense subalgebra} and Theorem~\ref{the:ADA-characterisations}.
We set $\gamma:=\beta^\weaks,$ $N:=\bk_\weaks(\M)$ and $A:=\bk_c(\M).$
Recall that $N$ is a W*-completion of $\bk(\M)$ and that $A$ is norm dense in $\bk(\M).$
Hence $A$ is $\weaks$-dense in $N.$
Moreover, $A$ is $\gamma$ invariant because $A$ is invariant under the canonical action on $\bk(\M).$

We claim that for every finite ordered set $F=\{t_1,\ldots,t_n\}$ (a finite sequence without repetitions) $\mathbb{M}_F(\M)$ is a W*-subalgebra of $N.$
This is important because in such a case the convergence in the
topology of $\mathbb{M}_F(\M)$ relative to the
$\weaks$-topology of $N$ is just entrywise $\weaks$-convergence on the
matrix algebra $\mathbb{M}_F(\M).$

Recall that we defined $N=\bk_\weaks(\M)$ as the $\weaks$-closure of the image of the map $\pi^\beta\colon \bk(\M)\to \ell^\infty(G,\Lb(\ell^2_\weaks(\M)))$ (see Section~\ref{ssec:wstar algebra of kernels}).
Thus it suffices to prove that the image of $\rho\colon \mathbb{M}_F(\M)\to \Lb(\ell^2_\weaks(\M)),$ $\rho(k)f(r)=\sum_{s\in G} k(r,s)f(s),$ is a W*-subalgebra.
Here we think of the matrix $k$ as a kernel of compact support.
For all $f,g\in \ell^2_\weaks(\M)$ the map $\mathbb{M}_F(\M)\to M_e,$ $k\mapsto \langle f,\rho(k)g\rangle$ is $\weaks$-continuous.
It follows that the closed unit ball $\rho(\mathbb{M}_F(\M))_1$ is
$\weaks$-compact, therefore $\rho(\mathbb{M}_F(\M))$ is a W*-subalgebra of $\Lb(\ell^2_\weaks(\M)).$

Take a net of functions $\{a_j\}_{j\in J}$ as in the definition of W*AP for $\M.$
Let $\mathcal{F}$ be the set of finite ordered subsets of $G$ and consider $\Xi:=\mathcal{F}\times J$ as a directed set with the order $(U,j)\leq (V,i)$ $\Leftrightarrow$ $U\sbe V$ and $j\leq i$.
For each $\xi=(U,j)\in \Xi$ let $a_\xi\colon G\to \mathbb{M}_U(\M)$ be such that for every $r\in G$, $a_\xi(r)$ is the diagonal matrix with all the entries in the diagonal equal to $a_j(r)$.
Note that $\|\langle a_\xi,a_\xi\rangle \|=\|\langle a_j,a_j\rangle\|.$
Observe also that
  $\gamma_t(\mathbb{M}_U(\M))=\mathbb{M}_{Ut^{-1}}(\M)$.

Fix $t\in G$ and $k\in A.$
Take a finite set $U_0\sbe G$ such that $\supp(k)\sbe U_0\times U_0$.
If $\xi=(U,i)\in \Xi$ is such that $U\supseteq U_0\cup U_0t,$ then
 \begin{equation*}
 \langle a_\xi,k\tilde{\gamma}_t(a_\xi)\rangle  =\sum_{r\in G} a_\xi(tr)^* k\gamma_t(a_\xi(r))  ,
 \end{equation*}
 and $a_\xi(tr)^*k\gamma_t(a_\xi(r))\in \mathbb{M}_{\supp(k)}(\M)$.
 Moreover, considering the left and right entrywise action of $M_e$ on $\mathbb{M}_{\supp(k)}(\M)$, $a_\xi(tr)^*k\gamma_t(a_\xi(r))=a_j(tr)^*ka_j(r)$.
 It is then clear that $\lim_\xi \langle a_\xi,k\tilde{\gamma}_t(a_\xi)\rangle = k$ $\weaks$-entrywise and hence $\weaks$ in $N.$
\end{proof}

\begin{remark}\label{rem:net in the center in the WAP}
 In the proof above we incidentally showed that the net $\{a_i\}_{i\in I}$ in the Definition of the W*AP can be taken in the unit ball of $\ell^2(G,Z(M_e)),$ without altering the definition.
\end{remark}

\begin{corollary}\label{corollary:AD amenability and WAP}
 A W*-partial $\gamma=\{\gamma_t\colon M_{t^{-1}}\to M_t\}$ is W*AD-amenable if and only if it has the W*AP; which specifically means the existence of a bounded net $\{a_i\}_{i\in I}\subseteq \ell^2(G,M)$ of finitely supported functions such that, for all $t\in G$ and $b\in M_t,$ 
 \begin{equation}\label{equ:limit of W*AP}
  b=\lim_i \sum_{s\in G} a_i(ts)^* \gamma_{t}(\gamma_{t^{-1}}(b)a)
 \end{equation}
 in the $\weaks$-topology.
 Moreover, by changing the net one can assume the norm bound of the net to be 1 and the ranges of the functions $a_i$ to be contained in $Z(M);$ in which case it suffices to verify the limit~\eqref{equ:limit of W*AP} with $b$ being the unit of $M_t$ (for all $t\in G$).
\end{corollary}
\begin{proof}
The translation of the W*AP of the semidirect product bundle $\cB_\gamma=\{M_t\delta_t\}_{t\in G}$ to the existence of the net is the easy part, as it follows immediately from the identity $(a\delta_e)^*(b\delta_t)(c\delta_e)=a\gamma_t(\gamma_{t^{-1}}(b)c)\delta_t.$
Now, the claim about the replacement of the net holds by Remark~\ref{rem:net in the center in the WAP} and because if we denote $p_t$ the unit of $M_t,$ then for all $b\in M_t$ we have $a\gamma_t(\gamma_{t^{-1}}(b)c) =b a\gamma_t(\gamma_{t^{-1}}(p_t)c).$
The rest follows at once from the last Theorem above and Remark~\ref{rem:AD amenability of partial action and semidirect product bundle}.
\end{proof}

 We are now in a good position to prove the Fell bundle version of Theorem~\ref{thm: WAP iff AD amenable for wstar Fell bundles}. Basically, we translate the WAP into an approximation property very close to (an apparently weaker version of) Exel's AP. As already mentioned in the introduction, we then use the ideas of \cite{ozawa2020characterizations} to show that WAP and AP are equivalent.

\begin{theorem}\label{thm:the mega theorem}
 For every Fell bundle $\cB$ over a group $G$ the following are equivalent:
 \begin{enumerate}[(i)]
  \item\label{item:AD} $\cB$ is AD-amenable (i.e. $\cB''$ is the W*AD-amenable).
  \item\label{item:WAP} $\cB$ has the WAP (i.e. $\cB''$ has the W*AP).
  \item\label{item:net in center bidual} There exists a bounded
    net $\{a_i\}_{i\in I}\subseteq \ell^2(G,Z(B_e''))$ of functions with
    finite support  such that, for every $t\in G$ and $b\in B_t,$
    $\lim_i \sum_{r\in G}a_i(tr)^*ba_i(r)=b$ in $B_t''$ with respect
    to the $\weaks$-topology.
  \item\label{item:netbe} There exists a bounded net $\{a_i\}_{i\in I}\subseteq \ell^2(G,B_e)$ of functions with finite support such that, for every $t\in G$ and $b\in B_t,$ $\lim_i \sum_{r\in G}a_i(tr)^*ba_i(r)=b$ in the weak topology of $B_t.$
  \item\label{item:AP} $\cB$ has the AP.
 \end{enumerate}
\end{theorem}
\begin{proof}
Theorem~\ref{thm: WAP iff AD amenable for wstar Fell bundles} implies
that (\ref{item:AD}) and (\ref{item:WAP}) are equivalent. By Remark~\ref{rem:net in the center in the WAP}, (\ref{item:WAP}) and
(\ref{item:net in center bidual}) are equivalent.
To prove that (\ref{item:netbe}) implies  (\ref{item:WAP}) we can proceed exactly as in the proof of the converse in Theorem~\ref{thm: WAP iff AD amenable for wstar Fell bundles}, noticing that convergence in the weak topology of $\mathbb{M}_F(\cB)$ is entrywise convergence in the weak topology and, also, $\weaks$-convergence in $\mathbb{M}_F(\cB)''=\mathbb{M}_F(\cB'').$

We now prove that (\ref{item:WAP}) implies (\ref{item:netbe}).
  First we indicate how to approximate elements of $\ell^2(G,B_e'')$ by elements of $\ell^2(G,B_e)$ in a certain particular way.
We start by representing $\ell^2(G,B_e'')$ and $\ell^2(G,B_e)$ faithfully.
Let $\pi\colon \cB''\to \Lb(H)$ be a unital *-representation, fibre-wise faithful and $\weaks$-continuous (which exists by Proposition~\ref{prop:faithful representation of W(B)}).
Define $\rho:=\pi|_{B_e''}\colon B_e''\to \Lb(H)$ and note that we have canonical identifications
\begin{equation*}
 K:= \ell^2(G,H)  =  \ell^2(G,B_e)\otimes_\rho H = \ell^2(G,B_e'')\otimes_\rho H.
\end{equation*}
The proof now continues by representing $\ell^2(G,B_e'')$ and
$\bK(\ell^2(G,B_e''))$ by the process we described right before
Remark~\ref{rem:isomorphism between the centres}. The map $\hat\pi
\colon \ell^2(G,B_e'')\to \Lb(H,K)$ given by $\hat{\pi}(f)h =
f\otimes_\rho h$ is a faithful representation of the W*-tring
$\ell^2(G,B_e'').$ Notice that $\hat{\pi}(f)h$ may be identified with the function $t\mapsto \pi(f(t))h.$ Then we have a canonical nondegenerate representation $\hat{\pi}^l\colon \bK(\ell^2(G,B_e''))\to \Lb(K)$ such that $\hat{\pi}^l(T)\pi(f)=\pi(Tf)$, and thus we get a nondegenerate representation of the linking algebra $L$ of $\ell^2(G,B_e'')$:
\begin{equation*}
 \hat{\pi}^L\colon L\to \Lb(K\oplus H)
\qquad                                          \hat{\pi}^L\left(  \begin{array}{cc}
                                          T & f\\ \tilde{g} & S
                                         \end{array}\right) = \left(  \begin{array}{cc}
                                          \hat{\pi}^l(T) & \hat{\pi}(f)\\ \hat{\pi}(g)^* & \hat{\pi}^r(S)
                                         \end{array}\right),
\end{equation*}
where $\hat{\pi}^r\colon B_e''\to \Lb(H)$ is just $\rho$.

\par Fix an element $c\in C_c(G,B_e'')$.
Using a net in $B_e$ to approximate $c(t)$ (for each $t$ in the finite support of $c$) with respect to the w*-topology, we can construct a net $\{c_i\}_{i\in I}\subseteq C_c(G,B_e)$ such that $\supp(c_i)\subseteq \supp(c)$ and $c_i(t)\to c(t)$ in the $\weaks$-topology for every $t\in G$.
This construction implies that $\{\hat{\pi}(c_i)\}_{i\in I}$ $\wot-$converges to $\hat{\pi}(c)$ because, for all $h\in H$ and $k\in K$:
\begin{align*}
\lim_i\langle\hat{\pi}(c_i)h,k\rangle
=\lim_i \sum_{r\in\supp(c)}\langle\pi(c_i(r))h,k(r)\rangle
=\langle\hat{\pi}(c)h,k\rangle.
\end{align*}
It follows from the previous comments that $\hat{\pi}(c)\in
  \overline{\hat{\pi}(C_c(G,B_e))}^{\wot}$.
Now, according to \cite{MR996436}*{Theorem 4.8} and \cite{MR641217}*{Part I
  Ch. 3}, the unit ball of $\hat{\pi}^L(L)$ is *-strongly dense in the
unit ball of $\hat{\pi}^L(L)''$, and this bicommutant is the $\wot-$closure
of $\hat{\pi}^L(L).$
Hence there exists a net $\{\left( \begin{smallmatrix} T_j & a_j\\ \tilde{b_j} & S_j \end{smallmatrix}
 \right)\}_{j\in J}\sbe L$ in the closed ball of radius $\|\hat{\pi}(c)\|=\|c\|$ such that $\{\hat{\pi}^L\left( \begin{smallmatrix} T_j & a_j\\ \tilde{b_j} & S_j \end{smallmatrix}
 \right)\}_{j\in J}$ converges to $\left( \begin{smallmatrix} 0 & \hat{\pi}(c)\\ 0 & 0 \end{smallmatrix}
 \right)$ *-strongly.
Then, in the strong operator topology:
 \begin{align*}
  \lim_j \left( \begin{smallmatrix} 0 & \hat{\pi}(a_j)\\ 0 & 0 \end{smallmatrix}
 \right)
   & = \lim_j \left( \begin{smallmatrix} 1 & 0\\ 0 & 0 \end{smallmatrix}
 \right)\hat{\pi}^L\left( \begin{smallmatrix} T_j & a_j\\ \tilde{b_j} & S_j \end{smallmatrix}
 \right)\left( \begin{smallmatrix} 0 & 0\\ 0 & 1 \end{smallmatrix}
 \right)
     =  \left( \begin{smallmatrix} 1 & 0\\ 0 & 0 \end{smallmatrix}
 \right) \left( \begin{smallmatrix} 0 & \hat{\pi}(c)\\ 0 & 0 \end{smallmatrix}
 \right)\left( \begin{smallmatrix} 0 & 0\\ 0 & 1 \end{smallmatrix}
 \right)
   = \left( \begin{smallmatrix} 0 & \hat{\pi}(c)\\ 0 & 0 \end{smallmatrix}
 \right).
 \end{align*}
 Now we arrange the supports of the $a_j$'s to be contained in $\supp(c)$.
 Let $P\in \Lb(K)=\Lb(\ell^2(G,H))$ be the multiplication by the indicator function of $\supp(c)$.
 Then, in the strong operator topology: $\lim_j P\hat{\pi}(a_j)=P\hat{\pi}(c)=\hat{\pi}(c)$ and $P\hat{\pi}(a_j)=\hat{\pi}(a_j|_{\supp c})$.
 Thus we are allowed to assume $\supp(a_j)\sbe \supp(c)$ for all $j\in J$.
 We must retain the following facts about the net $\{a_j\}_{j\in J}\sbe \ell^2(G,B_e)$:
 \begin{itemize}
  \item $\supp(a_j)\sbe \supp(c)$ for all $j\in J$.
  \item $\|a_j\|\leq \|c\|$ for all $j\in J,$ with the norm of $\ell^2(G,B_e).$
  \item $\{\hat{\pi}(a_j)\}_{j\in J}$ converges strongly to $\hat{\pi}(c)$.
 \end{itemize}
We claim that these conditions imply, for every $t\in G$, $b\in B_t$ and $\varphi\in B_t',$ that
\begin{equation}\label{equ:wstar convergence}
 \lim_j \varphi\left( \sum_{r\in G}a_j(tr)^*ba_j(r) \right) = \varphi\left( \sum_{r\in G}c(tr)^*bc(r) \right).
\end{equation}
In other words, we claim that the net $\{ \sum_{r\in G}a_j(tr)^*ba_j(r)
\}_{k\in J}$ weakly converges to $\sum_{r\in G}c(tr)^*bc(r)$ in $B_t$.
Indeed, since $\pi|_{B_t''}$ is an isomorphism over its
image, and a homeomorphism considering in $B_t''$ and in
$\mathcal{L}(H)$ the w*-topology and the ultraweak topology ($\sigma
w$-topology), respectively,
it is enough to prove that $\pi(\sum_{r\in
  G}a_j(tr)^*ba_j(r))\stackrel{\sigma w}{\to}\pi(\sum_{r\in
  G}c(tr)^*bc(r))$.

Let $U\colon G\to \Lb(K)=\Lb(\ell^2(G,H))$ be the unitary representation given by $U_tf(r)=f(t^{-1}r)$, and $\pi^G\colon \cB''\to \Lb(K)$ be the $\ell^2$-direct sum of $G$ copies of $\pi$, that is, $\pi^G(b)f(r):=\pi(b)f(r)$.
Notice that $\{ \sum_{r\in G}a_j(tr)^*ba_j(r) \}_{j\in J}$ is bounded because, for all $u,v\in H$,
\begin{equation}\label{equ:inner product}
 \langle u,\pi\left( \sum_{r\in G}a_j(tr)^*ba_j(r)  \right)v\rangle
     = \langle U_t^* \hat{\pi}(a_j)u,\pi^G(b)\hat{\pi}(a_j)v\rangle.
\end{equation}
Since the ultraweak topology coincides with the weak operator topology
on bounded sets, to prove~\eqref{equ:wstar convergence} it is enough
to show that $\{ \pi(\sum_{r\in G}a_j(tr)^*ba_j(r)) \}_{j\in J}$
converges to $\pi(\sum_{r\in G}c(tr)^*bc(r))$ in the $\wot$ topology.
But our construction of $\{a_j\}_{j\in J}$ and
\eqref{equ:inner product} implies
\begin{align*}
 \lim_j \langle u,\pi\left( \sum_{r\in G}a_j(tr)^*ba_j(r)  \right)v\rangle
     & = \lim_j \langle U_t^* \hat{\pi}(a_j)u,\pi^G(b)\hat{\pi}(a_j)v\rangle\\
     & =\langle U_t^* \hat{\pi}(c)u,\pi^G(b)\hat{\pi}(c)v\rangle\\
     & = \langle u,\pi\left( \sum_{r\in G}c(tr)^*bc(r)  \right)v\rangle.
\end{align*}
Therefore \eqref{equ:wstar convergence} holds (note that
\eqref{equ:inner product} does not imply \eqref{equ:wstar convergence}
if we only know that $\pi(a_i)\stackrel{\wot}{\to}\pi(c)$).

\par Now assume that $\cB''$ has the W*AP and take a net $\{c_i\}_{i\in I}$ as
in the definition of the W*AP for $\cB''$, with all the $c_i$'s with
compact support.
Set $M:=\sup_{i\in I}\|\sum_{t\in G}c_i(t)^*c_i(t)\|$ and let
$\mathcal{F}$ and $\mathcal{F}'$ be the families of finite subsets of
$\cB$ and $\uplus_{t\in G} {B_{t,1}'}$, respectively, where ${B_{t,1}'}$ is the
closed unit ball of $B_t'$.
On $\Lambda:=(0,1)\times \mathcal{F}\times \mathcal{F}'$ we consider the canonical order $(\varepsilon,U,V)\leq (\delta,Y,Z)$ $\Leftrightarrow$ $\delta\leq \varepsilon$, $U\sbe Y$ and $V\sbe Z$.
For every $\lambda=(\varepsilon, U,V)\in \Lambda$ there exists $i_0\in I$ such that, for every $t\in G$, $b\in B_t\cap U$ and $\varphi\in B_t'\cap V$,
\begin{equation*}
 \left|\varphi\left(b - \sum_{r\in G} c_{i_0}(tr)^*bc_{i_0}(r)\right)\right|<\varepsilon/2.
\end{equation*}
Our approximation procedure ensures the existence of $a_\lambda\in \ell^2(G,B_e)$ such that: $\supp(a_\lambda)\sbe \supp(c_i)$, $\|a_\lambda\|^2\leq \|c_i\|^2\leq  M$ and
\begin{equation*}
 \left|\varphi\left(b - \sum_{r\in G} a_\lambda (tr)^*ba_\lambda (r)\right)\right|<\varepsilon,
\end{equation*}
for every $t\in G$, $b\in B_t\cap U$ and $\varphi\in B_t'\cap V$.
It is then clear that $\{a_\lambda\}_{\lambda\in \Lambda}$ is a net
satisfying (\ref{item:netbe}).

Note that the AP clearly implies (\ref{item:netbe}), so the proof will be completed once we show the converse also holds.
 This is the part of the proof we adapted from \cite{ozawa2020characterizations}. 
We shall assume $G$ is infinite, for otherwise $G$ is amenable and any Fell bundle over an amenable group has the AP \cite{Exel:Partial_dynamical}.
  
  Let $\{a_i\}_{i\in I}\subseteq \ell^2(G,B_e)$ be a net as in (\ref{item:netbe}).
  We set $C:=\sup_{i\in I} \|a_i\|$ and fix $\varepsilon>0;$ a positive integer $n;$ $t_1,\ldots,t_n\in G$ and $b_j\in B_{t_j}$ for $j=1,\ldots,n.$
  It suffices to construct $\widetilde{a}\in \ell^2(G,B_e)$ with finite support, $\|\widetilde{a}\|\leq C$ and 
  \begin{equation*}
   \left\|b_j-\sum_{s\in G} \widetilde{a}(t_js)^*b_j\widetilde{a}(s)\right\|<\varepsilon
  \end{equation*}
  for all $j=1,\ldots,n.$
  
  Define $E$ as the (Banach space) direct sum $B_{t_1}\oplus \cdots \oplus B_{t_n}.$
  Then the net 
  \begin{equation*}
   \left\{  \left(\sum_{s\in G}a_i(t_1s)^*b_1a_i(s),\ldots,\sum_{s\in G}a_i(t_ns)^*b_na_i(s)\right) \right\}_{i\in I}\subseteq E
  \end{equation*}
  converges to $(b_1,\ldots,b_n)\in E$ in the weak topology.
  By the Hahn-Banach Theorem there exists $i_j\in I$ and $\lambda_j\in [0,+\infty)$ (for $j=1,\ldots,m$) such that  $\lambda_1+\cdots+\lambda_m\leq 1$ and $\| b_j  - \sum_{k= 1}^m  \lambda_k\sum_{s\in G} a_{i_k}(t_js)^*b_ja_{i_k}(s)\|<\varepsilon$ for all $j=1,\ldots,n.$

  Name $F$ the union of the supports of the functions $a_{i_k}$ ($k=1,\ldots,m$) and set $F':=F\cup t_1^{-1}F\cup \cdots \cup t_n^{-1}F.$
  Since $G$ is infinite, there exists $r_1,\ldots,r_m\in G$ such that $(F'r_j)\cap (F'r_k)=\emptyset $ if $j\neq k.$
  From this it follows that $(tFr_j)\cap (Fr_k)=\emptyset$ for all $t=e,t_1,\ldots,t_n$ and $j,k=1,\ldots,m$ with $j\neq k.$
  Also, the support of the function $s\mapsto \lambda_k^{1/2}a_{i_k}(sr_k^{-1})$ is contained in $Fr_k$ for all $k=1,\ldots,m.$ Hence these new functions have disjoint supports and the sum $\widetilde{a}\colon G\to B_e,$ $s\mapsto \sum_{k=1}^m\lambda_k^{1/2}a_{i_k}(sr_k^{-1}),$ may be alternatively described as
  \begin{equation*}
  \widetilde{a}(s)
    =\begin{cases}
         0\mbox{ if }s\notin \cup_{k=1}^m Fr_k\\
          \lambda_k^{1/2}a_{i_k}(sr_k^{-1})\mbox{ if } s\in Fr_k
     \end{cases} .
  \end{equation*}
  
  A simple computation shows that $ \langle \widetilde{a},\widetilde{a}\rangle = \sum_{k=1}^m \lambda_k \langle a_{j_k},a_{j_k}\rangle,$ which implies $\|\widetilde{a}\|^2\leq \sum_{k=1}^m \lambda_kC^2\leq C^2.$
  Moreover, for all $j=1,\ldots,n,$
  \begin{equation*}
   \sum_{s\in G} \widetilde{a}(t_js)^*b_j\widetilde{a}(s) = \sum_{k= 1}^m  \lambda_k\left(\sum_{s\in G} a_{i_k}(t_js)^*b_ja_{i_k}(s)\right)
  \end{equation*}
  and we obtain  $\|b_j-\sum_{s\in G} \widetilde{a}(t_js)^*b_j\widetilde{a}(s)\|<\varepsilon$ for all $j=1,\ldots,n.$
\end{proof}

\begin{remark}\label{rem:on norm bound of the nets}
 By the proof above and Remark~\ref{rem:net in the center in the WAP}, we could replace  (i) in Definition~\ref{def:AP} by $\sup_{i\in I}\|\sum_{r\in G}a_i(r)^*a_i(r)\|\leq 1$ and all the occurrences of ``bounded net'' in Theorem~\ref{thm:the mega theorem} with ``net in the closed unit ball'' and this would still give conditions equivalent to the AP.
\end{remark}

\begin{corollary}\label{cor:WAP and canonical action on kernels}
 A Fell bundle $\cB$  has the AP if and only if the canonical action on its $C^*$-algebra of kernels $\bk(\cB)$ is AD-amenable.
\end{corollary}
\begin{proof}
 Follows from Theorem~\ref{thm:the mega theorem} and Corollary~\ref{cor:Fell bundle AD amenable and action on the kernels}.
\end{proof}

 Since we used the ideas of \cite{ozawa2020characterizations} to prove Theorem~\ref{thm:the mega theorem}, the Corollary below should come as no surprise.
 
\begin{corollary}[{c.f. Corollary~\ref{corollary:AD amenability and WAP}}]\label{cor:AP and amenability for partial actions}
  A $C^*$-partial action $\alpha=\{\alpha_t\colon A_{t^{-1}}\to A_t\}$ is AD-amenable if and only if it has the AP. AD amenability is equivalent to the existence of a bounded net $\{a_i\}_{i\in I}\subseteq \ell^2(G,A)$ of finitely supported functions such that, for all $t\in G$ and $b\in A_t,$
 \begin{equation*}
  \lim_i \sum_{s\in G} a_i(ts)^*\alpha_t(\alpha_{t^{-1}}(b)a_i(s)) = a 
 \end{equation*}
  in the weak topology of $A.$
  Moreover, by changing the net one may assume it to be norm bounded by 1 and the limit to converge in the norm topology (which implies the AP).
 \begin{proof}
  By \cite{Abadie:Enveloping}, the canonical action on $\bk(\cB_\alpha)$ is a Morita enveloping action for $\alpha.$
  Now Corollaries~\ref{cor:WAP and canonical action on kernels} and~\ref{cor:ADamenability and enveloping} imply $\alpha$ has the AP if and only if it is AD-amenable.  The rest follows as in the proof of Corollary~\ref{corollary:AD amenability and WAP} (use Remark~\ref{rem:on norm bound of the nets}).
 \end{proof}
\end{corollary}

From now on there is no point in making any difference between AD-amenability, WAP and AP and we can choose at will which of these conditions to use. We prefer to use the AP in our statements because its definition does not require to go to biduals or to $C^*$-algebras of kernels to be stated, and becuase it was the first notion of amenability to be stated in full generality (i.e. Fell bundles).

\begin{remark}
Notice that by Example~\ref{exa:trivial-partial} AD-amenable partial actions exist in abundance because every \cstar{}algebra $A$ admits an AD-amenable partial action of any fixed discrete group $G$.
This is in contrast with global actions  (which correspond to saturated bundles). Here the situation is different because no non-amenable group can act globally AD-amenably on a finite dimensional non-zero $C^*$-algebra.
\end{remark}

\section{Cross-sectional \texorpdfstring{$C^*-$}{C*-}algebras and the AP}\label{sec:Cross-sec-WAP}

In a previous version of this paper we did not know whether the AP was equivalent to AD-amenability, thus we were in need to prove that the full and reduced cross-sectional $C^*$-algebras of an AD-amenable bundle agree. Today we obtain this fact at least from two sources \cites{ExelNg:ApproximationProperty,Exel:Partial_dynamical}.

We continue with a result involving W*-Fell bundles.

\begin{proposition}
Let $\M=\{M_t\}_{t\in G}$ be a W*-Fell bundle. Then the following assertions are equivalent.
\begin{enumerate}[(i)]
\item  $M_e$ injective and $\M$ has the W*AP (or, equivalently, $\M$
  is W*AD-amenable);
\item $\bk_\weaks(\M)$ is injective and its canonical W*-action $\beta^\weaks$ has the W*AP (or is W*AD-amenable);
\item $\bk_\weaks(\M)\bar\rtimes_{\beta^\weaks} G$ is injective;
\item $W^*_\red(\M)$ is injective.
\end{enumerate}
\end{proposition}
\begin{proof}
First notice that $M_e$ is injective if and only if the W*-algebra of kernels $\bk_\weaks(\M)$ is injective.
Indeed, $M_e$ is W*-Morita equivalent to $\Lb(\ell^2_\weaks(\M))$ (via the W*-equivalence bimodule $\ell^2_\weaks(\M))$);
it follows that $M_e$ is injective if and only if $\Lb(\ell^2_\weaks(\M))$ is injective. But $\bk_\weaks(\M)$ carries a W*-action that is  enveloping for a partial W*-action on $\Lb(\ell^2_\weaks(\M))$. The claim now follows from Remark~\ref{rem:abelian-injective}. Also observe that $M_e$ is injective if $W^*_\red(\M)$ is injective because we have a canonical (normal) conditional expectation $W^*_\red(\M)\onto M_e$ (Remark~\ref{rem:cond-exp}).

The discussion above implies that (i) is equivalent to (ii). Since (ii) involves a W*-action, (ii)$\Leftrightarrow$(iii) by Theorem~\ref{the:ADA-characterisations}.
Finally, (iii)$\Leftrightarrow$(iv) by Corollary~\ref{cor:duality}.
\end{proof}

 It is known \cite{Exel:Partial_dynamical}*{Proposition 25.10} that the full and reduced cross-sectional $C^*$-algebras of a Fell bundle with the AP and nuclear unit fibre are nuclear, this are implications (iii)$\Rightarrow$(i) and (iii)$\Rightarrow$(ii) of the proposition below. The converse of this first appeared in \cite{buss2020amenability}*{Corollary 4.23}, where the authors use some of the results of our paper to prove it. 
In a first version of this paper the following proposition was stated with the WAP instead of the AP. But now using our Theorem~\ref{thm:the mega theorem} we can re-state it in terms of the AP, improving the result; this ends up yielding the same result as in \cite{buss2020amenability}*{Corollary 4.23}. Our original proof remains unchanged though.

\begin{proposition}[c.f. \cite{Anantharaman-Delaroche:Systemes}*{Th{\'e}or{\`e}me 4.5}]\label{prop:AD amenability and nuclearity}
 Let $\cB$ be a Fell bundle over a group $G$ with $B_e$ nuclear.
 Then the following are equivalent:
 \begin{enumerate}[(i)]
  \item $C^*(\cB)$ is nuclear.
  \item $C^*_{\red}(\cB)$ is nuclear.
  \item $\cB$ has the AP.
 \end{enumerate}
\end{proposition}
\begin{proof}
 Let $\bk(\cB)$ be the $C^*$-algebra of kernels and $\beta$ the canonical action of $G$ on $\bk(\cB)$.
 By the proof of \cite{Abadie-Buss-Ferraro:Morita_Fell}*{Theorem 6.3}, $\bk(\cB)$ is nuclear.
Moreover, by \cite{Anantharaman-Delaroche:Systemes}*{Th{\'e}or{\`e}me 4.5} and  Corollaries~\ref{cor:WAP and canonical action on kernels} and~\ref{cor:AP and amenability for partial actions}, (3) is equivalent to any of the following:
 \begin{itemize}
  \item[(1')] $\bk(\cB)\rtimes_\beta G=C^*(\cB_\beta)$ is nuclear.
  \item[(2')] $\bk(\cB)\rtimes_{\red,\beta} G=C^*_{\red}(\cB_\beta)$ is nuclear.
  \item[(3')] $\beta$ has the AP.
 \end{itemize}
 Since nuclearity is preserved by Morita equivalence of $C^*$-algebras, by \cite{Abadie-Buss-Ferraro:Morita_Fell} and \cite{Abadie-Ferraro:Equivalence_of_Fell_Bundles} we know that ($n$) is equivalent to ($n$'), for $n=1,2,3$.
\end{proof}

When specialised to partial actions the above proposition takes the following form:

\begin{corollary}\label{cor:AD amenability of pa and nuclearity}
 Let $\alpha$ be a partial action of the group $G$ on a nuclear $C^*$-algebra $A$.
 Then the following are equivalent:
 \begin{enumerate}[(i)]
  \item The full crossed product $A\rtimes_{\alpha} G$ is nuclear.
  \item The reduced crossed product $A\rtimes_{\red,\alpha} G$ is nuclear.
  \item $\alpha$ has the AP.
 \end{enumerate}
\end{corollary}
\begin{proof}
 Follows directly from the last Proposition above and Corollary~\ref{corollary:AD amenability and WAP}.
\end{proof}

The last two results are examples of a general way of extending known results from $C^*$-actions to Fell bundles.
The trick is to use the weak equivalence of Fell bundles and the canonical action on the $C^*$-algebra of kernels.
We use this very same idea to treat exactness of cross-sectional $C^*$-algebras, but first we introduce the spatial tensor product of a Fell bundle (over a discrete group) and a $C^*$-algebra.
This construction is a special case of the tensor products of Fell bundles developed in~\cite{Abadie:Tensor}.
We recall the basic facts here for the convenience of the reader;  while doing this we denote by $\otimes$ the spatial (minimal) tensor product of $C^*$-algebras.

Take a Fell bundle $\cB$ and a $C^*$-algebra $C.$
Let $L_t$ be the linking algebra of $B_t$ and define $B_t  \otimes C$ as the closure of the algebraic tensor product $B_t\odot C$ in $L_t  \otimes C.$ We claim that $\cB\otimes C:=\{ B_t\otimes C \}_{t\in G}$ is a Fell bundle with a multiplication and involution such that $(a\otimes x)(b\otimes y)=ab\otimes xy$ and $(a\otimes x)^*=a^*\otimes x^*.$

Since the spatial tensor product behaves nicely with respect to $C^*$-subalgebras and we may view $L_t$ as a $C^*$-subalgebra of the $2\times 2$ matrices with entries in $C^*(\cB),$ we get an isometric inclusion $B_t\otimes C\subseteq C^*(\cB)\otimes C$ as the upper right corner of the inclusion $L_t\otimes C\subseteq \mathbb{M}_2(C^*(\cB))\otimes C.$ Then we can use the $C^*$-algebra structure of $C^*(\cB)\otimes C$ to make $\cB\otimes C$ into a Fell bundle (with the desired properties). Actually, to perform this construction one may replace $C^*(\cB)$ with any $C^*$-algebra in which $\cB$ embeds (fibrewise) isometrically and inherits its Fell bundle structure from as for example $C^*_\red(\cB)$.

There exists a unique unitary $U\colon \ell^2(\cB)\otimes C\to \ell^2(\cB\otimes C)$ such that $U(f\otimes c)(t) = f(t)\otimes c.$ Moreover, conjugation by $U$ transforms the images of the canonical tensor product representation $\kappa\colon C^*_\red(\cB)\otimes C\to \Lb(\ell^2(\cB)\otimes C)$ into $C^*_\red(\cB\otimes C)\subseteq \Lb(\ell^2(\cB\otimes C)).$
Since $\kappa$ is faithful, we obtain a $C^*$-isomorphism $C^*_\red(\cB)\otimes C\cong C^*_\red(\cB\otimes C)$ which is the identity when restricted to $\cB\otimes C$ (recall we can construct the Fell bundle $\cB\otimes C$ ``inside'' $C^*_\red(\cB)\otimes C$).

We now state a proposition and a corollary that hold for general locally compact
(Hausdorff) groups. After them we will continue working with discrete
groups. 

\begin{proposition}\label{prop:kB exact iff Be exact}
 If $\cB=\{B_t\}_{t\in G}$ is a Fell bundle over a locally compact group, then $B_e$ is exact if and only if $\bk(\cB)$ is
 exact. 
\end{proposition}
\begin{proof}
Note that $B_e$ and $\mathcal{K}(L^2(\cB))$ are Morita equivalent, so
one of them is exact if and only if so is the other one. In case
$\bk(\cB)$ is exact, also its ideal $\mathcal{K}(L^2(\cB))$ is exact,
hence $B_e$ is exact. 
\par To prove the converse, suppose first that $G$ is discrete. In the proof of Theorem~\ref{thm: WAP iff AD amenable for wstar Fell bundles} we constructed an inclusion $\mathbb{M}_F(\M)\subseteq \bk_\weaks(\M),$ $F\subseteq G$ being a non empty finite set.
 That inclusion can be used to prove that $\bk(\cB)$ is the direct limit of $\{\mathbb{M}_F(\cB)\}_{F},$ where $F$ runs over the finite subsets of $G.$
 Hence $\bk(\cB)$ is exact if and only if $\mathbb{M}_F(\cB)$ is exact for every finite set $F\subseteq G.$
Assume $B_e$ is exact and take a finite set $F\subseteq G$ and a short exact sequence (s.e.s.) of $C^*$-algebras $I\into A\onto A/I.$
 For every $t\in G$ the linking algebra of $B_t,$ $L_t,$ is an exact $C^*$-algebra because it is Morita equivalent to the ideal $\cspan(B_t^*B_t)$ of $B_e.$
 Thus we get the s.e.s. $L_t \otimes I\into L_t  \otimes A\onto L_t \otimes A/I$ and, by our construction of spatial tensor products, we obtain the following s.e.s. of Fell bundles
 \begin{equation*}
  \cB  \otimes I\into \cB  \otimes A\onto \cB \otimes A/I,
 \end{equation*}
 that by entrywise computation of the arrow produces the s.e.s.
\begin{equation}\label{equ:SEC for tensor matrix}
  \mathbb{M}_F(\cB  \otimes I)\into \mathbb{M}_F(\cB  \otimes A)\onto \mathbb{M}_F(\cB \otimes A/I).
 \end{equation}
Say $F$ has precisely $n$ elements. Then $\mathbb{M}_F(\cB  \otimes I)$ is a $C^*$-subalgebra of 
\begin{equation*}
 \mathbb{M}_n( C^*_\red(\cB\otimes I)) = \mathbb{M}_n( C^*_\red(\cB)\otimes I) = \mathbb{M}_n(\C)\otimes C^*_\red(\cB) \otimes I.
\end{equation*}
Besides, we may also view $\mathbb{M}_F(\cB)  \otimes I$ as a $C^*$-subalgebra of 
\begin{equation*}
 \mathbb{M}_n( C^*_\red(\cB)) \otimes I = \mathbb{M}_n(\C)\otimes C^*_\red(\cB) \otimes I.
\end{equation*}
It then turns out that $\mathbb{M}_F(\cB  \otimes I)$ and $\mathbb{M}_F(\cB)  \otimes I$ get identified inside $\mathbb{M}_n(\C)\otimes C^*_\red(\cB) \otimes I$ and we may write $\mathbb{M}_F(\cB  \otimes I)=\mathbb{M}_F(\cB)  \otimes I.$
This identification works for every $C^*$-algebra $I,$ thus \eqref{equ:SEC for tensor matrix} becomes the s.e.s
 \begin{equation}
  \mathbb{M}_F(\cB)  \otimes I\into \mathbb{M}_F(\cB)  \otimes A\onto \mathbb{M}_F(\cB) \otimes A/I;
 \end{equation}
 and we conclude $\mathbb{M}_F(\cB)$ is exact. 
\par Next we consider the general case. Let $B:=\bk(\cB)$,
$A:=\mathcal{K}(L^2(\cB))$, and $\beta$ the canonical action on
$\bk(\cB)$, but considered as an action of $G^d$, where the latter is
the group $G$ with the discrete topology. Let $\alpha:=\beta|_A$ be
the restriction of $\beta$ to $A$, which is a partial action of $G^d$
on $A$ with enveloping action $\beta$. Note that $A$, the unit fibre
of the semidirect product bundle $\mathcal{B}_\alpha$ over $G^d$,
is exact, because it is Morita equivalent to $B_e$. Then the first
part of the proof implies that 
$\bk(\cB_\alpha)$ is exact. Since both $\beta$ and the canonical
action on $\bk(\cB_\alpha)$ are Morita enveloping actions of $\alpha$,
and Morita enveloping actions are unique up to Morita equivalence, as
shown in \cite{Abadie:Enveloping}*{Proposition~6.3}, we conclude that
$B$ and $\bk(\cB_\alpha)$ are Morita equivalent. Since
$\bk(\cB_{\alpha})$ is exact, it follows that also $B$ is exact.  
\end{proof}

\begin{corollary}\label{cor:exactness and enveloping actions}
 If $B$ is the (Morita) enveloping $C^*$-algebra of a partial action $\alpha$ of a locally compact group $G$ on a $C^*$-algebra $A,$ then $B$ is exact if and only if $A$ is exact.
\end{corollary}
\begin{proof}
Recall that the canonical action of $G$ on $\bk(\cB_\alpha)$ is,
up to Morita equivalence, the only Morita enveloping action of
$\alpha$. Thus the statement follows from Proposition~\ref{prop:kB exact
  iff Be exact}.  
\end{proof}

In the rest of the article we deal only with discrete groups. 

In \cite{kirchberg1994commutants}*{Proposition 7.1} Kirchberg proved that the crossed product of an amenable group acting on an exact $C^*$-algebra is exact. This was proved in \cite{buss2019injectivity}*{Theorem 6.1} for AD-amenable actions on  unital exact $C^*$-algebras, but this result holds for non unital algebras too. Indeed, in the more recent paper \cite{buss2020amenability}*{Proposition~7.5} this is extended to general actions of locally compact groups on non-unital $C^*$-algebras. The preprint version of this paper predates this general result, so we keep our proof below.

\begin{proposition}\label{prop:exact and amenable}
 Let $\beta$ be a $C^*$-action of the group $G$ on $B.$
 If $B$ is exact and $\beta$ has the AP (i.e. it is AD-amenable) then $B\rtimes_{\alpha}G=B\rtimes_{\red,\alpha}G$ is exact.
 \begin{proof}
 Consider a s.e.s of $C^*$-algebras $I\into A\onto A/I.$
 Following Kirchberg we now invoke the universality of full crossed products to obtain a s.e.s 
 \begin{equation}\label{equ:ses of full crossed products}
  (B\otimes I)\rtimes_{\beta\otimes \id_I}G\into (B\otimes A)\rtimes_{\beta\otimes \id_A}G \onto (B\otimes (A/I))\rtimes_{\beta\otimes \id_{A/I}}G.
 \end{equation}
 
 The semidirect product bundle of $\beta\otimes \id_I,$ $\cB_{\beta\otimes \id_I},$ is no other thing than $\cB_\beta\otimes I$ and this bundle has the AP by \cite{Exel:Partial_actions}*{Proposition 25.9}.
 Hence,
 \begin{equation*}
  (B\otimes I)\rtimes_{\beta\otimes \id_I}G = C^*(\cB_{\beta}\otimes I)= C^*_\red(\cB_{\beta}\otimes I)= C^*_\red(\cB_\beta)\otimes I = (B\rtimes_{\red,\beta}G)\otimes I;
 \end{equation*}
 and \eqref{equ:ses of full crossed products} becomes the s.e.s
 \begin{equation*}
  (B\rtimes_{\red,\beta}G)\otimes I \into (B\rtimes_{\red,\beta}G)\otimes A \onto (B\rtimes_{\red,\beta}G)\otimes (A/I);
 \end{equation*}
 giving the desired result.
 \end{proof}
\end{proposition}

\begin{corollary}
 If $\cB$ is a Fell bundle (over $G$) with the AP and $B_e$ is exact, then so is $C^*_{\red}(\cB)=C^*(\cB).$
 In particular, Proposition~\ref{prop:exact and amenable} extends to $C^*$-partial actions.
\end{corollary}
\begin{proof}
 We know, by Proposition \ref{prop:kB exact iff Be exact}, that $\bk(\cB)$ is exact and the canonical action on $\bk(\cB)$ is AD-amenable (Corollaries \ref{cor:Fell bundle AD amenable and action on the kernels} and \ref{cor:AP and amenability for partial actions}).
 By Proposition~\ref{prop:exact and amenable}, $\bk(\cB)\rtimes_{\red} G=\bk(\cB)\rtimes G$ is exact.
 Since this algebra is Morita equivalent to $C^*_{\red}(\cB)=C^*(\cB),$  $C^*_{\red}(\cB)$ is exact.
\end{proof}

In \cite{Exel:Partial_dynamical}*{Definition~21.19} Exel introduces the following notion of conditional expectation for Fell bundles: if $\A$ is a Fell subbundle of $\B$, a conditional expectation from $\B$ to $\A$ is a map $P\colon \B\to \A$ which restricts to bounded surjective idempotent linear maps $P_g\colon B_g\onto A_g\sbe B_g$ such that $P_e\colon B_e\onto A_e$ is an ordinary conditional expectation and $P_g(b)^*=P_{g^{-1}}(b^*)$ and $P_{gh}(ba)=P_g(b)a$ for all $b\in B_g$ and $a\in A_h$, $g,h\in G$.

In  \cite{Exel:Partial_dynamical}*{pp. 188} Exel asks if the AP passes from a Fell bundle $\cB$ to a Fell subbundle $\cA$ provided a conditional expectation $P\colon \cB\to \cA$ exists. To prove this we consider the $W^*$-Fell bundle $\cA''=\{A_t''\}_{t\in G}$ as a $W^*$-Fell subbundle of $\cB''=\{B_t''\}_{t\in G}$ by identifying $A_t''$ with the $\weaks$-closure of $A_t$ in $B_t''.$ The dual map ${P_g}'\colon A_g'\to B_g',$ $\eta\mapsto \eta\circ P_g,$ between the dual spaces $A_g'$ and $B_g'$ of $A_g$ and $B_g,$ respectively, is $\weaks$-continuous and the bidual map $P_g'':=({P_g}')'\colon B_g''\to A_g''$ is the unique $\weaks$-continuous extension of $P_g.$
Notice that the function $P''\colon \cB''\to \cA'',$ mapping $b\in B_g''$  to $P_g''(b),$ is a conditional expectation between $W^*$-Fell bundles. This situation motivates our Corollary~\ref{cor:WAP to subbundles} (which answers Exel's question) and Corollary~\ref{cor:Exel in W form} (which answers the ``$W^*$-form of Exel's question'').

\begin{theorem}\label{theo:WAP and conditional expectation for Wstar bundles}
Let $\M$ be a W*-Fell bundle over $G,$ $\mathcal{N}$ a W*-Fell subbundle of $\M$ and $P\colon \M\to \mathcal{N}$ a (not necessarily $\weaks-$continuous) conditional expectation.
Then there exists a conditional expectation $P_\bk\colon
\bk_\weaks(\M)\to \bk_\weaks(\mathcal{N})$ which is equivariant with
respect to the canonical W*-actions.
Moreover,  $P_\bk$ can be constructed in such a way that its restriction to $\bk(\M)$ is a conditional expectation onto $\bk(\mathcal{N})$.
\end{theorem}
\begin{proof}
By Corollary \ref{cor:inclusion of wstar algebra of kernels} we can think of $\bk_\weaks(\mathcal{N})$ as a W*-subalgebra of $\bk_\weaks(\M).$
The matrix algebras $\mathbb{M}_F(\M)$ (for $F\subseteq G$ finite) of Remark~\ref{rem:matrix algebras of Fell bundles} form an upward directed set of $C^*$-subalgebras of $\bk_\weaks(\M)$ with norm closure equal to $\bk(\M).$
Moreover, $\mathbb{M}_F(\M)$ is hereditary in $\bk(\M)$ and in the proof of Theorem~\ref{thm: WAP iff AD amenable for wstar Fell bundles} we showed that $\mathbb{M}_F(\M)$ is in fact a W*-subalgebra of $\bk_\weaks(\M).$
Hence the family $\{\mathbb{M}_F(\M)\}_{F}$ is an upward directed family of hereditary W*-subalgebras of $\bk_\weaks(\M)$ whose union is $\weaks-$dense in $\bk_\weaks(\M).$
We denote $1_{F}$ the unit of $\mathbb{M}_F(\M)$. Then $1_F$ may or
may not be equal to the unit of $\mathbb{M}_F(\mathcal{N}),$ which we
denote $1'_F.$

Define, for each finite subset $F=\{t_1,\ldots,t_n\}\subseteq G,$ the map $P_F\colon \mathbb{M}_F(\M)\to \mathbb{M}_F(\mathcal{N})\subseteq \bk_\weaks(\mathcal{N})$ as the entrywise application of $P.$
We claim that this map is a conditional expectation.
Indeed, by Tomiyama's theorem it suffices to prove it is contractive.

As done in Remark~\ref{rem:matrix algebras of Fell bundles}, we regard $\mathbb{M}_F(\M)$ as a $C^*$-subalgebra of $\mathbb{M}_n(C^*(\cB)).$ Now, $\mathbb{M}_n(C^*(\cB))$ is a right $C^*(\cB)-$Hilbert module with entrywise multiplication and the inner product $(A,B)\mapsto \operatorname{trace}(A^*B)=\sum_{j,k=1}^{n} (A_{j,k})^*B_{j,k}.$ The restriction of this Hilbert module structure to $\mathbb{M}_F(\M)$ gives a right $B_e-$Hilbert module $X_{F\M}$ with with inner product $\langle A,B\rangle_{F\M} :=\operatorname{trace}(A^*B).$ Matrix multiplication on the left gives a faithful representation $\mathbb{M}_F(\M)\to \Lb(X_{F\M}).$
If $A\in \mathbb{M}_F(\M)$ and $B,C\in X_{F\mathcal{N}},$ then
\begin{equation*}
 \langle P_F(A)B,C\rangle_{F\mathcal{N}} = \operatorname{trace}(P_F(B^*A^*C))=P(\langle AB,C\rangle_{F\M}).
\end{equation*}
This implies $\|P_F(A)\|\leq \|A\|$ and $P_F$ is contractive.

We can extend $P_F$ to $\bk_\weaks(\M)$ by defining $P_F\colon \bk_\weaks(\M)\to \bk_\weaks(\mathcal{N})$ as $P_F(x)=P_F(1_Fx1_F).$
Then $P_F$ is clearly ccp, in fact it is a conditional expectation over $\mathbb{M}_F(\mathcal{N}).$
In this way we get a net of ccp maps $\{P_F\}_F$ from $\bk_\weaks(\M)$ to $\bk_\weaks(\mathcal{N}).$
Let $P_\bk$ be a pointwise $\weaks-$limit of a converging subnet $\{P_{F_j}\}_j.$
Clearly $P_\bk$ is ccp.
Take $x\in \bk_\weaks(\mathcal{N}).$
Since $P_F(x)=1'_Fx1'_F,$ both $\{P_F(x)\}_F$ and $\{P_{F_j}(x)\}_{F_j}$ $\weaks-$converge to $x.$
Thus $P_\bk(x)=x$ and $P_\bk$ is a conditional expectation.

We claim that $P_\bk$ is equivariant with respect to the canonical W*-actions.
Take $t\in G$ and note that $\beta_t^\weaks(1_F)=1_{Ft^{-1}}$ and that, given $x\in \mathbb{M}_{F}(\mathcal{N}),$ it follows that $P_{Ft^{-1}}(\beta_t^\weaks(x))=\beta^\weaks_t(P_F(x)).$
Considering $\weaks-$limits we have
\begin{align*}
 P_\bk(\beta^\weaks_t(x))
    &  = \lim_j P_{F_j}(1_{F_j}\beta^\weaks_t(x)1_{F_j})
        = \lim_j P_{F_j}(\beta^\weaks_t(1_{F_j t}x1_{F_j t}))\\
    & = \lim_j \beta^\weaks_t(P_{F_j t}(1_{F_j t}x1_{F_j t}))
    = \beta^\weaks_t( \lim_j  P_{F_jt}(1_{F_j t}x1_{F_j t}))
\end{align*}
Thus $\{P_{F_j t}(1_{F_j t}x1_{F_j t})\}_j$ actually has a
$\weaks-$limit (for every $x$) and it suffices to show that this limit is $P_\bk(x).$

Every element $v$ of $\bk_\weaks(\mathcal{N})$ is completely determined by the products $uvw,$ for $u,v\in \bk_c(\M).$
Then it suffices to prove that $\lim_j u P_{F_jt}(1_{F_j t}x1_{F_j t})w = uP_\bk(x)w,$ for all $u,w\in \bk_c(\M).$
Fix $u,w\in \bk_c(\M)$ and take a finite set $K\subseteq G$ such that $K\times K$ contains both the supports of $u$ and $w.$
Since the families $\{F_j\}_{j}$ and $\{F_jt\}_{j}$ are cofinal in the finite subsets of $G$ we have
\begin{align*}
 \lim_j u P_{F_jt}(1_{F_j t}x1_{F_j t})w
    & = \lim_j u 1'_K P_{F_jt}(1_{F_j t}x1_{F_j t})1'_Kw\\
    & = \lim_j u P_{F_jt}(1'_K 1_{F_j t}x1_{F_j t}1'_K)w
      = uP_K(1'_Kx1'_K)v\\
    & = uP_\bk(1'_Kx1'_K)v = u1'_KP_\bk(x)1'_Kv=uP_\bk(x)v.
\end{align*}
Finally, the last statement is clear from the computations
  above. In fact if $\mathcal{F}:=\{F\subseteq G:F\textrm{ is
    finite}\}$, then it is easy to see also that the $C^*$-limits of the
  direct systems $\{\mathbb{M}_F(\M)\}_{F\in\mathcal{F}}$ and
  $\{\mathbb{M}_F(\mathcal{N})\}_{F\in \mathcal{F}}$ are $\bk(\M)$ and
  $\bk(\mathcal{N})$ respectively, and that
  $P_\bk|_{\bk(\M)}$ is the limit of the direct system
  $\{\mathbb{M}_F(\M)\}_{F\in
    \mathcal{F}}\stackrel{P_F}{\to}\{\mathbb{M}_F(\mathcal{N})\}_{F\in
    \mathcal{F}}$.
\end{proof}

\begin{corollary}\label{cor:Exel in W form}
Let $\M$ be a W*-Fell bundle over $G,$ $\mathcal{N}$ a W*-Fell subbundle of $\M$ and $P\colon \M\to \mathcal{N}$ a (not necessarily $\weaks-$continuous) conditional expectation.
 If $\M$ has the W*AP then so does $\mathcal{N}.$
\end{corollary}
\begin{proof}
 Recall that $\M$ has the W*AP iff the canonical W*-action on $\bk_\weaks(\M)$ is W*AD-amenable.
 Then everything follows from \cite{Anantharaman-Delaroche:ActionI}*{Proposition 3.8} and the Theorem above.
\end{proof}

 A Fell subbundle $\cA=\{A_t\}_{t\in G}$ of $\cB=\{B_t\}_{t\in G}$ is hereditary if $A_eB_tA_e\subseteq A_t$ for all $t\in G,$ which actually implies $A_rB_sA_t\subseteq A_{rst}$ for all $r,s,t\in G$ (i.e. $\cA\cB\cA\subseteq \cA$).
The condition of $A_e$ being hereditary in $B_e$ does not guarantee $\cA$ is hereditary in $\cB$  \cite{Exel:Partial_dynamical}*{pp 188} but it is enough to show that $\cA$ has the AP if $\cB$ does \cite{Exel:Partial_dynamical}*{Proposition 21.32}.

 If $\cA$ is hereditary in $\cB$ then $\cA''$ is hereditary in $\cB''$ and we may define a conditional expectation $P\colon \cB''\to \cA''$ by $P(b)=q bq,$ where $q\in B_e''$ is the unit of $A_e''.$
This motivates the following result (which is an immediate consequence of the last corollary above).

\begin{corollary}\label{cor:WAP to subbundles}
Let $\cA$ be a Fell subbundle of $\cB.$ If $\cB$ has the AP and the associated inclusion $\A''\into \B''$ admits a conditional expectation, then $\cA$ also has the AP.
\end{corollary}

\section{AD-amenability as a spectral property}\label{sec:commutative-unit-fibre}

This section is dedicated to study amenability  of Fell bundles $\cB=\{B_t\}_{t\in G}$ whose unit fibre are Morita equivalent to a commutative $C^*$-algebra. We start by constructing an action of $G$ on the spectrum $\widehat{B}_e$ (consisting of equivalence classes of irreducible representations).

 Let $X$ be a locally compact Hausdorff space and $Y\equiv {}_{\contz(X)}Y_{B_e}$ an equivalence bimodule.
We let the multiplier algebras $M(\contz(X))=C_b(X)$ and $M(B_e)$ act on the left and right of $Y$ by extending the actions of $\contz(X)$ and $B_e,$ respectively, and identify the centre $ZM(B_e)$ with $C_b(\widehat{B}_e)$via the Dauns-Hofmann Theorem. Recall from \cite{raeburn1998morita}*{Proposition 5.7} that $Y$ induces a homeomorphism $h\colon X\to \widehat{B}_e$ such that $(f\circ h) y = yf$ for all $f\in C_b(\widehat{B}_e)$ and $y\in Y.$

 We view $M(B_e)$ as a $C^*$-subalgebra of $B_e''$ and this gives an inclusion $ZM(B_e)\subseteq Z(B_e'').$ The bidual $Y''$ induces an isomorphism $\pi\colon Z(B_e'')\to \contz(X)''$ such that $\pi(f)y = yf$ for all $y\in Y''$ and $f\in Z(B_e'').$ Hence, $Z(B_e'')$ is the $\weaks-$closure of $\contz(\widehat{B}_e)$ and $\pi(f)=f\circ h$ for all $f\in C_b(\widehat{B}_e).$

 From now on we make no distinction between $X$ and $\widehat{B}_e,$ so we may write $C_b(X)=ZM(B_e)$ and $\contz(X)''=Z(B_e'').$ Let $\theta=\{\theta_t\colon X_{t^{-1}}\to X_t\}_{t\in G}$ be the partial action of $G$ on $X\equiv \widehat{B}_e$ defined by $\cB,$ as constructed in \cite{Abadie-Abadie:Ideals}. The $C^*$-form of $\theta$ is the $C^*$-partial action $\alpha=\{\alpha_t\colon \contz(X_{t^{-1}})\to \contz(X_t)\}_{t\in G},$ which is determined by the identity $\alpha_t(f)=f\circ\theta_{t^{-1}}.$ If we consider each fibre $B_t$ as a $M(B_e)-M(B_e)-$bimodule, then \cite{raeburn1998morita}*{Proposition 5.7} implies $\alpha_t(f)b = bf$ for all $t\in G,$ $b\in B_t$ and $f\in \contz(X_{t^{-1}}).$ Notice that this last condition determines $\alpha_t$ uniquely. We call $\theta$ and $\alpha$ the spectral partial action and the spectral $C^*$-partial action determined by $\cB,$ respectively.

\begin{proposition}\label{prop:amenability bundle is quasi comm fibre}
 Let $\cB=\{B_t\}_{t\in G}$ be a Fell bundle such that $B_e$ is Morita equivalent to a commutative $C^*$-algebra, set $X:=\widehat{B}_e$ and name $\alpha$ the spectral $C^*$-partial action of $G$ on $\contz(X)$ determined by $\cB.$
 Then $\alpha''$ is the central partial action of $\cB''.$
 Consequently, $\cB$ has the AP (i.e. $C^*_\red(\cB)$ is nuclear) if and only if $\alpha$ has the AP (i.e. $\contz(X)\rtimes_{\red,\alpha}G$ is nuclear).
 \begin{proof}
  Keeping the notation we were using before the statement, a straightforward $\weaks-$continuity argument implies $\alpha''_t(f)b = bf$ for all $t\in G,$ $b\in B''_t$ and $f\in \contz(X_{t^{-1}})'';$ this means that $\alpha''$ is the central partial action of $\cB''.$ The rest is a straightforward consequence of Proposition~\ref{prop:AD amenability and nuclearity} and Theorem~\ref{thm:the mega theorem} because $B_e$ is nuclear.
 \end{proof}
\end{proposition}

 The bundle $\cB$ and the one associated to $\alpha,$ $\cB_\alpha=\{\contz(X_t)\delta_t\}_{t\in G},$ may be very different or ``far from being isomorphic''. For example, $B_e$ may not be commutative while $\contz(X)\equiv \contz(X)\delta_e$ always is. Not even the identity $B_e=\contz(X)$ ensures that $\cB$ is isomorphic to $\cB_\alpha,$ this is so because $\cB$ may contain some ``twist''.
 
 According to Exel \cite{Exel:TwistedPartialActions}, a \emph{twisted partial action} of the group $G$ on the $C^*$-algebra $A$ is a pair $(\gamma, \omega)$ where $\gamma=\{\gamma_t\colon A_{t^{-1}}\to A_t\}_{t\in G}$ is a set of $C^*$\nb-isomorphisms between $C^*$-ideals of $A$ and $\omega=\{\omega(s,t)\in M(A_s\cap A_{st})\}_{s,t\in G}$ (the ``twist'') is a set of unitary operators.
 For $(\gamma,\omega)$ to be a twisted partial action the following conditions must be satisfied for all $r,s,t\in G:$
 \begin{enumerate}
  \item $A_e=A$ and $\gamma_e\colon A_e\to A_e$ is the identity.
  \item $\gamma_r(A_{r^{-1}}\cap A_s)= A_r\cap A_{rs}.$
  \item $\gamma_r(\gamma_s(a))=\omega(r,s)\gamma_{rs}(a)\omega(r,s)^*,$ for all $a\in A_{s^{-1}}\cap A_{s^{-1}r^{-1}}.$
  \item $\omega(t,e)=\omega(e,t)=1.$
  \item $\gamma_r(a\omega(s,t))\omega(r,st)=\gamma_r(a)\omega(r,s)\omega(rs,t),$ for all $a\in A_{r^{-1}}\cap A_s\cap A_{st}.$
 \end{enumerate}
As shown by Exel, the conditions above imply that the Banach bundle $\cB_{\gamma,\omega}:=\{ A_t\delta_t \}_{t\in G}$  is a Fell bundle when equipped with the product and involution:
\begin{equation*}
(a\delta_s)\cdot (b\delta_t):=\omega(s,t)\gamma_s(\gamma_s^{-1}(a)b)\delta_{st},\quad (a\delta_s)^*:=\omega(s^{-1},s)^*\gamma_{s^{-1}}(a^*)\delta_{s^{-1}},\end{equation*}
where $s,t\in G,$ $a\in A_s$ and $b\in A_t.$

Condition (4) and the formula for the product of $\cB_{\gamma,\omega}$ imply that one may forget $\omega$ when determining  the $A-A-$bimodule structure on the fibres $A_t\delta_t.$ Hence, the central partial action of $(\cB_{\gamma,\omega})''$ is completely determined by $\gamma,$ implying that $\omega$ plays no r\^{o}le in the AD-amenability of $\cB_{\gamma,\omega}.$ More can be said about this. Notice that the definition of twisted partial action does not require $\gamma=\{\gamma_t\colon A_{t^{-1}}\to A_t\}_{t\in G}$ to be $C^*$-partial action, but we can still construct the family of W*-isomorphism $\gamma''=\{\gamma_t''\colon A_{t^{-1}}''\to A_t''\}_{t\in G}$ and the central partial action of $(\cB_{\gamma,\omega})''$ turns out to be the set of restrictions $\gamma''|_{Z(A'')}:=\{Z(A_{t^{-1}}'')\to Z(A_t''), \ a\mapsto \gamma_t''(a)\}_{t\in G}$.

Assume all the unitaries $\omega(s,t)$ are central (which is the case if $A$ is commutative). Then conditions (1-3) imply that $\gamma$ is a $C^*$-partial action. In this case, the restriction $\gamma''|_{Z(A'')}$  is the central partial action of  both $(\cB_\gamma)''=\cB_{\gamma''}$ and $(\cB_{\gamma,\omega})''.$ Consequently, $\cB_\gamma$ is AD-amenable if and only if $\cB_{\gamma,\omega}$ is.
If $A$ is nuclear, then $C^*_\red(\cB_{\gamma,\omega})$ is nuclear if and only if $A\rtimes_{\gamma,\red}G$ is.
In case $A$ is Morita equivalent to a commutative $C^*$-algebra, the spectral partial actions of $\cB_\gamma$ and $\cB_{\gamma,\omega}$ are (both) the partial action $\hat{\gamma}$ of $G$ on $\hat{A}$ defined by $\gamma,$ see \cite{Abadie:Enveloping}*{Proposition 7.2}.
More precisely, for every $t\in G$ the map $\hat{\gamma}_t\colon \hat{A}_{t^{-1}}\to \hat{A}_t$ is the homeomorphism between the spectrums induced by the $C^*$-isomorphism $\gamma_t\colon A_{t^{-1}}\to A_t.$ 
The central partial action of $\cB_{\gamma,\omega}''$ is then $\hat{\gamma}''$ and $C^*_\red(\cB_{\gamma,\omega})$ is nuclear if and only if $\contz(\hat{A})\rtimes_{\hat{\gamma},\red}G$ is nuclear.
\subsection{Fell bundles with commutative unit fibre}
 Let $\cB=\{B_t\}_{t\in G}$ be a Fell bundle with $B_e=\contz(X)$ commutative ($X\equiv \widehat{B}_e$). As before we denote by $\theta$ and $\alpha$ the spectral partial actions of $G$ on $X$ and $\contz(X)$ defined by $\cB$, respectively.

The main result in \cite{Exel:TwistedPartialActions} states that every \emph{regular} Fell bundle is isomorphic to one associated to a twisted partial action.
The regularity of $\cB$ concerns the structure of the fibres $B_t$ as imprimitivity $\contz(X_t)$-$\contz(X_{t^{-1}})$-bimodules.
Since the $C^*$-algebras $\contz(X_t)$ are commutative, such imprimitivity bimodules are necessarily given as $\contz$-sections of a certain (complex) line bundle $L_t$ over $X_t$.
The $\contz$-section $\contz(L_t)$ of such a line bundle may be viewed as an imprimitivity $\contz(X_t)$-$\contz(X_{t^{-1}})$-bimodule; using the isomorphism $\alpha_t\colon \contz(X_{t^{-1}})\congto \contz(X_t)$ we may also view $\contz(L_t)$ as an imprimitivity $\contz(X_t)$-$\contz(X_{t^{-1}})$-bimodule which is then isomorphic to $B_t$. The regularity of $B_t$ is then equivalent to $L_t$ being topologically trivial as a complex line bundle. This is always the case for Fell bundles associated with twisted partial actions but it might be not the case in general, so that our original Fell bundle $\cB$ is not necessarily isomorphic to  a twisted bundle $\cB_{\alpha,\omega}$, not even as Banach bundles. However, as already explained,  the AP does not see these differences.

 AD-amenability of $\cB,$ which is equivalent to $\contz(X)\rtimes_{\red,\alpha}G$ being nuclear, can also be described using a groupoid description that crossed product. To explain this, let us first
recall that the partial action $\theta$ of $G$ on $X$ yields a locally
compact Hausdorff \'etale transformation groupoid
$\Gamma=X\rtimes_\theta G$ (see \cite{Abadie:On_partial}).  We call $\Gamma$ the \emph{spectral groupoid} of $\cB$. As a set it consists of pairs $(x,t)$ with $t\in G$ and $x\in X_{t^{-1}}$. The source and range maps are $\s(x,t):=x$ and $\rg(x,t):=t\cdot x:=\theta_t(x)$ and multiplication and inversion are given by
\begin{equation*}(x,s)\cdot (y,t)=(y,st),\quad (x,t)^{-1}=(t\cdot x,t^{-1})\quad \mbox{for }x=t\cdot y.\end{equation*}
The topology is the one inherited from the product topology on $X\times G$. The domains $X_t$ give rise to subsets $\Gamma_t:=X_{t}\times\{t^{-1}\}\sbe \Gamma$ that are clopen bisections of $\Gamma$ (although the domains $X_t$ are only assumed to be open in $X$). Hence $\Gamma$ decomposes as a disjoint union $\Gamma=\sqcup_{t\in G}\Gamma_t$ of clopen subsets. In particular the vector space $\contc(\Gamma)$ identifies canonically with the algebraic direct sum $\oplus^{\alg}_{t\in G}\contc(\Gamma_t)$, that is, functions $\zeta\in \contc(\Gamma)$ correspond bijectively to finite sets of functions $\zeta_t\in \contc(\Gamma_t)$, $t\in G$. This identification extends to a canonical isomorphism $C^*_{(\red)}(\Gamma)\cong \contz(X)\rtimes_{(\red),\alpha} G$,  where the parenthesis $(\red)$ indicates the identity hols for full and reduced crossed products.

Next we relate amenability of $\cB$ in terms of amenability of its spectral groupoid. Amenable groupoids are defined and studied mainly in \cite{Renault_AnantharamanDelaroche:Amenable_groupoids}. We shall use the characterisation from \cite{Brown-Ozawa:Approximations}*{Lemma~5.6.14} that says that an \'etale groupoid $\Omega$ is amenable if and only if there is a net $(\zeta_i)\sbe \contc(\Omega)$ with $\|\zeta_i\|_2\leq 1$ for all $i$ and $(\zeta_i^**\zeta_i)(\gamma)\to 1$ uniformly for $\gamma$ in compact subsets of $\Omega$. One of the main results in this direction states that $\Omega$ is amenable if and only if $C^*_{(\red)}(\Omega)$ is nuclear.

We want to perform a finer analysis by identifying $C^*_\red(\cB)$ with some kind of $C^*$-algebra associated to the spectral groupoid $\Gamma.$
To do this we identify each fibre $B_t$ with the sections $\contz(L_t)$ of a line bundle $L_t$ over $X_t,$ as we explained before. The disjoint union $L:=\sqcup_{t\in G} L_t$ can then be viewed as a line bundle over $\Gamma=\sqcup_{t\in G}\Gamma_t$. Moreover, with the Fell bundle structure inherited from $\cB$, $L$ is indeed a Fell line bundle over $\Gamma$; such a Fell bundle is also usually viewed as a \emph{twist} over $\Gamma$. By construction we get an obvious identification $\contc(\Gamma,L)\cong \contc(\cB)$ that extends to an isomorphism $C^*_{(\red)}(\Gamma,L)\cong C^*_{(\red)}(\cB)$. In other words, we have described every Fell bundle over a discrete group with commutative unit fibre in terms of a twisted groupoid. This result can be deduced from the constructions and results in \cite{BussExel:Regular.Fell.Bundle} that describe Fell bundles over inverse semigroups with commutative fibres over idempotents (also call semi-abelian Fell bundles in \cite{BussExel:Regular.Fell.Bundle}) in a similar way via twisted groupoids. The Fell bundles in \cite{BussExel:Regular.Fell.Bundle} are assumed to be saturated, but the same constructions can also be done in general for non-saturated ones; alternatively, one can view a non-saturated Fell bundle over $G$ as a saturated Fell bundle over the inverse semigroup $S(G)$ constructed by Exel in \cite{Exel:Partial_actions}, see \cite{BussExel:InverseSemigroupExpansions}.

Using the description of $\cB$ in terms of a twisted groupoid $(\Gamma,L)$ and that AD-amenability is equivalent to nuclearity of the corresponding $C^*$-algebras, we can also interpret the above result as the statement that $C^*_{(\red)}(\Gamma,L)$ is nuclear if and only if $C^*_{(\red)}(\Gamma)$ is nuclear. In other words, nuclearity of a twisted groupoid $C^*$-algebra is independent of the twist. Indeed, in this form this result is already known, see \cite{Takeishi:Nuclearity}.

 Notice we do not really need $B_e$ to be commutative to construct the spectral groupoid of $\cB,$ it suffices to assume $B_e$ is Morita equivalent to a commutative $C^*$-algebra. Thus the preceeding discussion and Proposition \ref{prop:amenability bundle is quasi comm fibre} produce the following.

\begin{corollary}
Let $\cB=\{B_t\}_{t\in G}$ be a Fell bundle with $B_e$ Morita equivalent to a commutative $C^*$-algebra, set $X\equiv \widehat{B}_e$ and let  ($\alpha$) $\theta$ be the spectral ($C^*$-)partial action of $G$ on ($C_0(X)$) $X.$
If $\Gamma:=X\rtimes_\theta G$ is the spectral groupoid of $\cB,$ then the following are equivalent:
\begin{enumerate}[(i)]
\item $\cB$ has the AP, that is to say $C^*_{(\red)}(\cB)$ is nuclear.
\item $\cB_\alpha$ has the AP, meaning that $C^*_{(\red)}(\cB_\alpha)=C^*_{(\red)}(\Gamma)=\contz(X)\rtimes_{(\red),\alpha} G $ is nuclear.
\item $\Gamma$ is amenable or, equivalently, $C^*_{(\red)}(\Gamma)$ is nuclear.
\item for every twisted partial action of the form $(\alpha,\omega),$ the corresponding Fell bundle $\cB_{\alpha,\omega}$ has the AP.
\end{enumerate}
\end{corollary}

 We close this article with an example relating nuclearity of graph $C^*$-algebras, AP and exactness of free groups.

\begin{example}
In \cite{Exel:Partial_dynamical}*{Proposition~37.9}  Exel provides a partial crossed product description for the $C^*$-algebra of every directed graph $E=(\s,\rg\colon E^1\to E^0)$ with no sinks (i.e. $s^{-1}(v)\not=\emptyset$ for all $v\in E^0$). In other words, we have an isomorphism
\begin{equation*}C^*(E)\cong \contz(X)\rtimes_\alpha G\end{equation*}
for a certain partial action $\alpha$ of the free group $G=\F_n$ on $n=|E^1|$ generators (this can be infinite), and $X$ is a certain (totally disconnected) locally compact Hausdorff space. The exact description of this space and the partial action is slightly complicated in general but it simplifies under certain regularity conditions on $E$. For instance, if every vertex $v\in E^0$ is \emph{regular} in the sense that $\rg^{-1}(v)$ is non-empty and finite, $X$ is just the infinite path space $E^\infty$ of $E$.

Regardless of how $X$ and the partial action $\alpha$ above are defined, using that graph $C^*$-algebras are always nuclear (a well-known fact, see \cite{Kumjian-Pask:C-algebras_directed_graphs}*{Proposition~2.6}), it follows from our previous theorem that $\alpha$ has the AP. Indeed, Exel gives a more direct proof of this fact in \cite{Exel:Partial_dynamical}*{Theorem~37.10}.

We shall give more details about the partial action $\alpha$ and its amenability in what follows in the case of the graph $E$ that describes the Cuntz algebra $\mathcal{O}_n$, that is, the graph with one vertex and $n$ loops with $2\leq n<\infty$. This is a special and representative case. This is a finite graph that has no sinks or sources. In this case, $X\cong \{1,\ldots, n\}^\infty$ is Cantor space
and $G=\F_n$ is the free group on $n$ generators that we also view as the free group generated by $E^1$. The partial action $\alpha$ is defined as follows: the domains $D_g$ for $g\in \F_n$ are defined in terms of the cylinders $X_a=\{a\mu:\mu\in X\}$ if $g\in \F_n$ can be written in reduced form as $g=ab^{-1}$ for $a,b\in E^*$, the set of finite paths viewed as elements of $\F_n$.
In this case $D_{g^{-1}}=\cont(X_b)$ and $D_g=\cont(X_a)$ and $\alpha_g\colon D_{g^{-1}}\congto D_g$ is given $\alpha_g(f)=f\circ\theta_g^{-1}$, where $\theta_g\colon X_b\congto X_a$ is the canonical homeomorphism sending $b\mu\mapsto a\mu$. If $g$ is not of the form $ab^{-1}$, then $D_g$ is defined to be the zero ideal (and $\alpha_g$ is the zero map).

The AP for $\alpha$ means the existence of a net of finitely supported functions $\xi_i\colon G\to \cont(X)$ that is uniformly bounded for the $\ell^2$-norm and satisfying
\begin{equation}\label{eq:AP-partial}
\braket{\xi_i}{a\tilde\alpha_g(\xi_i)}_2:=\sum_{h\in G}\xi_i(h)^*\alpha_g(\alpha_g^{-1}(a\xi_i(g^{-1}h)))\to a
\end{equation}
for all $g\in G$ and $a\in D_g$. Notice that all the ideals $D_g$ are unital here. If $1_g$ denotes its unit (so that $D_g=A\cdot 1_g$), then~\eqref{eq:AP-partial} is equivalent to
\begin{equation*}\sum_{h\in \F_n}\xi_i(h)\alpha_g(1_{g^{-1}}\xi_i(g^{-1}h))\to 1_g\end{equation*}
for all $g\in G$. One explicit sequence $\xi_i\colon G\to C(X)$ that gives the AP for this partial action can be defined by
$\xi_i(g)=\frac{1}{\sqrt{i}}1_g$ if $g\in \F_n^+$ (the positive cone of $\F_n$) with length $|g|\leq i$ and $\xi_i(g)=0$ otherwise.
Recall that $1_g$ denotes the characteristic function on the cylinder set $X_g=\{g\mu:\mu\in X=E^\infty\}$ which makes sense because $g$ is positive.

The fact that all domain ideals $D_g$ are unital also implies that $\alpha$ has an enveloping global action and we know from Corollary~\ref{cor:ADamenability and enveloping} that this global action also has the AP. Indeed, a concrete description of the enveloping action for the partial action of $\F_n$ on $X$ is as follows: instead of considering only positive words, we also consider their inverses, that is, we consider the generators of $\F_n$ and their inverses, and then look at all infinite reduced words on this new alphabet. This yields a new space, denoted $\bar X$ that naturally contains $X$ as a clopen subspace. 
Now notice that $\F_n$ naturally acts (globally) on $\bar X$ by (left) concatenation and the partial action on $X$ is just the restriction of this global action. Moreover, the global action of $\F_n$ on $\bar X$ is known to be amenable: this action can be viewed as the action on a certain boundary of $\F_n$, and this is an amenable action, see \cite{Anantharaman-Delaroche:Amenability}*{Examples~2.7(4)} and \cite{Brown-Ozawa:Approximations}*{Proposition~5.1.8}. Indeed, this is the standard way to see that $\F_n$ is an exact group.
\end{example}

\vskip 1pc

\end{document}